\documentclass[10pt]{article}
\usepackage[utf8]{inputenc}
\usepackage{amsmath, amsthm}
\usepackage[margin=1in]{geometry}
\usepackage{tikz}
\usetikzlibrary{arrows.meta}

\usepackage{hyperref}

\title{Singularity Formation in the Incompressible Porous Medium Equation without Boundary Mass}
\author{Kevin H. Dembski \footnote{Department of Mathematics, Duke University. E-mail: kevin.dembski@duke.edu.}}
\date{\today}

\usepackage{amssymb}
\newcommand{\R}{\mathbb{R}}
\newcommand{\C}{\mathbb{C}}
\newcommand{\p}{\partial}
\newcommand{\Ca}{\mathring{C}^{\alpha}}
\renewcommand{\d}{\text{\ d}}

\renewcommand{\H}{\mathcal{H}}
\newcommand{\tH}{\widetilde{\mathcal{H}}}

\newcommand{\inner}[2]{\left\langle #1, #2\right\rangle}
\def\L{\mathcal{L}}
\def\Lb{\overline{\mathcal{L}}}
\def\Lk{\mathcal{L}_K}
\def\P{\mathbb{P}}

\def\tG{\widetilde{G}}
\def\tM{\widetilde{M}}
\def\l{\ell}
\def\N{\mathcal{N}}

\newtheorem{theorem}{Theorem}[section]
\newtheorem{lemma}[theorem]{Lemma}
\newtheorem{proposition}[theorem]{Proposition}

\newtheorem{cor}[theorem]{Corollary}
\newtheorem{remark}[theorem]{Remark}
\newtheorem{definition}[theorem]{Definition}

\usepackage{cleveref}
\numberwithin{equation}{section}

\usepackage[style=numeric, url=false, doi=false, isbn=false, maxbibnames=10]{biblatex}
\usepackage{mathtools}

\begin{document}

\maketitle

\begin{abstract}
We prove finite-time singularity formation for Lipschitz continuous solutions of the inviscid porous medium equation which vanish on the boundary of the domain. As the density vanishes on the boundary of the domain, the full regularizing effect of transport is present and must be overcome. The solutions are smooth away from the origin and the density can be made compactly supported.
\end{abstract}
\tableofcontents
\section{Introduction}

\subsection{The 2D Inviscid Porous Medium Equation}
The inviscid porous medium (IPM) equation models the flow of an incompressible fluid through a porous medium such as sand
\begin{align}\label{eq:IPM}\begin{dcases}
u + \nabla p = (-\rho, 0) \\
\p_t \rho + u \cdot \nabla \rho = 0 \\
\nabla \cdot u = 0.
\end{dcases}
\end{align}
Here, \(\Omega \subset \R^2\) is the region occupied by the fluid, \(u: \R_{\geq 0} \times \Omega \to \R^2\) denotes the velocity field of the fluid which evolves according to Darcy's law, and \(\rho: \R_{\geq 0} \times \Omega \to \R\) denotes the density which is transported by the velocity. On the boundary of \(\Omega\), we impose the usual no-penetration boundary condition \(u \cdot n = 0\) where \(n\) is the normal to \(\p\Omega\). Note that we adopt the slightly unusual convention in which gravity acts horizontally in the negative \(x\)-direction for sake of analogy with the centrifugal force in the 3D axisymmetric Euler equations.

The study of singularity formation in incompressible fluids is a major problem in partial differential equations. One of the main difficulties is the seemingly perfect balance between the effects of advection and stretching in such systems. In the case of the IPM equations, the gradient of the density \(\nabla \rho\) is advected  and stretched by an incompressible velocity field \(u\). Thus, while the stretching effect works to cause rapid growth in \(\nabla \rho\) towards a potential blow-up, the advection then rapidly ejects particles from regions of high growth limiting this effect. This phenomenon is known as regularization by transport and has been a fundamental obstruction for proving blow-up results in incompressible fluids. Many of the examples of blow-up in incompressible fluids occur in settings designed to limit this regularizing effect. To the author's knowledge, there is no known scenario to weaken the effect of advection given smooth initial data in a domain without boundary, and thus understanding this regularization effect is key to proving any such result. In this work, we establish finite-time blow-up for Lipschitz solutions to the IPM equation in a setting in where the regularizing effect of transport is present.

\subsection{Main Result}
In order to state our main result, we first recall the scale invariant H\"older spaces, first introduced in \cite{EJ-S}. These spaces have the same scaling as \(L^\infty\), however the singular integral operators which arise in the Biot--Savart law are bounded on \(\Ca(\Omega)\) when \(\Omega\) is a wedge domain strictly smaller than the half plane and \(\rho\) is even in \(x_2\). This allows our solutions to be placed in a suitable local well-posedness class.
\begin{definition}\label{def:Ca}
The scale-invariant H\"older space \(\mathring{C}^{\alpha}(\Omega)\) is defined by the norm \(\|\cdot\|_{\Ca}\),
\begin{align*}
\|f\|_{\Ca} &= \|f\|_{L^{\infty}} + \sup_{x \neq x'}\frac{||x|^{\alpha}f(x) - |x'|^{\alpha}f(x')|}{|x-x'|^{\alpha}}.
\end{align*}
The higher order spaces \(\mathring{C}^{k, \alpha}\) are then defined by
\[
\|f\|_{\mathring{C}^{k, \alpha}} = \|f\|_{\mathring{C}^{k-1, \alpha}} + \sup_{x \neq x'}\frac{||x|^{k+\alpha}\nabla^k f(x) - |x'|^{k+\alpha}\nabla^k f(x')|}{|x-x'|^{\alpha}}.
\]
\end{definition}
\noindent Note that for scale-invariant (0-homogeneous) functions, the \(\Ca\) norm is equivalent to the usual \(C^{\alpha}\) norm considered in the angular variable. We also remark that \(\nabla f \in \mathring{C}^{k, \alpha}\) clearly implies \(f\) is Lipschitz continuous. We now state our main theorem. 
\begin{theorem}\label{theorem:main}
For any \(k \geq 0\) and \(0 < \alpha < 1\), there exist \(\nabla u_0, \nabla \rho_0 \in \mathring{C}^{k,\alpha}(\Omega)\) with compact support such that the unique local in time solution to \eqref{eq:IPM} satisfies
\[
\limsup_{t \to 1}\int_{0}^{t}\|\nabla \rho(s)\|_{L^{\infty}}\d s = +\infty.
\]
Here, \(\Omega = \{(r, \theta) : -\beta \pi < \theta < \beta \pi\}\) for some \(\beta < \pi/2\) where \((r, \theta)\) denote the standard polar coordinates on \(\R^2\). Moreover, the initial density \(\rho_0\) can be chosen to be compactly supported in the angular variable.
\label{main}
\end{theorem}
\begin{remark}
    In \S 3.4 we prove that for any \(\beta' < \pi/2\) we can take \(\beta' < \beta < \pi/2\). Thus, we obtain blow-up on domains arbitrarily close to the half-plane. 
\end{remark}
The primary novelty of the result is achieving blow-up for data which vanishes on the boundary of the domain. In \cite{EJ-B}, Elgindi and Jeong construct solutions of the Boussinesq equations which blow-up in finite time in the same setting as the current work. The key difference however, is that their solutions do not vanish on the boundary of the domain. In this way, they are able to use the boundary to diminish the regularizing effect of transport. Indeed, if the final assumption that \(\rho_0\) vanishes on the boundary is removed from Theorem \ref{main}, the result can be obtained via an ODE type argument as in \cite{EJ-B}. Our goal in this work is to provide a scenario in which the blow-up is not driven by the boundary, and the regularizing effect of transport is present.

We emphasize that the singularity is not merely an artifact of the geometry of the domain but is tied to the structure of the IPM equation itself. Indeed in \cite{EJ-S}, the Yudovich theory for global well-posedness of 2D Euler is carried over to domains with acute corners. In this sense, the singularity is not merely generated by the singularity of the domain but instead from the structure of the equation itself.

\subsection{Previous Results}
Among the first rigorous mathematical studies of the IPM equation was in \cite{CGR_IPM} where local well-posedness was shown and the question of global well-posedness was explored. While the global well-posedness of smooth solutions to the IPM equation remains open, there have been numerous works dedicated to addressing this issue. In \cite{KY-IPM}, Kiselev and Yao construct solutions which exhibit infinite in time growth of the density in any \(H^s\) space. In particular, they obtain gradient growth faster than \(t^{1/4}\). Very recently, in \cite{CM-IPM}, Cordoba and Martinez--Zoroa constructed smooth solutions of the forced equation which blow-up in finite time. 

The Muskat equation models the evolution of the interface between two fluids of different densities evolving according to the IPM equation. In the remarkable series of papers \cite{Z1, Z2}, the Muskat problem is studied on the half-plane and it is proven that singularities can form, even from smooth initial data.

\subsection{Regularization by Transport}

We now comment on the role of advection in our setting. As mentioned, a key difficulty in proving blow-up or growth in incompressible fluid models is overcoming the regularizing effect of advection. To suppress this regularizing effect, two key strategies have been employed: working in low regularity, and adding a boundary. 

Considering solutions of low regularity allows the data to be more highly concentrated near the origin, and then the transport is too weak to deplete such highly concentrated data. We refer the reader to \cite{EC1a, EJ-adv, CH-DG, CH_C1a, CMZZ_Euler} for a non-exhaustive list of works which employ this strategy. 

When considering a model posed on a domain with boundary, since the velocity field is tangent to the boundary, any mass on the boundary will remain attached to the boundary and cannot be ejected. Thus, in cases where the density does not vanish on the boundary of the domain, the regularizing effect of the transport is absent. This strategy was adopted in the work \cite{EJ-B} to construct scale-invariant solutions to the Boussinesq equation which blow-up in finite time. Kiselev and \v Sver\'ak \cite{KS} use a boundary to obtain the sharp double exponential growth of the gradient of the vorticity in the 2D Euler equations. In \cite{EJ-S}, the authors consider scale-invariant solutions of the 2D Euler equations. They prove that in the presence of the boundary, if there is initially non-zero vorticity on the boundary, then the vorticity gradient grows exponentially. They also prove that this is the optimal growth rate possible for scale-invariant solutions. It is a consequence of \cite{EMS} however, that if one does not place mass on the boundary then exponential growth is not possible for scale-invariant solutions. It remains a difficult open problem what growth rate can be achieved for the vorticity gradient in the 2D Euler equations without the use of a boundary in general. Finally, we mention that Chen and Hou \cite{CH1, CH2} have put forth an interesting, computer assisted proof of singularity formation for smooth solutions of the 2D Boussinesq and 3D Euler equations on domains with boundary.

If one wishes to construct smooth solutions to incompressible fluid models which blow up and exist in the whole space, it is necessary to deal with the effect of transport. In the recent works \cite{EP1, EP2}, Elgindi and Pasqualotto study the regularizing role of transport in the Boussinesq and 3D Euler equations. In their work, the solution is smooth in the angular variable and thus the transport cannot be treated fully perturbatively. In the present work, we establish blow-up for solutions of the IPM equation which vanish on the boundary of the domain. Since there is no mass present on the boundary of the domain, the regularizing effect of transport is present and must be overcome. Below, we include the difference in the setting of the boundary driven blow-up considered in \cite{EJ-B} and the setting studied here.
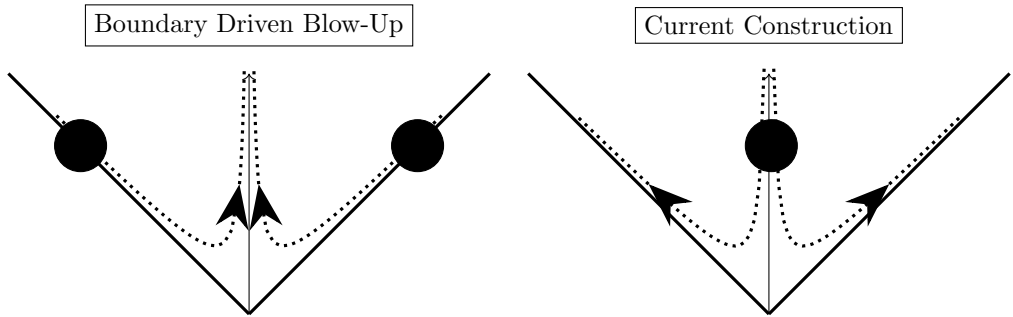
\begin{center}
    \begin{tikzpicture}[domain=-10:10, scale=0.32]
		\draw[->] (0, 0) -- (0,10) ;
            \draw[-, very thick] (0, 0) -- (10, 10) ;
            \draw[-, very thick] (0,0) -- (-10, 10);

		\draw[color=black, domain=8:0.4, ->, arrows = {-Stealth[scale=2]}, dotted, very thick]   plot (\x,{2/\x+\x});
        \draw[color=black, domain=0.4:0.2, dotted, very thick]   plot (\x,{2/\x+\x});

        \draw[color=black, domain=-0.4:-0.2, dotted, very thick]   plot (\x,{-2/\x-\x});
        \draw[color=black, domain=-8:-0.4, ->, arrows = {-Stealth[scale=2]}, dotted, very thick]   plot (\x,{-2/\x-\x});
		
		\node[draw, circle, scale=1, fill=black] at (7,7) {$+$};
            \node[draw, circle, scale=1, fill=black] at (-7,7) {$+$};

            \node[draw, scale=1] at (0, 12) {Boundary Driven Blow-Up};

	\end{tikzpicture}
        \quad
	\begin{tikzpicture}[domain=-10:10, scale=0.32]
		\draw[->] (0, 0) -- (0,10) ;
            \draw[-, very thick] (0, 0) -- (10, 10) ;
            \draw[-, very thick] (0,0) -- (-10, 10);

		\draw[color=black, domain=0.2:5, ->, arrows = {-Stealth[scale=2]}, dotted, very thick]   plot (\x,{2/\x+\x});
        \draw[color=black, domain=5:8, dotted, very thick]   plot (\x,{2/\x+\x});

        \draw[color=black, domain=-0.2:-5, ->, arrows = {-Stealth[scale=2]}, dotted, very thick]   plot (\x,{-2/\x-\x});
        \draw[color=black, domain=-5:-8, dotted, very thick]   plot (\x,{-2/\x-\x});

		\node[draw, circle, scale=1, fill=black] at (0.1,7) {$+$};

            \node[draw, scale=1] at (0, 12) {Current Construction};

	\end{tikzpicture}
\end{center}

    In both cases, a hyperbolic flow is set up near the corner. In \cite{EJ-B}, mass is made to flow along the boundary towards the corner. In the current construction, we have mass flowing down along the axis of symmetry towards the corner. As no mass is attached to the boundary, it is now possible for mass to be ejected away from the symmetry axis by the transport which would serve to limit potential growth. We prove that a singularity can still form in spite of this regularizing effect.

\subsection{Scale Invariance and Corner Domains}
The IPM equation possesses the following scaling symmetry: if \(\rho(t, x)\) is a solution of \eqref{eq:IPM}, then \(\rho_{\lambda}(t, x) := \lambda^{-1}\rho(t, \lambda x)\) is also a solution. This leads us to consider solutions which are invariant under this scaling. Such solutions take the form \(\rho(t, x) = |x|P(t, x/|x|)\). One advantage of considering such solutions is that they reduce the full two-dimensional system to the following one-dimensional system which is more amenable to analysis,
\begin{align*}
\begin{cases}
    \p_t P(t, \theta) + 2G(t, \theta)P'(t, \theta) = G'(t, \theta)P(t, \theta) \\
    G''(t, \theta) + 4G(t, \theta) = P(t, \theta)\sin(\theta) + P'(t, \theta)\cos(\theta).
\end{cases}
\end{align*}
In \cite{EIS}, the authors study the above system in the formal limit where the angle is small. That is, they study the above system with \(\sin\theta = 0,\ \cos\theta = 1\). By using a trajectory based approach, they are similarly able to produce solutions which vanish on the boundary of the domain and become singular in finite time. 

There are however, two primary drawbacks of considering scale-invariant solutions. First, they can be at most Lipschitz at the origin regardless of the regularity in the angular direction. Second, in order to place the solutions in a well-posedness class, they must only be defined on an acute subset of the plane. With the presence of additional symmetries of the IPM system, namely the even symmetry in \(x_2\) (recall our convention that gravity acts horizontally), there is well-posedness in any wedge domain strictly smaller than the half-plane. Unfortunately, due to the anisotropy of the IPM equation, it does not seem possible to study scale-invariant solutions without introducing a corner domain. As noted however, the Yudovich theory for 2D Euler can be carried over to such domains and consequently the blow-up is generated by the gradient stretching in the IPM equation and is not merely an artifact of the singular domain.

Another seeming drawback of scale invariant solutions is that they naturally have infinite energy. Indeed, such solutions must grow linearly at infinity. Fortunately, these solutions are stable under truncations at infinity and thus scale invariant solutions which blow up can be made compactly supported and therefore finite energy. This truncation procedure was developed by Elgindi and Jeong in \cite{EJ-S, EJ-B} and we apply it again in \S 8 to obtain compactly supported, finite energy solutions. 

\subsection{Discussion of The Proof}
We begin by considering solutions which are 1-homogeneous in space, 
\[
    \rho(t, r, \theta) = rP(t, \theta), \quad \Psi(t, r, \theta) = r^2G(\theta)
\]
where \(\Psi\) denotes the stream function, \(u = \nabla^{\perp} \Psi\). This reduces the two dimensional system to the following one dimensional system,
\begin{align}\label{eq:IPM1D}
\begin{cases}
\p_t P(t, \theta) + 2G(t, \theta)P'(t, \theta) = G'(t, \theta)P(t, \theta) \\
G''(t, \theta) + 4G(t, \theta) = P(t, \theta)\sin\theta + P'(t, \theta)\cos\theta \\
G(\cdot, -L) = G(\cdot, L) = 0
\end{cases}
\end{align}
posed on \([-L, L]\) for some \(L < \pi/2\). From this point forward, we consider angular functions \(P, G\) which are even and odd respectively, which is propagated by the equations. Our approach centres around constructing a blow-up profile for \eqref{eq:IPM1D}. Thus, we consider solutions of the form
 \[
 P(\theta, t) = (1-t)^{-1}P_*(\theta), \quad G(\theta, t) = (1-t)^{-1}G_*(\theta).
\]
In order to obtain solutions which vanish on the boundary, it is necessary to have \(P_*(L) = 0\).  The profile \(P_*\) then satisfies the following singular, nonlocal system of ODEs
\begin{align}\label{eq:profile_ode_P}
\begin{cases}
    P_*(\theta) + 2G_*(\theta)P_*'(\theta) = G_*'(\theta)P_*(\theta) \\
    G_*''(\theta) + 4G_*(\theta) = P_*(\theta)\sin\theta + P_*'(\theta)\cos\theta \\
    G_*(0) = G_*(L) = 0.
\end{cases}
\end{align}
A key observation is that by evaluating \eqref{eq:profile_ode_P} at \(\theta = 0\) it must be that \(G_*'(0) = 1\). Therefore, the system \eqref{eq:profile_ode_P} can be viewed instead as a more local initial value problem. A local solution near \(\theta = 0\) can then be constructed via Taylor expansion, coupled with a fixed point scheme. The next key observation is a monotonicity property of \(P_*\). This monotonicity then allows us to extend the locally constructed profile. The monotonicity also requires carefully fixing the initial condition \(P_*(0)\) which corresponds to choosing the strength of a background flow. Finally, to ensure the boundary condition \(G_*(L) = P_*(L) = 0\) is satisfied, we use a shooting method argument.

Once the profile is constructed, we wish to truncate it near the boundary \(\theta = L\). To do so requires examining the stability of the profile in a certain weighted, Sobolev type space, \(\tH^4\) (see Definition \ref{def:tH4}). After passing to the logarithmic time \(s = -\log(1-t)\) we study the linear stability of the profile. We show the linearized operator \(\L\) about the profile can be decomposed as \(\L = \mathcal{L}_K + \Lb\) where \(\mathcal{L}_K\) is a finite-rank, smoothing operator and \(\Lb\) is coercive on \(\tH^4\). The key observation in constructing such a decomposition is that there is only one nonlocal quantity in the equations and thus they can be localized by a finite-rank perturbation. It can then be seen that the localized system is coercive. From standard semigroup theory it then follows that \(\mathcal{L}\) has a finite-dimensional unstable subspace. Finally, a stable manifold theorem argument allows us to perform a cutoff near the boundary \(\theta = L\) while avoiding any instabilities of the profile.

After producing the desired singular solution of the associated 1D system \eqref{eq:IPM1D}, we pass to blow-up in the full IPM system \eqref{eq:IPM}. Finally, a truncation argument as in \cite{EJ-B} allows us to construct solutions which are compactly supported at spatial infinity and hence have finite energy. 
\subsection{Organization}
In \S 2 we prove local well-posedness of the IPM equation in \(\Ca\) on wedge domains strictly smaller than the half plane. In \(\S 3\) we construct a blow-up profile for the 1D system obtained by considering scale-invariant solutions. In \(\S 4\), we linearize about the profile and define the weighted space \(\tH^4\) in which we will show finite co-dimension stability of the linearized operator. We also prove some useful Hardy-type inequalities. In \(\S 5\) and \(\S 6\) we perform the coercive estimates in the low and high norm respectively. In \(\S 7\), we use a stable manifold theorem argument to construct a smooth solution to the 1D system which converges to the profile at the blow-up time. In \(\S 8\) we then pass to blow-up for finite energy solutions of the full 2D IPM system.

\subsection{Notation and Definitions}
Recall the space \(\Ca\) from Definition \ref{def:Ca}.
The following product estimate in \(\Ca\) is then immediate
\begin{align}
\|fg\|_{\Ca} \leq \|f\|_{L^{\infty}}\|g\|_{\Ca} + \|f\|_{\Ca}\|g\|_{L^{\infty}}\label{Ca_prod}.
\end{align}
We also have the following product estimate (see Lemma 2.3 in \cite{EJ-B})
\begin{lemma}\label{lemma:Ca_product}
There exists a constant \(C > 0\) such that for all \(f \in C^\alpha\) with \(f(0) = 0\)  and \(g \in \Ca\),
\[
\|fg\|_{C^\alpha} \leq C\|f\|_{C^\alpha}\|g\|_{\Ca}.
\]
\end{lemma}
\noindent We denote the \(k^{th}\) Taylor polynomial of a function \(f\) by \(\P_k(f)\),
\begin{equation}\label{def:taylor}
\mathbb{P}_k(f)(\theta) = \sum_{j=0}^{k}\frac{f^{(j)}(0)}{j!}\theta^j.
\end{equation}
\noindent We take the convention that \(\nabla^{\perp} = (-\p_y, \p_x)\).
When working on the one-dimensional system in the angular variable, we write \('\) for angular derivatives \(\p_\theta\). The letter \(C\) is reserved for an inconsequential constant factor which may change from line to line.

\section{Local Well-Posedness in Scale-Invariant Spaces}
We consider the equation \eqref{eq:IPM} posed on domains
\[
\Omega_{\beta}' = \{(r, \theta) : -\beta \pi < \theta < \beta \pi\}
\]
for \(\beta < 1/2\) where \((r, \theta)\) denote the standard polar coordinates on \(\R^2\). By considering solutions \(\rho\) which are even with respect to the \(x_1\)-axis, we may instead consider domains
\[
\Omega_{\beta} = \{(r, \theta) : 0 < \theta < \pi \beta\}.
\]
It is easily seen that such symmetry is preserved by the equation. The condition \(\beta < 1/2\) is crucial to obtaining well-posedness. In the case of the half-plane in which \(\beta = 1/2\), the equation degenerates and it is no longer possible to uniquely recover the velocity from the density. This phenomenon was first observed in \cite{E_Remarks} and is expanded upon in \cite{EJ-B, EJ-S}. We first prove the following theorem regarding the local well-posedness of \eqref{eq:IPM} in \(\Ca\).
\begin{proposition}\label{prop:lwp_si}
    Let \(\beta < 1/2\) and \(0 < \alpha < 1\). Let \(\rho_0: \Omega_\beta \to \R\) be such that \(\nabla^{\perp}\rho_0 \in \Ca(\overline{\Omega_\beta})\). Assume that \(\rho\) is even with respect to \(x_2\). Then, there exists \(T > 0\) such that there is a unique solution \(\rho \in C([0, T], \Ca(\overline{\Omega_\beta}))\) to the system \eqref{eq:IPM}.  Moreover, the solution can be extended beyond the interval \([0, T]\) if and only if
    \[
    \int_{0}^{T}\|\nabla u(s)\|_{L^\infty}\d s < \infty.
    \]
\end{proposition}
\begin{proof}
The proof follows that of \cite{EJ-B}, but we include a sketch here for sake of completeness. We begin by providing a priori estimates. Let \(\rho\) be a smooth solution to \eqref{eq:IPM} on \([0, T]\). From Lemma 3.5 of \cite{EJ-B}, we recall that \(D^2\Delta^{-1}\) is bounded on \(\Ca(\overline{\Omega_\beta})\) and thus
\[
\|\nabla u\|_{\Ca} \leq C\|\nabla^{\perp}\rho\|_{\Ca}.
\]
In particular, the flow map \(\Phi_t\) which solves
\begin{align*}
 \frac{d\Phi_t}{dt} = u(\Phi_t(x), x), \quad \Phi_0(x) = x
\end{align*}
is well-defined as \(u\) is Lipschitz. Now, to estimate \(\nabla^{\perp}\rho\) in \(\Ca\), consider the evolution equation of \(\nabla^{\perp}\rho\)
\begin{align}
\p_t\nabla^{\perp}\rho + (u \cdot \nabla)\nabla^{\perp}\rho = \nabla u \nabla^{\perp}\rho \label{grad_rho}.
\end{align}
Writing \eqref{grad_rho} along the flow \(\Phi\) gives,
\begin{align*}
\p_t\nabla^{\perp} \rho \circ \Phi_t = \nabla u \circ \Phi_t \nabla^{\perp}\rho \circ \Phi_t
\end{align*}
and we immediately obtain the \(L^{\infty}\) bound
\begin{equation}
\frac{d}{dt}\|\nabla^{\perp}\rho\|_{L^{\infty}} \leq \|\nabla u\|_{L^{\infty}}\|\nabla^{\perp} \rho\|_{L^{\infty}}. \label{eq:apriori_L_infty}
\end{equation}
 Considering \(x, x' \in \Omega\) two arbitrary points in \(\Omega\), we suppress the time dependence for ease of notation and write \(z = \Phi_t(x), z' = \Phi_t(x')\). We now compute
\begin{align*}
\frac{d}{dt}\frac{|z|^{\alpha}\nabla^{\perp}\rho(z) - |z'|^{\alpha}\nabla^{\perp}\rho(z')}{|z-z'|^{\alpha}} &= I + II + III
\end{align*}
where \(I, II\) denote the terms from taking the time derivative of \(z\) in the numerator and denominator respectively and \(III\) denotes the term arising from the time derivative of \(\rho \circ \Phi\). Looking first at I, by the triangle inequality
\begin{align*}
|I| &= \frac{\alpha}{|z-z'|^{\alpha}}\left| \frac{z \cdot u(z)}{|z|^2} |z|^{\alpha}\nabla^{\perp}\rho(z) - \frac{z' \cdot u(z')}{|z'|^2}|z'|^{\alpha}\nabla^{\perp}\rho(z')\right| \\
&\leq \alpha \frac{1}{|z-z'|^{\alpha}}\left| \frac{z \cdot u(z)}{|z|^2} |z|^{\alpha}\nabla^{\perp}\rho(z) - \frac{z' \cdot u(z')}{|z'|^2}|z|^{\alpha}\nabla^{\perp}\rho(z)\right| + \alpha\frac{|u(z')|}{|z'|}\frac{||z|^{\alpha}\nabla^{\perp}\rho(z) - |z'|^{\alpha}\nabla^{\perp}\rho(z')|}{|z-z'|^{\alpha}}\\
&=: \alpha(I_{A} + I_{B}).
\end{align*}
The second term \(I_B\) can be bound
\begin{align*}
I_{B} \leq \alpha\|\nabla u\|_{L^{\infty}}\|\nabla^{\perp}\rho\|_{\Ca}.
\end{align*}
Using the triangle inequality, we have
\begin{align*}
    I_{A} &= \frac{|z|^{\alpha}|\nabla^{\perp}\rho(z)|}{|z-z'|^{\alpha}}\left|\frac{z \cdot u(z)}{|z|^2} - \frac{z' \cdot u(z')}{|z'|^2}\right| \\
&\leq 2\|\nabla^{\perp}\rho\|_{L^{\infty}}\|\nabla u\|_{L^{\infty}}\frac{|z-z'|^{1-\alpha}}{|z|^{1-\alpha}} + \|\nabla^{\perp}\rho\|_{L^{\infty}}\frac{|z|^{\alpha}}{|z-z'|^{\alpha}}\left|\frac{u(z') \cdot z'}{|z|^2} - \frac{z' \cdot u(z')}{|z'|^2}\right|.
\end{align*}
Without loss of generality we may assume \(|z'| \leq |z|, |z-z'| \leq 2|z|\). We then have
\begin{align*}
I_A &\leq 2\|\nabla^{\perp}\rho\|_{L^{\infty}}\|\nabla u\|_{L^{\infty}}\frac{|z-z'|^{1-\alpha}}{|z|^{1-\alpha}} +\|\nabla^{\perp}\rho\|_{L^{\infty}} \frac{|z|^{\alpha}}{|z-z'|^{\alpha}}\frac{1}{|z|^2|z'|}(|z|^2-|z'|^2)|u(z')| \\
&\leq C \|\nabla^{\perp}\rho\|_{L^{\infty}}\|\nabla u\|_{L^{\infty}} + \|\nabla^{\perp}\rho\|_{L^{\infty}}\frac{|u(z')|}{|z'|}\frac{|z-z'|}{|z-z|^{\alpha}}\frac{|z|^{\alpha}}{|z|}\frac{|z|+|z'|}{|z|} \\
&\leq C \|\nabla^{\perp}\rho\|_{L^{\infty}}\|\nabla u\|_{L^{\infty}}.
\end{align*}
The second term, arising from the time derivative of the denominator, can be bound,
\begin{align*}
|II| &\leq \alpha \left[|z|^{\alpha}\nabla^{\perp}\rho(z) - |z'|^{\alpha}\nabla^{\perp}\rho(z')\right]  \left|\frac{(z - z') \cdot (u(z) - u(z')}{|z-z'|^{\alpha + 2}}\right| \leq \alpha \|\nabla^{\perp}\rho\|_{\Ca} \|\nabla u\|_{L^{\infty}}.
\end{align*}
Finally, using the product bound \eqref{Ca_prod}
\begin{align*}
|III| &= \left|\frac{|z|^{\alpha}\nabla u(z) \nabla^{\perp}\rho - |z'|^{\alpha}\nabla u(z') \nabla^{\perp}\rho(z')}{|z-z'|^{\alpha}}\right|\leq \|\nabla u\|_{\Ca}\|\nabla^{\perp}\rho\|_{L^{\infty}}+\|\nabla u\|_{L^{\infty}}\|\nabla^{\perp}\rho\|_{\Ca}.
\end{align*}
Now, from \eqref{eq:IPM}
\[
\|\nabla u\|_{\Ca} \leq C \|D^2p\|_{\Ca} + \|\nabla \rho\|_{\Ca}
\]
and since \(\Delta p = -\p_{y}\rho\), from Lemma 3.5 of \cite{EJ-B}, \(D^2(-\Delta)^{-1}\) is bounded on \(\Ca(\overline{\Omega_\beta})\) so it follows that \(\|D^2 p\|_{\Ca} \leq C \|\nabla^{\perp}\rho\|_{\Ca}\). Hence, \(\|\nabla u\|_{\Ca} \leq C \|\nabla \rho\|_{\Ca}\) so that
\[
|III| \leq C \|\nabla^{\perp} \rho\|_{\Ca}\|\nabla^{\perp}\rho\|_{L^{\infty}}+ C\|\nabla u\|_{L^{\infty}}\|\nabla^{\perp}\rho\|_{\Ca}.
\]
Altogether, integrating we obtain the a priori bound
\[
\|\nabla^{\perp}\rho(t)\|_{\Ca} \lesssim \|\nabla^{\perp}\rho_0\|_{\Ca} + \int_{0}^{t}(\|\nabla u (s)\|_{L^{\infty}} + \|\nabla^{\perp}\rho (s)\|_{L^{\infty}})\|\nabla^{\perp}\rho(s)\|_{\Ca}\d s
\]
Combining this apriori estimate with \eqref{eq:apriori_L_infty}, we see that the solution can be continued beyond \(T\) provided
\[
\int_{0}^{T}\|\nabla u(s)\|_{L^\infty}\d s < \infty.
\]
Existence can now be proven using a standard iteration scheme as in \cite{EJ-B}. 

Finally, we prove uniqueness of the solution. Suppose \(\rho_1, \rho_2\) are two solutions as in Proposition \ref{prop:lwp_si} with corresponding velocity fields \(u_1, u_2\) respectively. Consider the differences \(\overline{\rho} = \rho_1 - \rho_2, u = u_1 - u_2\). Since \(\nabla^{\perp}\overline{\rho} \in \Ca\), it follows that \(\overline{\rho}/|x| \in L^\infty\). Then,
\[
\p_t \overline{\rho} + u_1 \cdot \nabla \overline{\rho} + \overline{u} \cdot \nabla \rho_2 = 0.
\]
and composing with the flow, it follows that
\begin{equation}
\frac{d}{dt}\left\|\frac{\overline{\rho}}{|x|}\right\|_{L^\infty} \leq \left\|\frac{u_1}{|x|}\right\|_{L^\infty}\left\|\frac{\overline{\rho}}{|x|}\right\|_{L^\infty} + \left\|\nabla \rho_2\right\|_{L^\infty}\left\|\frac{\overline{u}}{|x|}\right\|_{L^\infty}. \label{eq:rho_x_infty}
\end{equation}
Now, for the second term we note that
\begin{equation*}
    \left\|\frac{\overline{u}}{|x|}\right\|_{L^\infty} \leq \left\|\frac{\nabla \overline{p}}{|x|}\right\|_{L^\infty} + \left\|\frac{\overline{\rho}}{|x|}\right\|_{L^\infty}
\end{equation*}
and proceeding as in \cite{EJ-B} we have that \(\|\nabla \overline{p}/|x|\|_{L^\infty}\) can be bound by \(\|\overline{\rho}/|x|\|_{L^\infty}\) with only a logarithmic loss,
\begin{equation*}
    \left\|\frac{\nabla \overline{p}}{|x|}\right\|_{L^\infty} \leq C\left\|\frac{\overline{\rho}}{|x|}\right\|_{L^\infty}\left(1 + \log\left(\frac{\|\nabla \overline{\rho}\|_{L^\infty}}{\|\overline{\rho}/|x|\|_{L^\infty}}\right)\right).
\end{equation*}
Thus, from \eqref{eq:rho_x_infty} we conclude
\[
\frac{d}{dt}\left\|\frac{\overline{\rho}}{|x|}\right\|_{L^\infty} \leq C_{\beta, \nabla u_1, \nabla \rho_2}\left\|\frac{\overline{\rho}}{|x|}\right\|_{L^\infty}\left(1 + \log\left(\frac{\|\nabla \overline{\rho}\|_{L^\infty}}{\|\overline{\rho}/|x|\|_{L^\infty}}\right)\right).
\]
Since \(\overline{\rho}(0, x) \equiv 0\), this implies \(\overline{\rho}(t, x) \equiv 0\) completing the proof.
\end{proof}

\section{Construction of the Profile}
We recall that the IPM equation has the scaling symmetry 
\[
\rho(t,x) \mapsto \lambda^{-1}\rho(t, \lambda x), \quad u(t, x) \mapsto \lambda^{-1}u(t, \lambda x).
\]
Thus, \(1\)-homogeneity of \(\rho\) and \(u\) are propagated. This leads us to consider scale-invariant solutions of the form
\[
\rho(t, r, \theta) = rP(t, \theta), \quad \Psi(t, r, \theta) = r^2G(t, \theta).
\]
To derive the \(1D\) system for \(P, G\), we recall
\[
\nabla^{\perp} = -\hat{r}r^{-1}\p_\theta + \hat{\theta}\p_r, \quad \Delta = \p_{rr} + r^{-1}\p_r + r^{-2}\p_{\theta\theta}, \quad \p_{x_2} = r^{-1}\cos\theta\p_\theta + \sin\theta\p_r
\]
where \(\hat{r}, \hat{\theta}\) are the usual unit vectors in the \(r, \theta\) directions respectively. Thus, the Biot--Savart law \(\Delta \Psi = \p_{x_2} \rho\) gives
\[
G''(t, \theta) + 4G(t, \theta) = P(t, \theta)\sin\theta + P'(t, \theta)\cos\theta.
\]
where \('\) indicates an angular derivative \(\p_\theta\). Since \(\Psi\) is constant on the boundary due to the no penetration boundary conditions, we have the boundary conditions \(G(0) = G(L) = 0\). The velocity field is then given by \(u = -\hat{r}rG' + 2\hat{\theta}rG\) and then the transport equation \(\p_t \rho + u \cdot \nabla \rho = 0\) gives
\[
\p_t P(t, \theta) + 2G(t, \theta)P'(t, \theta) = G'(t, \theta)P(t, \theta).
\]
Thus, we obtain the following 1D system for \(P, G\)
\begin{align}\label{eq:IPM_1d}
\begin{cases}
\p_t P(t, \theta) + 2G(t, \theta)P'(t, \theta) = G'(t, \theta)P(t, \theta) \\
G''(t, \theta) + 4G(t, \theta) = P(t, \theta)\sin\theta + P'(t, \theta)\cos\theta \\
G(\cdot, 0) = G(\cdot, L) = 0 \\
P(0, \theta) = P_0(\theta).
\end{cases}
\end{align}
We will prove the following theorem about the existence of solutions to \eqref{eq:IPM_1d} which blow-up in finite time.
\begin{theorem}\label{theorem:ipm_1d}
    For any \(L' < \pi/2\) there exists \(L' < L < \pi/2\) such that there exists an even initial datum \(P_0 \in C^{\infty}([-L, L])\) such that \(\text{supp}(P_0) \subset [0, L-\delta]\) for some \(\delta > 0\), and the unique, local in time solution of \eqref{eq:IPM_1d} posed on \([-L, L]\) blows-up in finite time.
\end{theorem}
Our approach will be to first construct a blow-up profile and then cut off this profile near the boundary. To this end, we now seek a solution to the system \eqref{eq:IPM1D} of the form \(P(t, \theta) = (1-t)^{-1}P_{*}(\theta)\) where \(P_*\) is even. The profile \((P_*, G_*)\) then satisfies the following ODE system
\begin{align}\label{eq:profile_ode_P2}
\begin{cases}
    P_*(\theta) + 2 G_*(\theta)P_*'(\theta) = G_*'(\theta)P_*(\theta) \\
    G_*''(\theta) + 4G_*(\theta) = P_*(\theta)\sin\theta + P_*'(\theta)\cos\theta \\
    G_*(0) = G_*(L) = 0.
\end{cases}
\end{align}
In order for \(\rho\) to vanish on the boundary of \(\Omega\), we seek a solution of \eqref{eq:profile_ode_P2} for which \(P_*(L) = 0\). Evaluating \eqref{eq:profile_ode_P2} at \(\theta = 0\) we obtain the condition \(P_*(0) = G_*'(0)P_*(0)\). Thus, in order to have non-zero mass at \(\theta = 0\), we must impose the condition \(G_*'(0) = 1\). With this added condition, \eqref{eq:profile_ode_P2} can be transformed into an initial value problem with the initial conditions \(G_*(0) = 0, G_*'(0) = 1\). We can now instead seek solutions of the initial value problem 
\begin{align}\label{eq:profile_ode_P_loc}
\begin{cases}
    P_*(\theta) + 2G_*(\theta)P_*'(\theta) = G_*'(\theta)P_*(\theta) \\
    G_*''(\theta) + 4G_*(\theta) = P_*(\theta)\sin\theta + P_*'(\theta)\cos\theta \\
    G_*(0) = 0, \quad G_*'(0) = 1.
\end{cases}
\end{align}
for which \(G_*(L) = 0\). For most choices of \(P_*(0)\) however, it is not possible to obtain a solution satisfying \(G_*(L) = 0\). We make a particular choice of \(P_*(0)\) which allows a degree of freedom corresponding to choosing \(P_*''(0)\). This extra degree of freedom allows us to perform a shooting method argument to show that there exists a choice of \(P_*''(0)\) for which the second boundary condition \(G_*(L) = 0\) is satisfied. The particular choice of \(P_*(0)\) corresponds to fixing the strength of the background flow. This same phenomenon appears in the work \cite{EP1} where the strength of the background flow must be finely tuned through a parameter \(A_*\) in order to construct a profile.

To construct a profile we then proceed as follows. First, we solve \eqref{eq:profile_ode_P_loc} locally near \(\theta = 0\). This is done by Taylor expanding and then closing a fixed point argument on the remainder. We then prove that the solution obeys a key monotonicity property which allows us to continue the local solution until it necessarily crosses zero. In this way, we obtain a solution to the boundary value \eqref{eq:profile_ode_P}.

\subsection{Local Solution of the Profile Equation}
Now, we construct a solution of the system \eqref{eq:profile_ode_P_loc}, on a small interval \([0, a]\) which satisfies \(M' \leq 0\). The solution is constructed by considering a Taylor polynomial approximation of the profile at \(\theta = 0\) and then closing a fixed point scheme for the fourth derivative. First, we consider the quantity \(M_*\), defined by \(P_*(\theta) = M_*(\theta)\cos\theta\). Then \(M_*\) then satisfies the local system
\begin{align} \label{eq:profile_ode_M_loc}
\begin{cases}
    M_*(\theta) + 2G_*(\theta)M_*'(\theta) = G_*'(\theta)M_*(\theta) + 2G_*(\theta)M_*(\theta)\tan(\theta) \\
    G_*''(\theta) + 4G_*(\theta) = M_*'(\theta)\cos^2(\theta) \\
    M_*(0) = M_0, \quad G_*(0) = 0, \quad G_*'(0) = 1.
\end{cases}
\end{align}
We wish to satisfy the second boundary condition \(G_*(L) = 0\) via a shooting method argument. This requires an extra degree of freedom which we will obtain through a careful choice of \(M_*(0)\). Taking two derivatives of \eqref{eq:profile_ode_M_loc} and evaluating at \(\theta = 0\) using the symmetries of \(M_*, G_*\) as well as the condition \(G_*'(0) = 1\) we arrive at
\begin{align}\label{eq:taylor_exp}
    4M_*''(0) = M_*(0)(G_*'''(0) + 4), \quad G_*'''(0) + 4 = M_*''(0)
\end{align}
which yields \(4 M_*''(0) = M_*(0) M_*''(0)\). Thus, we see that the choice \(M_*(0) = 4\) allows \(M_*''(0)\) to be chosen freely. Any other choice of \(M_*(0)\) forces \(M_*''(0) = 0\). In fact, it is easily seen that there is a trivial ``solution'' \(M_*(\theta) \equiv M_0, G_*(\theta) = \frac12 \sin(2\theta)\) to \eqref{eq:profile_ode_M_loc}. Unfortunately, this satisfies \(G_*(\pi/2) = 0\) for which the Biot--Savart law \(G_*'' + 4G_* = M_*'\cos^2(\theta)\) is no longer well-posed. This trivial solution for which \(M_*\) is constant corresponds to the classical, unstably stratified steady state \(\rho(x_1, x_2) = x_1\) which exhibits the Rayleigh--Taylor instability. That there is a nontrivial velocity field generated by the stream function \(G_* = \frac12 \sin(2\theta)\) is an artifact of the ill-posedness of the Biot--Savart law on the half-plane. However, choosing \(M_*(0) = 4\), we are able to freely choose \(M_*''(0)\) and avoid the trivial constant solution. In order to obtain a solution which decreases towards the boundary it is required that we choose \(M_*''(0) < 0\). This leads us to the following proposition regarding the local solvability of \eqref{eq:profile_ode_M_loc}.

\begin{proposition}\label{prop:profile_local}
	For any \(A > 0\), there exists \(a > 0\) such the system  \eqref{eq:profile_ode_M_loc} has a unique solution \((M_*, G_*) \in C^{\infty}([0, a]) \times C^{\infty}([0, a])\)) with \(M_*(0) = 4, M_*'(0) = 0, M_*''(0) = -2A\) and \(G_*(0) = 0\). In particular, there exists a non-constant solution to \eqref{eq:profile_ode_M_loc} such that \(M_*'(\theta) \leq 0\) for \(\theta \in [0, a]\).
\end{proposition}
The difficulty in obtaining a local solution of \eqref{eq:profile_ode_M_loc} comes from the singular behaviour near \(\theta = 0\) as \(G_*(0) = 0\). This prevents us from directly using standard fixed point methods to solve the ODE. This difficulty can be overcome by taking the Taylor polynomial as an approximate solution and closing a fixed point argument for the remainder. Proceeding in this way, we will obtain a system of the form
\[
    \theta y'(\theta) + Ay = b + R(y, y', \theta)
\]
where \(A\) is a constant coefficient, positive-definite matrix, \(b\) is a constant vector and \(R\) is a remainder which is small in \(\theta\). In the Appendix (Lemma \ref{lemma:sing_ode}), we prove that under suitable smallness assumptions on the remainder such systems can be solved locally through a fixed point method.

\begin{proof}
Given the above discussion, it behooves us to expand
\[
    M_*(\theta) = 4 - A\theta^2 + \theta^4m(\theta), \quad G_*(\theta) = \theta - \frac{A+2}{3}\theta^3  + \theta^5 g(\theta)
\]
where we have substituted the second order Taylor expansion obtained from \eqref{eq:taylor_exp} and set \(M_*''(0) = -2A\). We now wish to solve for the remainders \(m, g\) via a fixed point scheme. Since our ansatz satisfies the equation up to fourth order, we obtain the following system for \(m, g\)
\begin{align}\label{eq:mg_system}\begin{cases}
    \theta m' + 4m -10g -2\theta g' = C_{1, A} + R_1 \\
	\theta^2g'' + 10\theta g' + 20g - \theta m' - 4m = C_{2, A} + R_2
    \end{cases}
\end{align}
where \(C_{1, A}, C_{2, A}\) are constants depending on \(A\) and \(R_1, R_2\) are the following remainders,
\begin{align*}
\begin{dcases}
	R_1 = m'(\theta)\left[\frac{A+2}{3}\theta^3 - \theta^5 g\right] + \frac{m(\theta)}{2}\left[\frac{5(A+2)\theta^2}{3} - 3\theta^4g + \theta^4 (\theta g')+2\tan\theta\left(\theta+\theta^5g-\frac{(A+2)\theta^3}{3}\right)\right] \\
    \qquad + g(\theta)\left[2\theta^2 A - \frac{5A}{2}\theta^2 + \theta\tan\theta(4 - A\theta^2)\right] - \frac{A}{2}\theta^2(\theta g') \\
    \qquad + \frac{4\theta^2}{9}\left[9t(\theta) - (A+2)\theta^2 - 3(A+2)\theta^4t(\theta)\right]\\
	R_2 = -4\sin^2\theta m(\theta) - \theta \sin^2\theta m'(\theta) \\
    C_{1, A} = -\frac{A+2}{3}\left(A + 8\right) + \frac83 - 2A \\
    C_{2, A} = 2A
\end{dcases}
\end{align*}
and in the above expression \(t(\theta) := \theta^{-5}(\tan\theta - \P_3\tan\theta)\).
Adding the first equation to the second in \eqref{eq:mg_system} yields the system
\begin{align*}
\begin{cases}
	\theta m' + 4m -10g -2\theta g' = F_1\\
	\theta^2g'' + 8\theta g' + 10g  =  F_2,
\end{cases}
\end{align*}
where \(F_2 = C_2 + R_2 +  C_1 + R_1\). Re-writing as a first order system, we let \(h = \theta g'\) to obtain
\begin{align}
\begin{cases}
    \theta m' + 4m - 10g - 2h = F_1 \\
    \theta h' + 7h + 10 g = F_2 \\
    \theta g' - h = 0.
\end{cases}
\end{align}
Using variation of parameters then yields the following integral system
\begin{align}\label{eq:profile_integral}
\begin{dcases}
    	m(\theta) = \frac{1}{\theta^4}\int_{0}^{\theta}\phi^{3}\left[F_1(\phi) - F_2(\phi)\right]\d\phi + \frac{1}{\theta^2}\int_{0}^{\theta}\phi F_2(\phi)\d\phi \\
	g(\theta) = -\frac{1}{3\theta^5}\int_{0}^{\theta}\phi^4 F_2(\phi)\d\phi + \frac{1}{3\theta^2}\int_{0}^{\theta}\phi F_2(\phi)\d\phi \\
	\theta g'(\theta) = \frac{5}{3\theta^5}\int_{0}^{\theta}\phi^4 F_2(\phi)\d\phi - \frac{2}{3\theta^2}\int_{0}^{\theta}\phi F_2(\phi)\d\phi.
\end{dcases}
\end{align}
To obtain a system of the form in Lemma \ref{lemma:sing_ode}, we must eliminate the dependence on \(m'\) in \(R_1, R_2\). For \(\lambda > 0\), integrating by parts gives
\begin{align*}
&\theta^{-\lambda}\int_{0}^{\theta}\phi^{\lambda-1}m'(\phi)\left[\frac{A+2}{3}\phi^3 - \phi^5 g\right]\d\phi  \\
&= \frac{A+2}{3}\theta^2m(\theta) - \theta^4 m(\theta)g(\theta)-\theta^{-\lambda}\int_{0}^{\theta}\phi^{\lambda-1}\left[\phi^2\frac{(A+2)(\lambda+2)}{3} - \phi^4(\lambda+4)g - \phi^4 h\right]m(\phi)\d\phi.
\end{align*}
Proceeding similarly for \(R_2\),
\begin{align*}
\theta^{-\lambda}\int_{0}^{\theta}-\theta \sin^2\theta m'(\theta)\phi^{\lambda-1}\d\phi &=  - \sin^2\theta m(\theta) + \theta^{-\lambda}\int_{0}^{\theta}\phi^{\lambda-1}\left[\lambda\sin^2\phi + 2\phi \sin\phi\cos\phi\right]m(\phi)\d\phi.
\end{align*}
We therefore obtain a fixed point problem of the form considered in Lemma \ref{lemma:sing_ode}. Applying Lemma \ref{lemma:sing_ode}, we obtain a local solution \(m, g \in C^{\infty}([0, a])\) for some \(a > 0\). By construction, \(m, g\) solve \eqref{eq:mg_system} and thus \(M_*, G_*\) solve \eqref{eq:profile_ode_M_loc} with the initial conditions of Proposition \ref{prop:profile_local}.
\end{proof}
\subsection{Monotonicity Lemma}
With a local solution to the profile equation \eqref{eq:profile_ode_M_loc} in hand, we now wish to extend the solution. The key fact which allows us to continue the solution is the following monotonicity property of \(M_*\).
\begin{lemma}\label{lemma:monotonicity}
Let \((M_*, G_*)\) solve the system \eqref{eq:profile_ode_M_loc} with the initial conditions 
\[
M_*(0) = M_0, \quad G_*(0) = 0, \quad G_*'(0) =~1
\]
on a domain \([0, \theta_*]\). Assume moreover that
\begin{itemize}
\item \(M_* > 0\) on \([0, \theta_*)\) 
\item There exists an interval \([0, \delta], \delta > 0\) such that \(M_*'(\theta) < 0\) for all \(\theta \in (0, \delta]\).
\end{itemize}
Then, \(M_*'(\theta) \leq 0\) for all \(\theta \in [0, \theta_*]\). \label{key_monotone}
\end{lemma}
\begin{proof}
Differentiating \eqref{eq:profile_ode_M_loc} gives
\begin{align*}
    M_*' + 2G_*'M_*' + 2G_*M_*'' &= G_*'M_*' + (2G_*\tan\theta)'M_* + 2G_*\tan\theta M_*' + G_*''M_* \\
    2G_*M_*'' &= M_*'(-G_*'-1 + 2G_*\tan\theta) + M_*((2G_*\tan\theta)' + G_*'')
\end{align*}
Note that if \(G_*(\theta_1) = 0\) at some point \(0 < \theta_1 < \theta_*,\) then \(M_*(\theta_1) = 0\) since evaluating \eqref{eq:profile_ode_M_loc} at \(\theta_1\) yields \(M_*(\theta_1) = G_*'(\theta_1)M_*(\theta_1)\). Since \(G_* \geq 0\) near \(\theta = 0\), we must have \(G_*'(\theta_1) \leq 0\) and hence \(M_*(\theta_1) = 0\) contradicting the assumption that \(M_* > 0\). Thus, we have \(G_* \geq 0\) on \([0, \theta^*]\). It then suffices to show \((2G_*\tan\theta)' + G_*'' \leq 0\) since \(M_* \geq 0\). Note that \(G_*\) can be recovered explicitly from \(M_*\) as
\begin{align}\label{eq:G_formula}
    G_*(\theta) = \frac{1}{2}\sin(2\theta)\left[1 + \int_{0}^{\theta}M_*'(\theta')\cos^2(\phi)\cos(2\phi)\d\phi\right] - \frac{1}{2}\cos(2\theta)\left[\int_{0}^{\theta}M_*'(\phi)\cos^2(\phi)\sin(2\phi)\d\phi\right]. 
\end{align}
We can then compute,
\begin{align}\label{eq:G'_formula}
    G_*'(\theta) &= \cos(2\theta)\left[1 + \int_{0}^{\theta}M_*'(\phi)\cos^2(\phi)\cos(2\phi)\d\phi\right] + \sin(2\theta)\left[\int_{0}^{\theta}M_*'(\phi)\cos^2(\phi)\sin(2\phi)\d\phi\right] \\
    G_*''(\theta) &= M_*'(\theta)\cos^2(\theta) - 2\sin(2\theta)\left[1 + \int_{0}^{\theta}M_*'(\phi)\cos^2(\phi)\cos(2\phi)\d\phi\right] \\ & \qquad + 2\cos(2\theta)\left[\int_{0}^{\theta}M_*'(\phi)\cos^2(\phi)\sin(2\phi)\d\phi\right].\nonumber
\end{align}
Thus, we have the following explicit expression for the quantity of interest \(G_*'' + (2G_*\tan\theta)'\),
\begin{align*}
    G_*'' + (2G_*\tan\theta)' &= \left[-2\sin(2\theta) + \sec^2(\theta)\sin(2\theta) + 2\tan(\theta)\cos(2\theta)\right]\int_{0}^{\theta}M_*'(\phi)\cos^2(\phi)\cos(2\phi)\d\phi\\
    &\qquad + \left[2\cos2\theta - \sec^2(\theta)\cos(2\theta) + 2\tan\theta\sin2\theta\right]\int_{0}^{\theta}M_*'(\phi)\cos^2(\phi)\sin(2\phi)\d\phi \\
    & \qquad + M_*'(\theta)\cos^2(\theta) - 2\sin(2\theta) + 2\sec^2(\theta)\frac{1}{2}\sin(2\theta) + 2\tan(\theta)\cos(2\theta)
\end{align*}
which simplifies to
\begin{align*}
 G_*'' + (2G_*\tan\theta)' &=   M_*'(\theta)\cos^2\theta + \sec^2(\theta)\int_{0}^{\theta}M_*'(\phi)\cos^2(\phi)\sin(2\phi)\d\phi.
\end{align*}
Thus, if \(M_*'(\theta) \leq 0\) it follows that \(G_*''(\theta) + (2G_*\tan\theta))' \leq 0\). Since, we have \(M_*' \leq 0\) on \([0, \delta]\) a straightforward bootstrapping argument allows us to iteratively conclude \(M_*'(\theta) \leq 0\) for all \(\theta \in [0, L]\).
\end{proof}

\subsection{Continuation of the Profile}
Here, we argue using standard Cauchy--Lipschitz theory and the monotonicity property of Lemma \ref{lemma:monotonicity} that the previously constructed local solution of the profile equation can be continued to construct a global profile on an interval \([0, L]\). 

\begin{lemma}\label{lemma:profile_continuation}
    Consider the system
    \begin{align*}
    \begin{cases}
        M_* + 2G_*M_*' = G_*'M_* + 2M_*G_*\tan\theta \\
        G_*'' + 4G_* = M_*'\cos^2\theta \\
        M_*(\theta_0) = M_0, \quad G_*(\theta_0) = G_0, \quad G_*'(\theta_0) = G_0'.
    \end{cases}
    \end{align*}
    If \(G_0 > 0\) and \(\theta_0 < \pi/2\), then there exists \(\delta > 0\) such that there is a unique solution \(M_*, G_* \in C^{\infty}([\theta_0, \theta_0 + \delta]) \times C^{\infty}([\theta_0, \theta_0 + \delta])\).
\end{lemma}
\begin{proof}
    We first see that \(G_*\) can be recovered explicitly from \(M_*\) by the formula
    \begin{equation}\label{eq:G_continued}
    \begin{aligned}
    G_*(\theta) &= \frac{1}{2}\sin(2\theta)\left[G_s + \int_{\theta_0}^{\theta}M_*(\phi)(\sin2\phi\cos2\phi + 2\cos^2\phi\sin2\phi)\d\phi)\right]\\
    &\qquad - \frac{1}{2}\cos(2\theta)\left[ G_c - \int_{\theta_0}^{\theta}M_*(\phi)(-\sin^22\phi + 2\cos^2\phi\cos2\phi)\d\phi\right]
    \end{aligned}
    \end{equation}
    where \(G_s, G_c\) are the following constants
    \[
    G_s = 2G_0\sin2\theta_0 + \cos2\theta_0G_0'- M_0\cos^2\theta_0\cos2\theta_0, \  G_c = -2G_0\cos2\theta_0 + G_0'\sin2\theta_0 - M_0\cos^2\theta_0\sin2\theta_0.
    \]
    Since \(G_0 > 0\), we can divide by \(G_*\) to find that
    \[
    M_*'(\theta) = \frac{M_*(\theta)}{2G_*(\theta)}\big(G_*'(\theta) - 1 + 2G_*(\theta)\tan\theta\big)
    \]
    which yields the fixed point problem
    \begin{equation}\label{eq:M_continued}
    M_*(\theta) = M_0 + \int_{0}^{\theta}\frac{M_*(\phi)}{2G_*(\phi)}\big(G_*'(\phi) - 1 + 2G_*(\phi)\tan\phi\big)\d\phi
    \end{equation}
    which we consider on a ball \(B_R(0) \subset C^0([\theta_0, \theta_0+\delta])\) where \(R = 2|M_0|+1\). By choosing \(\delta\) sufficiently small, we can ensure that for all \(M_* \in B_R(0)\), the corresponding \(G_*\) given by \eqref{eq:G_continued} satisfies \(G_*(\theta) > G_0/2\) for all \(\theta \in [\theta_0, \theta_0 + \delta]\). It is then easily verified that for \(\delta\) sufficiently small (depending on \(M_0, G_0, G_0'\)) the right hand side of \eqref{eq:M_continued} maps \(B_R(0)\) to itself. For \(M_1, M_2 \in B_R(0)\) we define corresponding \(G_1, G_2\) by \eqref{eq:G_continued}. It is then easily seen that
    \begin{multline*}
    \left|\int_{0}^{\theta}\frac{M_1(\phi)}{2G_1(\phi)}\big(G_1'(\phi) - 1 + 2G_1(\phi)\tan\phi\big) - \frac{M_2(\phi)}{2G_2(\phi)}\big(G_2'(\phi) - 1 + 2G_2(\phi)\tan\phi\big)\d\phi\right| \\
    \leq C\delta R\left[\|M_1 - M_2\|_{C^0} + \|G_1 - G_2\|_{C^1} + \left\|\frac{1}{G_1}-\frac{1}{G_2}\right\|_{C^0}\right]
    \end{multline*}
    for some constant \(C > 0\) depending on the initial conditions. From \eqref{eq:G_continued}, it follows that \(\|G_1 - G_2\|_{C^1} \leq C\|M_1-M_2\|_{C^0}\) and since \(G_1, G_2 \geq G_0/2\) we conclude
    \begin{align*}
    \left|\int_{0}^{\theta}\frac{M_1(\phi)}{2G_1(\phi)}\big(G_1'(\phi) - 1 + 2G_1(\phi)\tan\phi\big) - \frac{M_2(\phi)}{2G_2(\phi)}\big(G_2'(\phi) - 1 + 2G_2(\phi)\tan\phi\big)\d\phi\right|
    \leq C\delta R\|M_1 - M_2\|_{C^0}.
    \end{align*}
    Choosing \(\delta\) sufficiently small, we conclude the mapping is a contraction and hence has a unique fixed point \(M_* \in C^0([\theta_0, \theta_0 + \delta])\). It then follows from \eqref{eq:G_continued} that \(G_* \in C^1([\theta_0, \theta_0+\delta])\) and hence from \eqref{eq:M_continued} we conclude that, in fact, \(M_* \in C^1([\theta_0, \theta_0+\delta])\). Continuing in this way, we conclude \(M_*, G_* \in C^{\infty}([\theta_0, \theta_0+\delta])\).
\end{proof}

\begin{proposition}\label{prop:profile_existence}
    For all \(A > 0\), there exists \(L < \pi/2\) such that there exists a unique solution  \((M_*, G_*) \in C^{\frac12}([0, L]) \times C^{1, \frac12}([0, L]) \cap C^{\infty}([0, L)) \times C^{\infty}([0, L))\) of \eqref{eq:profile_ode_M_loc} such that \(M_*''(0) = -2A\) and the boundary conditions \(G_*(L) = M_*(L) = 0\) are satisfied. 
\end{proposition}
\begin{proof}
By Proposition \ref{prop:profile_local}, there exists \(\delta > 0\) and \(M_*, G_* \in C^{\infty}([0, \delta])\) solving \eqref{eq:profile_ode_M_loc} such that \(M_*'(\theta) \leq 0\) for all \(\theta \in [0, \delta]\). Then, by Proposition \ref{lemma:profile_continuation} the solution can be continued until a point \(L\) where either \(G_*(L) = 0\) the first root of \(G\), or up to \(L = \pi/2\). First, note that \(M_*\) must be positive on \([0, L)\). Indeed, if \(M_*(z) = 0\) for some \(z < L\) then solving backwards from \(z\), by Gr\"onwall's inequality we would have \(M_* \equiv 0\) on \([0, L]\), a contradiction. Now, we claim that \(L < \pi/2\). Indeed, if the solution can be continued all the way up to \(\pi/2\), then by \eqref{eq:G_formula} and integration by parts we have that
\begin{align*}
\lim_{\theta \to \pi/2}G_*(\theta) &= \lim_{\theta \to \pi/2}\frac12\sin2\theta\left[1 + M_*(\theta)\cos(2\theta)\cos^2(\theta)- 4 + \int_{0}^{\theta}M_*(\phi)\sin(2\phi)(\cos2\phi + \cos^2\phi)\d\phi\right] \\
&\qquad + \lim_{\theta \to \pi/2}\int_{0}^{\theta}M_*'(\phi)\cos^2(\phi)\sin(2\phi)\d\phi.
\end{align*}
Since \(0 \leq M_* \leq 4\) we conclude that the first term goes to zero and thus
\begin{align*}
\lim_{\theta \to \pi/2}G_*(\theta) &= \lim_{\theta \to \pi/2}\int_{0}^{\theta}M'(\phi)\cos^2(\phi)\sin(2\phi)\d\phi < 0
\end{align*}
since \(M_*' \leq 0\) and not identically zero. Thus, it must be the case that \(G_*\) has a root at point \(L < \pi/2\). At this point we then have
\[
M_*(L) = G_*'(L)M_*(L)
\]
and since \(G_*'(L) \leq 0\) as \(G_* \geq 0\) before \(\theta = L\) we conclude \(M_*(L) = 0\). It now remains to prove the regularity of \(M_*, G_*\) at the boundary \(\theta = L\). First, we show that \(G_*\) vanishes only linearly at the boundary. Indeed, since \(G_*(L) = 0\) from \eqref{eq:G_formula} it follows that
\begin{equation}\label{eq:G_vanish_1}
\sin(2L)\left[1+\int_{0}^{L}M'(\phi)\cos^2(\phi)\cos(2\phi)\d\phi\right] = \cos(2L)\int_{0}^{L}M'(\phi)\cos^2(\phi)\sin(2\phi)\d\phi 
\end{equation}
and evaluating the formula for \(G_*'\) at \(\theta = L\) we have
\begin{equation}\label{eq:G_vanish_2}
G_*'(L) = \cos(2L)\left[1 + \int_{0}^{L}M'(\phi)\cos^2(\phi)\cos(2\phi)\d\phi\right] + \sin(2L)\int_{0}^{L}M'(\phi)\cos^2(\phi)\sin(2\phi)\d\phi.
\end{equation}
Substituting \eqref{eq:G_vanish_1} into \eqref{eq:G_vanish_2} we conclude
\begin{equation}\label{eq:G_vanishing}
G_*'(L) = \left[\frac{\cos^2(2L)}{\sin(2L)} + \sin(2L)\right]\int_{0}^{L}M'(\phi)\cos^2(\phi)\sin(2\phi)\d\phi < 0.
\end{equation}
Now, taking the limit as \(\theta \to L\) in \eqref{eq:profile_ode_M_loc}, we see that
\[
\lim_{\theta \to L}-\frac{(L-\theta)M_*'(\theta)}{M_*(\theta)} = \lim_{\theta \to L}\frac{G_*'(\theta)-1}{2\frac{G_*(\theta)}{L-\theta}} = \frac{1}{2} - \frac{1}{2G_*'(L)}
\]
and it follows (see Lemma \ref{lemma:holder}) that \(M_* \in C^{\alpha}([0, L])\) for \(\alpha < \min\left\{1, \frac{1}{2} - \frac12 G_*'(L)\right\}\). Since \(G_*'(L) < 0\), in particular we have that \(M_* \in C^{1/2}([0, L])\). 
\end{proof}
Note that the profile obtained in Proposition \ref{prop:profile_existence} does not possess the correct regularity at the boundary demanded by Theorem \ref{theorem:main} and is supported up to the boundary \(\theta = L\). In the remaining sections we construct a suitable, stable perturbation of \(M_*\) capable of smoothly truncating the solution near the boundary.

\subsection{Shooting Method}
Here, we consider the family of solutions from Proposition \ref{prop:profile_existence}. We show that by taking \(A \to 0\), for any \(\epsilon > 0\), there exists \(A > 0\) such that the first root of \(G\) after \(\theta = 0\) occurs at \(L \in (\pi/2 -\epsilon, \pi/2)\). Thus, we can choose \(\Omega\) arbitrarily close to the half-plane. To do so, we leverage the trivial solution \(M \equiv 4, G = \frac{1}{2}\sin2\theta\) obtained for \(A = 0\) and prove a stability result for solutions from Proposition \ref{prop:profile_existence} in the parameter \(A\).
\begin{proposition}
    For every \(L' < \pi/2\), there exists \(A > 0\) such that the solution \((M, G)\) of Proposition \ref{prop:profile_existence} satisfies \(M(L) = G(L) = 0\) for some \(L' < L < \pi/2\).
\end{proposition}
\begin{proof}
    Fix \(L' < \pi/2\). By Proposition \ref{prop:profile_existence}, there exists \(L = L(A)\) such that \(G(L) = M(L) = 0\) and \(0 < L < \pi/2\) so it suffices to show that \(L > L'\) for \(A\) sufficiently small. First, examining the fixed point problem \eqref{eq:profile_integral} we see that each \(F_i\) satsifes \(\|F_i\|_{C^1} \leq C(1+A)\) and the constants \(C_{i, A}\) satisfy \(|C_{i, A}| \leq C(1+A^2)\) for some constant \(C > 0\) so by Lemma \ref{lemma:sing_ode} we have a local bound \(|M(\theta) - 4| \leq CA^2\theta^2 + CA^2\theta^4\) for all \(\theta \in [0, d]\) for some \(d > c/(1+A)\). Thus, on an interval \([0, d]\) independent of \(A < 1\), we have \(|M(\theta) - 4| \leq CA^2\theta^2\) for all \(\theta \in [0, d]\). Now, for \(\theta > d\) we write
    \[
    M'(\theta) = \frac{1}{2G(\theta)}\big(-M(\theta) + G'(\theta)M(\theta) + 2M(\theta)G(\theta)\tan\theta\big).
    \]
    Considering the difference equation for \(\overline{M} := M - 4, \overline{G} = G - \frac12 \sin2\theta\), we have
    \[
    \overline{M}'(\theta) = \frac{1}{2G(\theta)}\left[-\overline{M}(\theta) + \overline{M}(\theta)G'(\theta) + 4\overline{G}'(\theta) + 2\overline{M}(\theta)G(\theta)\tan\theta + 8\overline{G}(\theta)\tan\theta\right].
    \]
    Define \(k_* = \min\{1, \inf_{[d, L']}\frac12\sin2\theta\} > 0\). We claim there exists \(\epsilon, D, \gamma > 0\) such that
    \[
    |\overline{M}(\theta)| \leq D A^2 e^{\gamma(\theta-d)/k_*^2}.
    \]
    for all \(A < \epsilon\) and \(d \leq \theta \leq L'\). First, observe that since \(M(\theta) \leq 4\) for all \(\theta > 0\), it follows that \(G, G'\) are bounded uniformly in \(A\). Thus, noting that \(\tan\theta \leq C/k_*\) it follows that
    \[
    |\overline{M}'(\theta)| \leq \frac{C}{k_*G(\theta)}\left[|\overline{M}(\theta)| + |\overline{G}(\theta)| + |\overline{G}'(\theta)|\right]
    \]
    for some constant \(C\) depending only on \(d\).
    Now, we bootstrap the assumption that \(|\overline{M}(\theta)|\leq D A^2 e^{\gamma(\theta-d)/k_*^2}\) for \(\gamma > 1\) to be chosen later and \(A\) sufficiently small depending on \(D, L', \gamma\). First, we note that by \eqref{eq:G_formula}, \eqref{eq:G'_formula}, the bootstrap assumption implies 
    \[
    |\overline{G}(\theta)| \leq 10A^2(De^{\gamma(\theta-d)/k_*^2} + d^3), \quad |\overline{G}'(\theta)| \leq 10A^2(DAe^{\gamma(\theta-d)/k_*^2} + d^3).
    \]
    For \(A\) sufficiently small, we then have
    \[
    G(\theta) = \frac{1}{2}\sin2\theta + \overline{G}(\theta) \geq \frac{k_*}{2}
    \]
    and hence
    \[
    |\overline{M}'(\theta)| \leq \frac{C}{k_*^2}\left[|\overline{M}(\theta)| + |\overline{G}(\theta)| + |\overline{G}'(\theta)|\right].
    \]
    Integrating, from \(d\) and using the bootstrap assumptions we then have
    \begin{align*}
    |\overline{M}(\theta)| &\leq |\overline{M}(d)| + \frac{CDA}{k_*^2}\int_{d}^{\theta}e^{\gamma(\phi-d)/k_*^2}\d\phi \leq C_1d^2A^2 + \frac{CDA^2}{\gamma}(e^{\gamma(\theta-d)/k_*^2}-1).
    \end{align*}
    For \(D > 2\gamma d^2C_1\) and \(\gamma > 2C\) we conclude
     \begin{align*}
         |\overline{M}(\theta)| \leq \frac{DA^2}{2}e^{\gamma(\theta-d)/k_*^2},
     \end{align*}
    closing the bootstrap. We then conclude that for all \(0 \leq \theta \leq L'\),
    \[
    G(\theta) \geq k_* - |\overline{G}(\theta)| \geq k_* - Ce^{\gamma(L'-d)/k_*^2}A^2
    \]
    which is positive for \(A^2 < k_*e^{-\gamma(L'-d)/k_*^2}/C\). Thus, \(L > L'\) as desired.
\end{proof}

\section{The Space \texorpdfstring{\(\tH^4\)}{H4}}
From this point forward, we consider fixed \(L < \pi/2\). As the profile \(M_*\) of Proposition \ref{prop:profile_existence} is not regular at the boundary \(\theta = L\) we wish to construct a small perturbation of \(M_*\) which truncates it near the boundary yet still blows-up in finite time. This leads us to consider the logarithmic time \(s = -\log(1-t)\), for which the blow-up happens at \(s = \infty\). Then, we obtain the system
\begin{align}\label{eq:ipm_1d}\begin{cases}
    \p_s M + M + 2MG' = G'M + 2MG\tan\theta \\
    G'' + 4G = M'\cos^2\theta \\
    G(0) = G(L) = 0.
    \end{cases}
\end{align}
The profile \((M_*, G_*)\) is then a stationary solution of \eqref{eq:ipm_1d}. In order to obtain a solution which continues to blow-up in finite time, we are lead to studying perturbations of \(M_*\). A perturbation \(M\) of \(M_*\) satisfies
\begin{equation}
    \p_s M + \L(M) = \N(M, M), \label{eq:ipm_pert}
\end{equation}
where \(\L\) is defined by
\[
\L(f) := f + 2G_*f' + 2FM_*' - G_*'f - F'M_* - 2M_*F\tan\theta - 2fG_*\tan\theta,
\]
\(F\) solves
\[
F'' + 4F = f'\cos^2\theta
\]
and
\[
\N(f, g) := -2Fg' + F'g.
\]
Our goal is to produce decaying solutions of \eqref{eq:ipm_pert} capable of truncating \(M_*\) near the boundary. We first explain the mechanism through which we obtain decay. Since \(G_* \geq 0\) for \(\theta > 0\), the background flow is always directed out towards the boundary \(\theta = L\). Moreover, since \(G_*' > 0\) near \(\theta = 0\) and \(G_*' < 0\) near \(\theta = L\), there is stretching occurring near the origin and damping occurring near the boundary. Thus, we see that perturbations are transported away from the stretching region towards the damping region. 
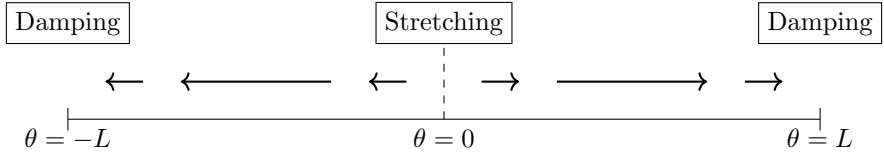
\begin{center}
    \begin{tikzpicture}[domain=-10:10, scale=0.5]
		\draw[-] (-10, 0) -- (10,0) ;
            \draw[->, thick] (3, 1) -- (7, 1);
            \draw[->, thick] (1, 1) -- (2, 1);
            \draw[->, thick] (8, 1) -- (9, 1);

            \draw[-] (10, -0.3) -- (10, 0.3);
            \draw[-] (-10, -0.3) -- (-10, 0.3);
            \draw[dashed] (0,0) -- (0, 2);

            \draw[->, thick] (-3, 1) -- (-7, 1);
            \draw[->, thick] (-1, 1) -- (-2, 1);
            \draw[->, thick] (-8, 1) -- (-9, 1);

            \node[below, scale=1] at (0, 0) {\(\theta = 0\)};
            \node[below, scale=1] at (10, 0) {\(\theta = L\)};
            \node[below, scale=1] at (-10, 0) {\(\theta = -L\)};

            \node[above, draw, scale=1] at (0, 2) {Stretching};
            \node[above, draw, scale=1] at (10, 2) {Damping};
            \node[above, draw, scale=1] at (-10, 2) {Damping};
	\end{tikzpicture}
\end{center}
The velocity field vanishes at zero however so the transport away from this region is weak. To counteract this, we work in a weighted space with a strong weight at the origin and a weak weight near the boundary. This effectively enforces that the perturbation vanishes to high order near \(\theta = 0\) so that there is very little mass near zero.

In practice, proving such decay requires careful study of the linearized operator \(\L\) in a certain weighted Sobolev space, which we denote \(\tH^4\). Namely, we will show that the linearized operator \(\mathcal{L}\) can be decomposed \(\mathcal{L} = \Lb + \Lk\) where \(\Lb\) is coercive on \(\tH^4\) and \(\Lk\) is a finite-rank smoothing operator. Standard semigroup theory then yields a decomposition \(\tH^4 = \tH^4_S \bigoplus \tH^4_U\) into the stable and unstable subspaces respectively, with \(\tH^4_U\) being finite dimensional since \(\Lk\) is compact. This general approach of decomposing the linearized operator into a coercive part and a compact part was used in \cite{MRRS} and has seen numerous applications to blow-up problems (see \cite{CH1, EP1, CCSV_CompressVorticity} for a non-exhaustive list). 
While we are only able to prove finite co-dimension stability and not actual stability, it is likely that the profile is indeed stable (up to modulating the blow-up time).

We now begin the construction of the inner product \(\inner{\cdot}{\cdot}_{\tH^4}\) and the decomposition \(\L = \Lb + \Lk\). As we will be working with solutions \(M\) which are even in \(\theta\), define \(L^2_*\) to be the space of even, \(L^2\) functions on \([-L, L]\). Recall that \(\P_k(f)\) denotes the \(k^{th}\) Taylor polynomial of \(f\).
\begin{definition}\label{def:tH4}
Define the weighted space \(\tH^4 := \{f \in L^2_*: \|f\|_{\tH^4} < \infty\}\) where the norm \(\| \cdot \|_{\tH^4}\) is defined by the associated inner product
\begin{align}\label{eq:inner_product_1}
    \inner{f}{g}_{\tH^4} &= f(0)g(0) + f^{(2)}(0)g^{(2)}(0) + B\int_{0}^{L}(f-\mathbb{P}_2 f)(g - \mathbb{P}_2g)\varphi(\theta)\d\theta + \int_{0}^{L}f^{(4)}(\theta)g^{(4)}(\theta)\psi(\theta)\d\theta
\end{align}
where \(\varphi, \psi\) are the weights,
\begin{align}
\varphi(\theta) = \begin{cases}\label{eq:weights}
\theta^{-8} & \theta \leq \ell_1\\
\l_1^{K-8}\theta^{-K} & \ell_1 \leq \theta \leq  \ell_2 \\
\l_1^{K-8}\l_2^{-K} & \ell_2 \leq \theta \leq L
\end{cases}, \quad \psi(\theta) = \begin{cases}
1 & \theta \leq \ell_1 \\
\l_1^{K}\theta^{-K} & \ell_1 \leq \theta \leq \ell_2 \\
\l_1^{K}\l_2^{-K}(L-\l_2)^{-13/2}(L-\theta)^{13/2} & \ell_2 \leq \theta \leq L
\end{cases}
\end{align}
and \(K, \l_1, \l_2\) and \(B\) are positive constants to be chosen later with \(\l_1, L-\l_2 < 1\), \(K > 10\). 
\end{definition}
Note that the constant factors in the definition of the weights simply come from imposing that \(\varphi, \psi\) be Lipschitz continuous away from zero. Moreover, the exact exponent \(13/2\) is not important, one need only choose an exponent small enough that \(\tH^4\) embeds in \(L^{\infty}\) (Lemma \ref{lemma:infty_embed}) yet large enough that the induced norm is weaker than \(C^{1/2}([0, L])\) where \(M_*\) lies. 

We henceforth consider the linearized operator \(\L\) on its natural domain
\[
\mathcal{D}(\mathcal{L}) = \{f \in \tH^4 : \mathcal{L}(f) \in \tH^4\}.
\]
The next two sections are devoted to proving the following proposition about \(\mathcal{L}\).
\begin{theorem}\label{thm:coercivity}
    The operator \(\mathcal{L}\) can be decomposed \(\mathcal{L} = \Lb + \Lk\) where \(\Lb\) is maximally accretive on \(\tH^4\), and \(\Lk\) is a bounded, finite-rank operator. Here, \(\Lb\) is again considered on \(\mathcal{D}(\L)\). \label{prop:decomp}
\end{theorem}
We first set up the decomposition. Define
\[
\tM := M - \mathbb{P}_2M, \quad \tG = \tG_{loc}(\theta) - \tG_{nl}(\theta)
\]
where \(\tG_{loc}\) solves
\begin{equation}\label{eq:G_loc_eqn}
\tG_{loc}'' + 4\tG_{loc} = \tM'\cos^2\theta, \quad \tG_{loc}(0) = \tG_{loc}'(0) = 0 
\end{equation}
and \(\tG_{nl}\) is the following non-local term which serves to enforce that \(\tG(L) = 0\),
\[
\tG_{nl}(\theta) = \frac{\tG_{loc}(L)}{\cos(2L)}\cos(2\theta)\eta\left(\frac{L-\theta}{L-\l_2}\right).
\]
where \(\eta\) is a compactly supported, non-negative bump function such that \(\eta \equiv 1\) on \([0, 1/2]\), \(\eta\) vanishes on \([1, \infty)\) and \(|\eta'| \leq 3\).
The decomposition can then be taken as
\begin{align*}
\Lb(M) &= \tM + 2G_*\tM' + 2\tG M_*' - G_*'\tM - \tG' M_*  - 2G_*\tM\tan\theta - 2\tG M_*\tan\theta + \mathbb{P}_2 M \\
\mathcal{L}_K &= \mathcal{L} - \Lb.
\end{align*}
Note that by replacing \(M\) with \(\tM\) in \eqref{eq:G_loc_eqn} we have only introduced a finite-rank perturbation and similarly, replacing the boundary value problem with the initial value problem in \eqref{eq:G_loc_eqn} also only introduces a finite-rank perturbation. It is then clear that \(\Lk\) is finite-rank.
\begin{remark}
The compact perturbation now reduces \(G\) to a ``local'' part \(\tG_{loc}\) in the sense that \(\tG_{loc}(\theta)\) does not depend on the values of \(\tM(\phi)\) for \(\phi \geq \theta\).
\end{remark}
\noindent The main issue in proving Theorem \ref{thm:coercivity} is proving that \(\Lb\) is coercive.
\begin{proposition}
\(\Lb\) is coercive on \(\tH^4\). That is, there exists a constant \(c > 0\) such that for all \(M \in \tH^4\) \(\inner{\Lb(M)}{M}_{\tH^4} \geq c \|M\|_{\tH^4}^2\).
\end{proposition}
For ease of notation going forward, we define the following splitting of the operator \(\Lb\) into its local and nonlocal parts,
\[
\Lb_{L}(M) := \tM + 2G_*\tM' - G_*'\tM - 2\tM G_*\tan\theta, \quad \Lb_{NL} := \Lb - \Lb_L.
\]
The coercivity of \(\Lb\) will come entirely from the local piece \(\Lb_L\). The primary technicality in proving coercivity arises in showing the nonlocal contributions from \(\Lb_{NL}\) are small so as not to defeat the positivity of \(\Lb_{L}\). Proving coercivity requires estimates in both the high and low portion of the \(\tH^4\) norm in each of the three regions \([0, \l_1], [\l_1, \l_2], [\l_2, L]\) referred to as the local, bulk and endpoint regions from here on. Before proceeding with the coercivity estimates, we conclude this section with some useful Hardy type inequalities.

\subsection{Hardy-Type Inequalities}
Here, we prove some necessary Hardy-type inequalities. As a consequence, we obtain that the norm \(\|\cdot\|_{\tH^4}\) is equivalent to a standard Sobolev norm \(\|\cdot\|_{\H^4}\) with a weight at \(\theta = L\). In addition, we also show there is a continuous embedding \(\H^4 \hookrightarrow L^\infty\).

\begin{lemma}\label{lemma:equiv_norm}
Consider the norm \(\|\cdot\|_{\H^4}\) induced by the inner product
\[
\inner{f}{g}_{\H^4} = \int_{0}^{L}f(\theta)g(\theta)\d\theta + \int_{0}^{L}f^{(4)}(\theta)g^{(4)}(\theta)(L-\theta)^{13/2}\d\theta.
\]
Then, there exists a constant \(C > 0\) such that for all \(f \in \tH^4 \cap \H^4\)
\[
C^{-1}\|f\|_{\tH^4} \leq \|f\|_{\H^4} \leq C\|f\|_{\tH^4}.
\]
\end{lemma}
\begin{proof}
The lemma follows from elementary calculus and Hardy's inequality. Indeed, by the fundamental theorem of calculus,
\[
(f - \P_2 f)(\theta) = \int_{0}^{\theta}\p_\theta (f - \P_2 f)\d\phi
\]
and hence
\[
\int_{0}^{\l_1}\frac{(f - \P_2 f)^2}{\theta^8}\d\theta = \int_{0}^{\l_1}\left(\theta^{-4}\int_{0}^{\theta}\p_\theta (f - \P_2 f)(\phi)\d\phi\right)^2\d\theta \leq C\int_{0}^{\l_1}(\p_\theta(f - \P_2 f)(\theta))^2\theta^{-6}\d\theta
\]
where we have applied the standard Hardy's inequality in the final step. Iterating this inequality, we obtain
\[
\int_{0}^{\l_1}\frac{(f - \P_2 f)^2}{\theta^8}\d\theta \leq C \int_{0}^{\l_1}(\p_\theta^{4}(f - \P_2 f)(\theta))^2\d\theta. 
\]
Moreover, by the fundamental theorem of calculus
\[
f(0)^2 + f''(0)^2 \leq C(\|f'''\|_{L^1(0, \l_1)}^2 + \|f\|_{L^1(0,\l_1)}^2) \leq C\|f\|_{\H^4}^2.
\]
Since the associated weight \(\varphi\) defining \(\tH^4\) satisfies \(\varphi(\theta) \leq C(L-\theta)^{13/2}\) on \([\l_1, L]\), it then follows that there exists \(C > 0\) such that \(\|f\|_{\tH^4} \leq C\|f\|_{\H^4}\). For the reverse inequality, we note that
\[
\int_{0}^{\l_1}f(\theta)^2\d\theta = \int_{0}^{\l_1}\left(f(0) + \int_{0}^{\theta}f'(\phi)\d\phi\d\theta\right)^2 \leq Cf(0)^2 + C\int_{0}^{\l_1}f'(\theta)^2\d\theta.
\]
Iterating this inequality, it follows that
\[
\int_{0}^{\l_1}f(\theta)^2\d\theta \leq Cf(0)^2 + Cf^{(2)}(0)^2 + \int_{0}^{\l_1}f^{(4)}(\theta)^2\d\theta
\]
and we then have \(\|f\|_{\H^4} \leq C\|f\|_{\tH^4}\) for some \(C > 0\).
\end{proof}

\begin{lemma}\label{lemma:weighted_interpolation}
    For all \(k \geq 1, p \geq 0\), there exists \(C > 0\) such that for all \(f \in \tH^4\), 
    \begin{equation}
        \int_{\l_2}^{L}f^{(k)}(\theta)^2(L-\theta)^p \d\theta \leq C\int_{\l_2}^{L}f(\theta)^2\d\theta  + C \int_{\l_2}^{L} f^{(k+1)}(\theta)^2(L-\theta)^{p+2}\d\theta.
    \end{equation}
\end{lemma}
\noindent We omit the proof as it is similar to that of Lemma \ref{lemma:equiv_norm} and follows from elementary calculus. As a consequence, we obtain the \(L^\infty\) embedding.
\begin{lemma}\label{lemma:infty_embed}
    For all \(0 \leq j \leq 3\), there exists a constant \(C > 0\) such that for all \(f \in \H^4\),
    \[
    \|(L-\theta)^jf^{(j)}\|_{L^{\infty}} \leq C \|f\|_{\H^4}.
    \]
    In particular, there is a continuous embedding \(\H^4 \hookrightarrow L^\infty\).
\end{lemma}
\begin{proof}
We prove the case \(j = 0\) as the others follow similarly. By the fundamental theorem of calculus and H\"older's inequality,
    \[
    |f(\theta)| \leq \int_{0}^{L}|f(\phi)|\d\phi + \int_{0}^{L}|f'(\phi)|\d\phi \leq \|f\|_{L^2} + C\left(\int_{0}^{L}f'(\phi)^2(L-\phi)^{1/2}\d\phi\right)^{1/2}.
    \]
    Repeatedly applying Lemma \ref{lemma:weighted_interpolation} gives
    \[
    |f(\theta)| \leq \|f\|_{L^2} + C\left(\int_{0}^{L}f^{(4)}(\phi)^2(L-\phi)^{13/2}\d\phi\right)^{1/2} \leq C\|f\|_{\H^4}
    \]
    as desired.
\end{proof}
\noindent Finally, we prove that \(\H^4\) (and thus \(\tH^4\)) is an algebra.
\begin{lemma}\label{lemma:algebra} If \(f, g \in \H^4\) then
\[
    \|fg\|_{\H^4} \lesssim \|f\|_{\H^4}\|g\|_{\H^4}.
\]
\end{lemma}
\begin{proof}
    We first see that
    \[
    \int_{0}^{L}f(\theta)^2g(\theta)^2\d\theta \leq \|g\|_{L^\infty}^2\|f\|_{\H^4}^2.
    \]
    Moreover, for \(0 \leq j \leq 3\) we have
    \begin{align*}
    \int_{0}^{L} f^{(j)}(\theta)^2g^{(4-j)}(\theta)^2(L-\theta)^{13/2}\d\theta &\leq C\|f^{(j)}(\theta)(L-\theta)^{j}\|_{L^\infty}^2\int_{0}^{L}g^{(4-j)}(\theta)^2(L-\theta)^{13/2 - 2j} \\
    &\leq C\|f\|_{\H^4}^2\|g\|_{\H^4}^2
    \end{align*}
    where we have used Lemmas \ref{lemma:infty_embed} and \ref{lemma:weighted_interpolation} in the final inequality. The lemma now follows from Leibniz rule and reversing the roles of \(f, g\). 
\end{proof}
\subsection{Local Well-Posedness of the 1D System}
This section is devoted to proving the following local well-posedness of equation \eqref{eq:ipm_1d} in \(\tH^4\).
\begin{proposition}\label{prop:LWP_X}
    Let \(M_0 \in \tH^4\). Then, there exists \(T = T(\|M_0\|_{\tH^4}) > 0\)  such that there is a unique solution \((M, G) \in C([0, T], \tH^4)\) to the system \eqref{eq:ipm_1d}.
\end{proposition}
We first note the following lemma bounding \((\p_\theta^2 + 4)^{-1}\).
\begin{lemma}\label{lemma:G_bound_lwp}
    Suppose \(G\) solves
    \[
    G'' + 4G = M'(\theta)\cos^2\theta, \quad G(0) = G(L) = 0.
    \]
    where we recall that \(L < \pi/2\) is fixed. Then, for all \(1 \leq j \leq 5\) and \(p > 0\),
    \[
    \int_{0}^{L}G^{(j)}(\theta)^2(L-\theta)^p\d\theta \lesssim \|M(\theta)\|_{L^1}^2 + \int_{0}^{L}M^{(j-1)}(\theta)^2(L-\theta)^{p}\d\theta, \quad \int_{0}^{L}G(\theta)^2(L-\theta)^p\d\theta \lesssim \|M(\theta)\|_{L^1}^2.
    \]
\end{lemma}
\begin{proof}
    We note that \(G\) can be recovered explicitly as
    \begin{align*}
        G(\theta) &= \frac{1}{2}\sin2\theta\left[-\int_{\theta}^{L}M'(\phi)\cos^2(\phi)\cos(2\phi)\d\phi + \cot(2L)\int_{0}^{L}M'(\phi)\cos^2(\phi)\sin(2\phi)\d\phi\right] \\
        &\qquad - \frac{1}{2}\cos(2\theta)\int_{0}^{\theta}M'(\phi)\cos^2(\phi)\sin(2\phi)\d\phi.
    \end{align*}
    Integrating by parts gives
    \begin{align*}
        G(\theta) &= \frac{1}{2}\sin(2\theta)\left[\int_{\theta}^{L}M(\phi)[\cos^2(\phi)\cos(2\phi)]'\d\phi - \cot(2L)\int_{0}^{L}M(\phi)[\cos^2(\phi)\sin(2\phi)]'\d\phi\right] \\
        &\qquad + \frac{1}{2}\cos(2\theta)\int_{0}^{\theta}M(\phi)[\cos^2(\phi)\sin(2\phi)]'\d\phi.
    \end{align*}
    Taking \(j\) derivatives we note that the only nonlocal terms appear as integrals of \(M\) against bounded kernels and therefore it is readily seen that we have the pointwise bound
    \begin{align*}
        |G^{(j)}(\theta)| &\lesssim \|M\|_{L^1} + \sum_{k = 0}^{j-1}|M^{(k)}(\theta)|
    \end{align*}
    from which the lemma immediately follows.
\end{proof}
We now proceed with the proof of Proposition \ref{prop:LWP_X}.
\begin{proof}
Due to the equivalence of norms from Lemma \ref{lemma:equiv_norm}, it suffices to provide a priori estimates using \(\|\cdot\|_{\H^4}\). Let \((M, G)\) be a smooth solution to \eqref{eq:ipm_1d}. Then,
\begin{align*}
    \frac12\frac{d}{dt}\|M\|_{\H^4}^2 = \int_{0}^{L}M(\theta)\p_t M(\theta) + M^{(4)}(\theta)\p_t M^{(4)}(\theta) (L-\theta)^{13/2}\d\theta.
\end{align*}
Dealing first with the low-norm integral,
\begin{align*}
\int_{0}^{L}M(\theta)\p_tM(\theta)\d\theta &= \int_{0}^{L} M(\theta)(-2G(\theta)M'(\theta) + G'(\theta) M(\theta) + 2M(\theta)G(\theta)\tan\theta)\d\theta \\
&\lesssim \int_{0}^{L}(|G'(\theta)|+|G(\theta)|)M(\theta)^2\d\theta \lesssim \|M\|_{L^\infty}\int_{0}^{L}M(\theta)^2\d\theta.
\end{align*}
For the high norm integral, differentiating the equation \eqref{eq:ipm_1d} for \(M\) four times, it follows using H\"older's inequality that
\begin{align*}
\int_{0}^{L}M^{(4)}(\theta)\p_t M^{(4)}(\theta)(L-\theta)^{13/2}\d\theta \lesssim \|M\|_{\H^4}\left(\sum_{\substack{k+j \leq 5 \\ k \leq 4}}\int_{0}^{L}G^{(j)}(\theta)M^{(k)}(\theta)(L-\theta)^{13/2}\d\theta\right)^{1/2}.
\end{align*}
For \(j \leq 1\), by Lemma \ref{lemma:G_bound_lwp} we have
\begin{align*}
    \int_{0}^{L}G^{(j)}(\theta)M^{(k)}(\theta)(L-\theta)^{13/2}\d\theta \lesssim \|M\|_{L^\infty}^2 + \int_{0}^{L}M^{(k)}(\theta)^2(L-\theta)^{13/2},
\end{align*}
while for \(j \geq 2\) using H\"older's inequality, 
\begin{align*}
\int_{0}^{L}G^{(j)}(\theta)M^{(k)}(\theta)(L-\theta)^{13/2}\d\theta \lesssim  \int_{0}^{L}G^{(j)}(\theta)^2(L-\theta)^{9/4 - k + j} + M^{(k)}(\theta)^2(L-\theta)^{17/4 + k - j}\d\theta.
\end{align*}
By Lemma \ref{lemma:G_bound_lwp} for \(j \geq 2\) we have
\begin{align*}
    \int_{0}^{L}G^{(j)}(\theta)^2(L-\theta)^{9/4-k+j} \d\theta \lesssim \sum_{i = 0}^{j-1}\int_{0}^{L}M^{(i)}(\theta)^2(L-\theta)^{9/4-k+j}\d\theta + \|M\|_{L^1}^2.
\end{align*}
By the interpolation inequality of Lemma \ref{lemma:weighted_interpolation} it follows that
\begin{align*}
    \int_{0}^{L}M^{(i)}(\theta)^2(L-\theta)^{9/4-k+j}\d\theta \lesssim \int_{0}^{L}M^{(4)}(\theta)^2(L-\theta)^{9/4-k+j+2(4-i)}\d\theta + \int_{0}^{L}M(\theta)^2\d\theta
\end{align*}
and similarly,
\begin{align*}
    \int_{0}^{L}M^{(k)}(\theta)^2(L-\theta)^{17/4+k-j}\d\theta \lesssim \int_{0}^{L}M^{(4)}(\theta)^2(L-\theta)^{17/4+k-j+2(4-k)}\d\theta + \int_{0}^{L}M(\theta)^2\d\theta.
\end{align*}
Since \(9/4-k+j+2(4-i) \geq 29/4\) and \(17/4+k-j+2(4-k) \geq 29/4\) we find that
\begin{align*}
\int_{0}^{L}M^{(4)}(\theta)\p_t M^{(4)}(\theta)(L-\theta)^{13/2}\d\theta &\lesssim \|M\|_{\H^4}\left[\int_{0}^{L}M(\theta)^2 + M^{(4)}(\theta)^2(L-\theta)^{29/2}\d\theta \right] + \|M\|_{\H^4}^3,
\end{align*}
and since \(29/4 > 13/2\), we conclude
\[
\frac{d}{dt}\|M\|_{\H^4}^2 \lesssim \|M\|_{\H^4}^3.
\]
Given the above a priori estimate, one can use a standard iteration scheme to prove existence of a solution \(M \in C([0, T]; \tH^4)\). 
\end{proof}

\section{Low-Norm Coercivity}
 In this section we prove the following proposition regarding the coercivity of \(\Lb\) for the low order portion of the inner product \(\inner{\cdot}{\cdot}_{\tH^4}\).
\begin{proposition}\label{prop:low_norm_coercive}
    There exists \(c_* > 0\) such that for all \(M \in \tH^4\),
    \[
    \int_{0}^{L}\tM(\theta)\Lb(M)(\theta)\varphi(\theta) \d\theta \geq c_1 \int_{0}^{L}\tM(\theta)^2\varphi(\theta) \d\theta
    \]
    where \(\varphi\) is the weight defined in \eqref{eq:weights}.
\end{proposition}
Integrating by parts, we note that the boundary terms vanish since \(G_*(0) = G_*(L) = 0\) and hence
\begin{align*}
\int_{0}^{L}\tM(\theta)\Lb_{L}(M)(\theta)\varphi(\theta)\d\theta &= \int_{0}^{L} \tM(\theta)[\tM(\theta) + 2G_*(\theta)\tM'(\theta) - G_*'(\theta)\tM(\theta) - 2\tM(\theta)G_*(\theta)\tan\theta]\varphi(\theta)\d\theta \\
&= \int_{0}^{L}\tM(\theta)^2\left[1 - 2G_*'(\theta) - \frac{\varphi'(\theta)}{\varphi(\theta)}G_*(\theta) - 2G_*(\theta)\tan\theta\right]\varphi(\theta)\d\theta.
\end{align*}
We now require separate estimates in the three regions, locally near \(\theta = 0\), in the bulk and near the endpoint \(\theta = L\).
\subsection{Local Estimates}
This section is devoted to proving the following proposition regarding the coercivity of \(\Lb\) at low order on a small interval \([0, \l_1]\). In this region, the coercivity comes from the local transport term due to the strong, singular weight. There will be a negative contribution from the stretching term, however the singular weight of \(\theta^{-8}\) serves to amplify the effect of the transport term, in order to overcome the stretching term. The nonlocal terms will provide contributions which can be made small by choosing \(\l_1\) sufficiently small. 
\begin{proposition}\label{prop:local_low}
	There exists \(\l_1^* > 0\) such that for all \(\l_1 < \l_1^*\)  
	\begin{align*}
	&\int_{0}^{\l_1}\tM(\theta)^2\left[1 - 2G_*'(\theta) - \frac{\varphi'(\theta)}{\varphi(\theta)}G_*(\theta) - 2G_*(\theta)\tan\theta\right]\theta^{-8}\d\theta + \int_{0}^{\l_1} \theta^{-8} \tM(\theta)\Lb_{NL}(M)\d\theta \\
    &\qquad \geq \left[3 - C\l_1^2\right]\int_{0}^{\l_1} \theta^{-8}\tM(\theta)^2\d\theta.
	\end{align*}
\end{proposition}

\begin{proof}
        Dealing with the local part first, for \(0 \leq \theta \leq \l_1\), \(\varphi'(\theta)/\varphi(\theta) = - 8\theta^{-1}\) and therefore
	\begin{multline*}
			\int_{0}^{\l_1}\tM(\theta)^2\left[1 - 2G_*'(\theta) - \frac{\varphi'(\theta)}{\varphi(\theta)}G_*(\theta) - 2G_*(\theta)\tan\theta\right]\theta^{-8}\d\theta \\
             = \int_{0}^{\l_1}\theta^{-8}\tM(\theta)^2\left[1 - 2G_*'(\theta) + 8 \theta^{-1}G_*(\theta) - 2G_*(\theta)\tan\theta\right]\d\theta.
	\end{multline*}
	Since \(G_*'(0) = 1\) it follows that
	\[
	1 - 2G_*' + 8 \theta^{-1}G_* - 2G_*\tan\theta \geq 7 - C\l_1^2.
	\]
        To deal with the nonlocal piece \(\Lb_{NL}\), we must first prove a local bound on the operator \(\tM \mapsto \tG\) in \(\tH^k\).
        \begin{lemma}\label{lemma:G_bound_low_loc}
            Suppose \(\tG\) solves \eqref{eq:G_loc_eqn}, then 
	\begin{equation*}
		\int_{0}^{\l_1}\tG(\theta)^2\theta^{-8}\d\theta \leq C\l_1^2\int_{0}^{\l_1}\tM(\theta)^2\theta^{-8}\d\theta, \quad \text{and} \quad \int_{0}^{\l_1}\tG'(\theta)^2 \theta^{-8}\d\theta \leq (1+C\l_1^2)\int_{0}^{\l_1}\tM(\theta)^2\theta^{-8}\d\theta
	\end{equation*}
    for some universal constant \(C > 0\).
     \end{lemma}
    \begin{proof}
    First, we note that \eqref{eq:G_loc_eqn} admits an explicit solution. Multiplying \eqref{eq:G_loc_eqn} by \(\sin2\theta\) and \(\cos2\theta\) respectively, and integrating by parts we find
\begin{align*}
\int_{0}^{\theta}\tM'(\phi)\cos^2\phi\sin2\phi\d\phi &=  \tG_{loc}'(\theta)\sin2\theta - 2\tG_{loc}(\theta)\cos 2\theta
\end{align*}
and similarly,
\begin{align*}
\int_{0}^{\theta}\tM'(\phi)\cos^2\phi\cos2\phi\d\phi &=  \tG_{loc}'(\theta)\cos2\theta + 2\tG_{loc}(\theta)\sin 2\theta.
\end{align*}
Multiplying the equations by \(-\cos2\theta\) and  \(\sin2\theta\) respectively and summing gives an explicit formula for \(\tG_{loc}\),
\begin{equation}
	\tG_{loc}(\theta) = -\frac{1}{2}\cos2\theta\int_{0}^{\theta}M'(\phi)\cos^2\phi\sin2\phi\d\phi + \frac{1}{2}\sin2\theta\int_{0}^{\theta}M'(\phi)\cos^2\phi\cos2\phi\d\phi.\label{eq:G_loc_form_1}
\end{equation}
Integrating the above by parts and noting that \(\tM(0) = 0\) we have the following alternative explicit solution
\begin{align}
	\tG_{loc}(\theta) &= \frac12 \cos2\theta\int_{0}^{\theta}\tM(\phi)\left[2\cos^2\phi\cos2\phi - \sin^22\phi\right]\d\phi \notag \\
	& \qquad + \frac12 \sin2\theta\int_{0}^{\theta}\tM(\phi)\left[\sin2\phi\cos2\phi + 2\cos^2\phi\sin2\phi\right]\d\phi \notag \\
    &=: \frac12 \cos2\theta\int_{0}^{\theta}\tM(\phi)K_1(\phi)\d\phi + \frac12 \sin2\theta\int_{0}^{\theta}\tM(\phi)K_2(\phi)\d\phi,
    \label{eq:G_loc_form_2}
\end{align}
where \(K_1, K_2\) are the two integral kernels defined by the preceding line. Differentiating, we obtain the following explicit formula for \(\tG_{loc}'\),
\begin{align}\label{eq:G_loc'_form}
    \tG_{loc}'(\theta) &= \tM(\theta)\cos^2\theta - \sin 2\theta \int_{0}^{\theta}\tM(\phi)K_1(\phi)\d\phi + \cos2\theta\int_{0}^{\theta}\tM(\phi)K_2(\phi)\d\phi.
\end{align}
Squaring and integrating \eqref{eq:G_loc_form_2} against the weight \(\theta^{-8}\),
\begin{align*}
    \int_{0}^{\l_1}\tG(\theta)^2\theta^{-8}\d\theta &\leq \int_{0}^{\l_1}\theta^{-8}\left[\frac12 \cos^22\theta\left(\int_{0}^{\theta}\tM(\phi)K_1(\phi)\d\phi\right)^2 + \frac12\sin^22\theta\left(\int_{0}^{\theta}\tM(\phi)K_2(\phi)\d\phi\right)^2\right]\d\theta \\
    &\leq \frac12 \int_{0}^{\l_1}\theta^{-8}\left[\theta \int_{0}^{\theta}\tM(\phi)^2(K_1(\phi)^2+K_2(\phi)^2)\d\phi\right]\d\theta.
\end{align*}
Since \(K_1(\theta)^2 + K_2(\theta)^2 = 4\cos^4(\theta)+\sin^2(2\theta) \leq 4\), we then conclude
\begin{align*}
        \int_{0}^{\l_1}\tG(\theta)^2\theta^{-8}\d\theta &\leq 2\int_{0}^{\l_1}\theta\int_{0}^{\l_1}\tM(\phi)^2\phi^{-8}\d\phi\d\theta \leq \l_1^2\int_{0}^{\l_1}\tM(\phi)^2\phi^{-8}\d\phi
\end{align*}
which proves the first statement of the lemma. Using the explicit formula \eqref{eq:G_loc'_form}, 
\begin{align*}
    \int_{0}^{\l_1}\tG_{loc}'(\theta)^2\theta^{-8}\d\theta &= \int_{0}^{\l_1}\theta^{-8}\left[\sin^2 2\theta \left(\int_{0}^{\theta}\tM(\phi)K_1(\phi)\d\phi\right)^2 + \cos^22\theta\left(\int_{0}^{\theta}\tM(\phi)K_2(\phi)\d\phi\right)^2\right]\d\theta \\
    &\quad + \int_{0}^{\l_1}\theta^{-8}\cos^4(\theta)\tM(\theta)^2\d\theta\\
    &\quad + 2 \int_{0}^{\l_1}\theta^{-8}\tM(\theta)\cos^2(\theta)\left[ \cos2\theta\int_{0}^{\theta}\tM(\phi)K_2(\phi)\d\phi-\sin2\theta\int_{0}^{\theta}\tM(\phi)K_1(\phi)\d\phi \right]\d\theta \\
    &\quad - 2\int_{0}^{\l_1}\theta^{-8}\sin2\theta\cos2\theta \left[\int_{0}^{\theta}\tM(\phi)K_1(\phi)\d\phi\int_{0}^{\theta}\tM(\phi)K_2(\phi)\d\phi\right]\d\theta.
\end{align*}
The second term satisfies
\[
\int_{0}^{\l_1}\theta^{-8}\cos^4(\theta)\tM(\theta)^2\d\theta \leq \int_{0}^{\l_1}\theta^{-8}\tM(\theta)^2\d\theta.
\]
Performing similar computations to the case of \(\tG_{loc}\), it is readily seen that there exists a universal constant \(C > 0\) such that the remaining terms are bound above by
\[
C \l_1^2 \int_{0}^{\l_1}\theta^{-8}\tM(\theta)^2\d\theta
\]
from which the second statement of the lemma immediately follows.
\end{proof}
Considering the nonlocal contribution \(\Lb_{NL}\), since \(M_*(\theta) \leq 4\), by Lemma \ref{lemma:G_bound_low_loc} we obtain
\[
	\left|\int_{0}^{\l_1}\tM(\theta) M_*(\theta)\tG'(\theta) \theta^{-8}\d\theta \right|  \leq 	4\int_{0}^{\l_1}|\tM(\theta)| |\tG'(\theta)| \theta^{-8}\d\theta  \leq 4(1+C\l_1^2) \int_{0}^{\l_1}\theta^{-8}\tM(\theta)^2\d\theta.
\]
The remaining terms provide only a small (in \(\l_1\)) contribution since \(M_*\tan\theta + M_*'\) vanishes linearly at zero and thus by Lemma \ref{lemma:G_bound_low_loc},
	\begin{align*}
	\left|\int_{0}^{\l_1}2\tM(\theta)\tG(\theta) (M_*'(\theta) + M_*(\theta)\tan\theta)\theta^{-8}\d\theta\right| &\leq C \l_1 \left(\int_{0}^{\l_1}\tM(\theta)^2 \theta^{-8}\d\theta\right)^{1/2}\left(\int_{0}^{\l_1}\tG(\theta)^2 \theta^{-8}\d\theta\right)^{1/2} \\
    &\leq C\l_1^{2}\int_{0}^{\l_1}\tM(\theta)^2 \theta^{-8}\d\theta.
	\end{align*}
	Thus, we have
	\begin{multline*}
	\int_{0}^{\l_1}\tM(\theta)^2\left[1 - 2G_*'(\theta) - \frac{\varphi'(\theta)}{\varphi(\theta)}G_*(\theta) - 2G_*(\theta)\tan\theta\right]\varphi(\theta)\d\theta + \int_{0}^{\l_1} \theta^{-8} \tM(\theta)\Lb_{NL}(M)\d\theta \\
     \geq \left[7  - 4 - C\l_1^2\right]\int_{0}^{\l_1}\tM(\theta)^2 \theta^{-8}\d\theta
	\end{multline*}
	which completes the proof.
\end{proof}
\subsection{Bulk Estimates}
Now, we prove coercivity in the bulk. In this region, we use the rapid decay of the weight \(\varphi\) obtained by choosing \(K\) large. Integrating by parts the local transport term gives a term
\[
-\int_{\l_1}^{\l_2}\left(\frac{\varphi'(\theta)}{\varphi(\theta)}\right)G_*(\theta)\tM(\theta)^2\varphi(\theta)\d\theta = K \int_{\l_1}^{\l_2}(\theta^{-1}G_*(\theta))\tM(\theta)^2\varphi(\theta)\d\theta
\]
Since \(\theta^{-1}G_*\) is bounded away from zero away from the boundary \(\theta = L\), this gives us a large coercive term. The remaining terms then need only provide bounded contributions and then choosing \(K\) sufficiently large will give coercivity. The bound however degenerates near the right hand boundary as \(G_*\) vanishes here necessitating the need to consider a boundary region \([\l_2, L]\) separately. The nonlocal terms from \(\Lb_{NL}\) will generate large terms in the local region \([0, \l_1]\) due to the high degree of vanishing of the weight \(\varphi(\theta) = \theta^{-K}\), however these terms have been controlled in the previous section. Consequently, in the bulk region we will prove the following coercivity lemma.
\begin{proposition}\label{prop:bulk_low}
There exists \(c, C > 0\) such that for all \(M \in \tH^4\),
\begin{multline*}
    \int_{\l_1}^{\l_2}\tM(\theta)^2\left[1 - 2G_*'(\theta) - \frac{\varphi'(\theta)}{\varphi(\theta)}G_*(\theta) - 2G_*(\theta)\tan\theta\right]\theta^{-K}\d\theta + \int_{\l_1}^{\l_2}\tM(\theta)\Lb_{NL}(M)\theta^{-K}\d\theta \\
    \geq  [Kc(L-\l_2) - C(L-\l_2)^{-1/2}]\int_{\l_1}^{\l_2}\tM(\theta)\theta^{-K}\d\theta - C(L-\l_2)^{-1/2}\l_1^{10-K}\int_{0}^{\l_1}\tM(\theta)^2\theta^{-8}\d\theta. 
 \end{multline*}
\end{proposition}
\begin{proof}
Again, we begin by considering the local terms. For \(\theta \in [\l_1, \l_2]\), \(\varphi'(\theta)/\varphi(\theta) = -K\theta^{-1}\) and therefore
\begin{multline*}
	    \int_{\l_1}^{\l_2}\tM(\theta)^2\left[1 - 2G_*'(\theta) - \frac{\varphi'(\theta)}{\varphi(\theta)}G_*(\theta) - 2G_*(\theta)\tan\theta\right]\theta^{-K}\d\theta \\
        = \int_{\l_1}^{\l_2}\theta^{-K}\tM(\theta)^2\left[1 + K\theta^{-1}G_*(\theta) - 2G_*'(\theta) - 2G_*(\theta)\tan\theta\right]\d\theta.
\end{multline*}
	Since \(G_*'(0) = 1\) and \(G_*\) vanishes linearly at \(\theta = L\),
	\[
	1 + K\theta^{-1}G_*(\theta) - 2G_*'(\theta) - 2G_*(\theta)\tan\theta \geq  Kc(L-\l_2) - C
	\]
	for constants \(c, C\) depending on the profile. Thus, by choosing \(K\) large (relative to \(L - \l_2\)) we can obtain a strong coercive term. To control the nonlocal terms, we prove the following lemma.

\begin{lemma}\label{lemma:G_bound_low_bulk}
There exists \(C > 0\) such that for any  \(K > 10\),
\begin{align}
\int_{\l_1}^{\l_2}(\tG'(\theta)^2+\tG(\theta)^2)\theta^{-K}\d\theta &\leq C\l_1^{10-K}\int_{0}^{\l_1}\tM(\theta)^2 \theta^{-8}\d\theta + C\int_{\l_1}^{\l_2}\tM(\theta)^2 \theta^{-K}\d\theta.
\end{align}
\end{lemma}
\begin{proof}
 Note that for \(\theta \in [\l_1, \l_2]\), \(\tG(\theta) = \tG_{loc}(\theta)\). Using the formula \eqref{eq:G_loc_form_2}, and recalling that \(K_1, K_2\) are uniformly bounded,
\begin{align}
    \int_{\l_1}^{\l_2}\tG(\theta)^2\theta^{-K}\d\theta &\leq \frac12 \int_{\l_1}^{\l_2}\theta^{-K}\left(\int_{0}^{\theta}\tM(\phi)K_1(\phi)\d\phi\right)^2 + \theta^{-K}\left(\int_{0}^{\theta}\tM(\phi)K_1(\phi)\d\phi\right)^2 \d\theta \\
    &\leq C \int_{\l_1}^{\l_2}\theta^{-K}\theta\int_{0}^{\theta}\tM(\phi)^2\d\phi\d\theta.
\end{align}
Splitting the inner integral into our local and bulk regions,
\begin{align*}
    \int_{\l_1}^{\l_2}\tG(\theta)^2\theta^{-K}\d\theta &\leq C\int_{\l_1}^{\l_2}\theta^{-K}\theta\int_{0}^{\l_1}\tM(\phi)^2\d\phi + \theta^{-K}\theta\int_{\l_1}^{\theta}\tM(\phi)^2\d\phi\d\theta \\
    &\leq C\int_{\l_1}^{\l_2}\theta^{-K}\theta^9\int_{0}^{\l_1}\tM(\phi)^2\phi^{-8}\d\phi + \theta\int_{\l_1}^{\theta}\tM(\phi)^2\phi^{-K}\d\phi\d\theta \\
    &\leq \frac{C}{K-10}\l_1^{10-K}\int_{0}^{\l_1}\tM(\phi)^2\phi^{-8}\d\phi + C\int_{\l_1}^{\l_2}\tM(\phi)^2\phi^{-K}\d\phi.
\end{align*}
Proceeding similarly for \(\tG'\) using \eqref{eq:G_loc'_form},
\begin{align*}
    \int_{\l_1}^{\l_2}\tG'(\theta)^2\theta^{-K}\d\theta &\leq C\int_{\l_1}^{\l_2}\left[\tM(\theta)^2 + \left(\int_{0}^{\theta}\tM(\phi)K_1(\phi)\d\phi\right)^2 + \left(\int_{0}^{\theta}\tM(\phi)K_2(\phi)\d\phi\right)^2\right]\theta^{-K}\d\theta \\
    &\leq C\int_{\l_1}^{\l_2}\tM(\theta)^2 \theta^{-K} + \theta^{-K+1} \int_{0}^{\theta} \tM(\phi)^2(K_1(\phi)^2 + K_2(\phi)^2)\d\phi \d\theta.
\end{align*}
Again splitting the inner integral and using boundedness of the kernels \(K_1, K_2\) we conclude
\begin{align*}
     \int_{\l_1}^{\l_2}\tG(\theta)^2\theta^{-K}\d\theta &\leq C\int_{\l_1}^{\l_2}\tM(\theta)^2\theta^{-K} + \theta^{-K+1}\int_{0}^{\l_1}\tM(\phi)^2\d\phi + \theta^{-K+1}\int_{\l_1}^{\theta}\tM(\phi)^2\d\phi \d\theta \\
     &\leq C\int_{\l_1}^{\l_2}\tM(\theta)^2\theta^{-K} + \theta^{-K+9}\int_{0}^{\l_1}\tM(\phi)^2\phi^{-8}\d\phi +\theta\int_{\l_1}^{\l_2}\tM(\phi)^2\phi^{-K}\d\phi \d\theta \\
     &\leq C\int_{\l_1}^{\l_2}\tM(\theta)^2\theta^{-K}\d\theta + \frac{C}{K-10}\l_1^{10-K}\int_{0}^{\l_1}\tM(\phi)^2\phi^{-8}\d\phi .
\end{align*}
Summing the estimates for \(\tG, \tG'\) then gives the lemma.
\end{proof}    
Turning our attention to the terms generated by \(\Lb_{NL}\) we see first that since \(M_* \in C^{1/2}([0, L])\), we have \(|M_{*}'(\theta)| \leq C(L-\l_2)^{-1/2}\) and hence
\begin{align*}
    2\left|\int_{\l_1}^{\l_2} \theta^{-K}\tM(\theta) \tG(\theta) M_*'(\theta)\d\theta \right| &\leq C(L-\l_2)^{-1/2}\int_{\l_1}^{\l_2} (\tM(\theta)^2 + \tG(\theta)^2) \theta^{-K}\d\theta.
 \end{align*}
Similarly, since \(M_* \leq 4\) is bounded,
    \[
    \left|\int_{\l_1}^{\l_2}\theta^{-K}(\tM(\theta) \tG'(\theta) M_*(\theta) + \tM(\theta)\tG(\theta) M_*(\theta)\tan\theta)\d\theta \right| \leq C\int_{\l_1}^{\l_2}(\tM(\theta)^2 + \tG(\theta)^2 + \tG'(\theta)^2)\theta^{-K}\d\theta.
    \]
By Lemma \ref{lemma:G_bound_low_bulk}, we conclude
\[
\left|\int_{\l_1}^{\l_2}\tM(\theta)\Lb_{NL}(M)\theta^{-K}\d\theta \right| \leq  C(L-\l_2)^{-1/2}\left[\int_{\l_1}^{\l_2}\tM(\theta)^2\theta^{-K}\d\theta + \l_1^{10-K}\int_{0}^{\l_1}\tM(\theta)^2\theta^{-8}\d\theta\right].
\]
Altogether, we then have
\begin{multline*}
    \int_{\l_1}^{\l_2}\tM(\theta)^2\left[1 - 2G_*'(\theta) - \frac{\varphi'(\theta)}{\varphi(\theta)}G_*(\theta) - 2G_*(\theta)\tan\theta\right]\theta^{-K}\d\theta + \int_{\l_1}^{\l_2} \theta^{-K} \tM(\theta)\Lb_{NL}(M)\d\theta \\
     \geq  [Kc(L-\l_2) - C(L-\l_2)^{-1/2}]\int_{\l_1}^{\l_2}\tM(\theta)^2\theta^{-K}\d\theta - C(L-\l_2)^{-1/2}\l_1^{10-K}\int_{0}^{\l_1}\tM(\theta)^2\theta^{-8}\d\theta,
\end{multline*}
completing the proof of Proposition \ref{prop:bulk_low}.
\end{proof}
\subsection{Endpoint Estimates}
Finally, we establish coercivity estimates near the boundary \(\theta = L\). In this region, the coercivity can no longer come from the transport as particles flow towards the boundary rather than away. However, \(G_*\) vanishes at the boundary and thus any contribution of transport is necessarily small. However, near the boundary, the density is being depleted by the stretching term as \(G_*' \leq 0\) here. Thus, near the boundary the coercivity comes from the stretching term rather than the transport. We now prove the following proposition regarding the coercivity of \(\Lb\) near the boundary.
\begin{proposition}\label{prop:endpoint_low}
There exists \(C > 0\) such that for all \(M \in \tH^4\),
    \begin{multline}
    \int_{\l_2}^{L}\tM(\theta)^2\left[1 - 2G_*'(\theta)- 2G_*(\theta)\tan\theta\right]\d\theta + \int_{\l_2}^{L}\tM(\theta)\Lb_{NL}(M)\d\theta \\ \geq \left(1 - C(L-\l_2)^{1/2}\right)\int_{\l_2}^{L}\tM(\theta)^2\d\theta \notag 
     - C(L-\l_2)^{-1/2}\left[\l_1^8\int_{0}^{\l_1}\tM(\theta)^2\theta^{-8}\d\theta + \l_2^K\int_{\l_1}^{\l_2}\tM(\theta)^2\theta^{-K}\d\theta\right].
    \end{multline}
\end{proposition}
\noindent Note that in this region, \(\varphi'(\theta) \equiv 0\) and thus the \((\varphi'/\varphi)G_*\) term is absent.
\begin{proof}
Considering the local terms first, since \(G_*(\theta) \geq 0\) with \(G_*(L) = 0\), for \(\l_2\) sufficiently close to \(L\), \(G_*'(\theta) < 0\) for all \(\theta \in [\l_2, L]\) and moreover \(G_*\tan\theta \leq C(L-\l_2)\). Thus,
\[
\int_{\l_2}^{L}\tM(\theta)^2\left[1 - 2G_*'(\theta)- 2G_*(\theta)\tan\theta\right]\d\theta \geq (1 - C(L-\l_2))\int_{\l_2}^{L}\tM(\theta)^2\d\theta.
\]
Now, we examine the terms from \(\Lb_{NL}\) which require some additional care as \(\tG_{loc}(\theta) \neq \tG(\theta)\) for \(\theta \in [\l_2, L]\).
\begin{lemma}\label{lemma:G_bound_low_endpoint}
    There exists a constant \(C > 0\) such that
\begin{align*}
    \int_{\l_2}^{L}\tG(\theta)^2(L - \theta)^{-1}\d\theta &\leq C(L-\l_2)\int_{\l_2}^{L}\tM(\theta)^2\d\theta + C\int_{0}^{\l_2}\tM(\theta)^2\d\theta
\end{align*}
and
\begin{align*}
    \int_{\l_2}^{L}\tG'(\theta)^2\d\theta &\leq C\int_{\l_2}^{L}\tM(\theta)^2\d\theta + C(L-\l_2)^{-1}\int_{0}^{\l_2}\tM(\theta)^2\d\theta.
\end{align*}
\end{lemma}
\begin{proof}
Writing,
\[
\tG(\theta) =  (\tG_{loc}(\theta) - \tG_{loc}(L)) + \tG_{loc}(L)\left(1 - \frac{\cos2\theta}{\cos 2L}\right) + \frac{\tG_{loc}(L)}{\cos 2L}\cos2\theta\left(1-\eta\left(\frac{L-\theta}{L-\l_2}\right)\right)
\]
we estimate the three terms separately. First,
\begin{align*}
    \tG_{loc}(\theta) - \tG_{loc}(L) &= \frac12(\cos2\theta - \cos2L)\int_{0}^{\theta}\tM(\phi)\d\phi - \frac{1}{2}\cos2L\int_{\theta}^{L}\tM(\phi)K_2(\phi)\d\phi \\
    &\qquad - \frac{1}{2}(\sin2\theta - \sin 2L)\int_{0}^{\theta}\tM(\phi)K_2(\phi)\d\phi + \frac{1}{2}\sin(2L)\int_{\theta}^{L}\tM(\phi)K_2(\phi)\d\phi.
\end{align*}
Using the fact that \(|\cos2\theta - \cos2L| + |\sin2\theta - \sin 2L| \leq C|L-\theta|\) it follows that
\begin{align*}
    \int_{\l_2}^{L}(\tG_{loc}(\theta) - \tG_{loc}(L))^2(L-\theta)^{-1}\d\theta &\leq C\int_{\l_2}^{L} (L-\theta)\int_{0}^{\theta}\tM(\phi)^2\d\phi\d\theta + C\int_{\l_2}^{L}\int_{\theta}^{L}\tM(\phi)^2\d\phi\d\theta \\
    &\leq C(L-\l_2)^2\int_{0}^{L}\tM(\phi)^2\d\phi + C(L-\l_2)\int_{\l_2}^{L}\tM(\phi)^2\d\phi.
\end{align*}
Integrating the remaining two terms, it is easily seen from the vanishing of \(1 - \cos2\theta/\cos2L\) and \(1 - \eta\) at \(\theta = L\) that
\begin{equation*}
    \int_{\l_2}^{L} \left[\tG_{loc}(L)^2\left(1 - \frac{\cos2\theta}{\cos 2L}\right)^2 + \frac{\tG_{loc}(L)^2}{\cos^2 2L}\cos^22\theta\left(1-\eta\left(\frac{L-\theta}{L-\l_2}\right)\right)^2\right](L-\theta)^{-1} \d\theta \leq C\tG_{loc}(L)^2.
\end{equation*}
Finally,
\begin{equation*}\label{eq:GL_bound}
    \tG_{loc}(L)^2 \leq C\int_{0}^{\l_2}\tM(\theta)^2\d\theta + C(L - \l_2)\int_{\l_2}^{L}\tM(\theta)^2\d\theta,
\end{equation*}
which completes the proof of the first inequality. Looking now at \(\tG'\) we see
\begin{equation*}
    \int_{\l_2}^{L}\tG'(\theta)^2\d\theta \leq C \int_{\l_2}^{L}\tG_{loc}'(\theta)^2\d\theta + C\tG_{loc}(L)^2\int_{\l_2}^{L}(L-\l_2)^{-2}\left(\eta'\left(\frac{L-\theta}{L-\l_2}\right)^2 + \eta\left(\frac{L-\theta}{L-\l_2}\right)^2\right)\d\theta
\end{equation*}
Since \(|\eta| \leq 1\) and \(|\eta'| \leq 3\) it follows from \eqref{eq:GL_bound} that
\begin{align*}
    \int_{\l_2}^{L}\tG'(\theta)^2\d\theta &\leq C \int_{\l_2}^{L}\tG_{loc}'(\theta)^2\d\theta + C\tG_{loc}(L)^2(L-\l_2)^{-1} \\
    &\leq C \int_{\l_2}^{L}\tG_{loc}'(\theta)^2\d\theta + C(L-\l_2)^{-1}\int_{0}^{\l_2}\tM(\theta)^2\d\theta + C\int_{\l_2}^{L}\tM(\theta)^2\d\theta.
\end{align*}
Finally looking at the first term, from the explicit formula \eqref{eq:G_loc'_form} it follows that
\begin{align*}
    \int_{\l_2}^{L}\tG_{loc}'(\theta)^2\d\theta \leq C\int_{0}^{L}\tM(\theta)^2\d\theta
\end{align*}
from which the second inequality now follows.
\end{proof}
Now, we look at the terms generated by \(\Lb_{NL}\),
\begin{align*}
    \int_{\l_2}^{L}\tM(\theta) \Lb_{NL}(M)\d\theta &= \int_{\l_2}^{L}2\tG(\theta) \tM(\theta) M_*'(\theta) - M_*(\theta)\tM(\theta)(\tG'(\theta) - 2\tG(\theta)\tan\theta)\d\theta.
\end{align*}
Since \(M_* \in C^{\frac12}([0, L])\), there exists a constant \(C > 0\) such that \(M_*'(\theta)(L-\l_2)^{1/2} \leq C\). Then,
\begin{align*}
\left|\int_{\l_2}^{L}2\tG(\theta) \tM(\theta) M_*'(\theta)\d\theta\right| &\leq C\int_{\l_2}^{L}|\tG(\theta)\tM(\theta)|(L-\theta)^{-1/2}\d\theta \\
&\leq C\left(\int_{\l_2}^{L}\tG(\theta)^2(L-\theta)^{-1}\d\theta\right)^{1/2}\left( \int_{\l_2}^{L}\tM(\theta)^2\d\theta\right)^{1/2}.
\end{align*}
By Lemma \ref{lemma:G_bound_low_endpoint} we then have
\begin{align*}
    \left|\int_{\l_2}^{L}2\tG(\theta)\tM(\theta) M_*'(\theta)\d\theta\right| &\leq C\left((L-\l_2)\int_{\l_2}^{L}\tM(\theta)^2\d\theta + C\int_{0}^{\l_2}\tM(\theta)^2\d\theta\right)^{1/2}\left( \int_{\l_2}^{L}\tM(\theta)^2\d\theta\right)^{1/2} \\
    &\leq C(L-\l_2)^{1/2}\int_{\l_2}^{L}\tM(\theta)^2\d\theta + C(L-\l_2)^{-1/2}\int_{0}^{\l_2}\tM(\theta)^2\d\theta.
\end{align*}
In the final term, since \(M_*(L) = 0\) and \(M_* \in C^{\frac12}([0, L])\) we have \(M_*(\theta) \leq C(L-\theta)^{1/2}\) and hence
\begin{align*}
    \left|\int_{\l_2}^{L}M_*(\theta) \tM(\theta)(\tG'(\theta) - 2\tG(\theta)\tan\theta)\d\theta\right| &\leq C(L-\l_2)^{1/2}\int_{\l_2}^{L}|\tM(\theta)||\tG' - 2\tG(\theta)\tan\theta|\d\theta \\
    &\leq C(L-\l_2)^{1/2}\int_{\l_2}^{L}\tM(\theta)^2 + \tG'(\theta)^2 + \tG(\theta)^2\d\theta.
\end{align*}
Applying Lemma \ref{lemma:G_bound_low_endpoint} (note that \((L-\theta) < 1\) for \(\theta \in [\l_2, L]\) so the inequality only improves without the singular weight) we then obtain
\begin{align*}
    \left|\int_{\l_2}^{L}M_*(\theta) \tM(\theta)(\tG'(\theta) - 2\tG(\theta)\tan\theta)\d\theta\right| &\leq C(L-\l_2)^{1/2}\int_{\l_2}^{L}\tM(\theta)^2\d\theta + C(L-\l_2)^{-1/2}\int_{0}^{\l_2}\tM(\theta)^2\d\theta.
\end{align*}
Altogether, we then have
\begin{multline*}
    \int_{\l_2}^{L}\tM(\theta)^2\left[1 - 2G_*'(\theta)- 2G_*(\theta)\tan\theta\right]\d\theta + \int_{\l_2}^{L}\tM(\theta)\Lb_{NL}(M)\d\theta \\
     \geq \left(1- C(L-\l_2)^{1/2}\right)\int_{\l_2}^{L}\tM(\theta)^2\d\theta - C(L-\l_2)^{-1/2}\int_{0}^{\l_2}\tM(\theta)^2\d\theta.
\end{multline*}
Finally, since
\[
\int_{0}^{\l_2}\tM(\theta)\d\theta \leq \l_1^8\int_{0}^{\l_1}\tM(\theta)^2\theta^{-8}\d\theta + \l_2^K\int_{\l_1}^{\l_2}\tM(\theta)^2\theta^{-K}\d\theta,
\]
Proposition \ref{prop:endpoint_low} now follows. 
\end{proof}
\subsection{Coercivity}
 Now, we prove Proposition \ref{prop:low_norm_coercive} to obtain coercivity at low order by summing the inequalities in Propositions \ref{prop:local_low}, \ref{prop:bulk_low}, and \ref{prop:endpoint_low} after multiplying each by \(1, \l_1^{K-8}\) and \(\l_1^{K-8}\l_2^{-K}\) respectively. Indeed, considering first all terms integrated on \([0, \l_1]\) we obtain
\begin{multline*}
    \left\{3 - C\l_1^2 - C\l_1^{K-8}(L-\l_2)^{-1/2}\l_1^{10-K} - C\l_1^{K-8}\l_2^{-K}(L-\l_2)^{-1/2}\l_1^8\right\}\int_{0}^{\l_1}\tM(\theta)^2\theta^{-8}\d\theta  \\
    =  \left\{3 - C\l_1^2 - C\l_1^{2}(L-\l_2)^{-1/2} - C\left(\frac{\l_1}{\l_2}\right)^{K}(L-\l_2)^{-1/2}\right\}\int_{0}^{\l_1}\tM(\theta)^2\theta^{-8}\d\theta.
\end{multline*}
Choosing \(L - \l_2 = \l_1^2\) and \(K = \l_1^{-4}\) gives
\begin{align*}
     &\left\{3 - C\l_1^2 - C\l_1 - C\l_1^{-1}\left(\frac{\l_1}{\l_2}\right)^{\l_1^{-4}}\right\}\int_{0}^{\l_1}\tM(\theta)^2\theta^{-8}\d\theta \geq 2\int_{0}^{\l_1}\tM(\theta)^2\theta^{-8}\d\theta
\end{align*}
for \(\l_1\) sufficiently small. For the bulk terms, we obtain
\begin{multline*}
    \left\{\l_1^{K-8}Kc(L-\l_2) - C\l_1^{K-8}(L-\l_2)^{-1/2} - C\l_1^{K-8}\l_2^{-K}(L-\l_2)^{-1/2}\l_2^K\right\}\int_{\l_1}^{\l_2}\tM(\theta)\theta^{-K}\d\theta \\
    = \l_1^{K-8}\left\{Kc\l_1^{2} - C\l_1^{-1}\right\}\int_{\l_1}^{\l_2}\tM(\theta)^2\theta^{-K}\d\theta
\end{multline*}
where we have recalled that \(L - \l_2 = \l_1^2\). Recalling \(K = \l_1^{-4}\) we find
\[
\l_1^{K-10}(c - C\l_1) \geq \frac{c_*\l_1^{K-10}}{2}
\]
for \(\l_1\) sufficiently small. Finally, for the endpoint terms we have
\begin{align*}
    \l_1^{K-8}\l_2^{-K}\left(1 - C(L-\l_2)^{1/2}\right)\int_{\l_2}^{L}\tM(\theta)^2\d\theta \geq \frac{\l_1^{K-8}\l_2^{-K}}{2}\int_{\l_2}^{L}\tM(\theta)^2\d\theta 
\end{align*}
for \(\l_1\) sufficiently small. Thus, we conclude that
\begin{align*}
    \int_{0}^{L}\tM(\theta)\Lb(M)\varphi(\theta)\d\theta \geq \min\left\{2, \frac{c\l_1^{K-10}}{2}, \frac{\l_1^{K-8}\l_2^{-K}}{2}\right\}\int_{0}^{L}\tM(\theta)\varphi(\theta)\d\theta
\end{align*}
which proves Proposition \ref{prop:low_norm_coercive}. 
\section{High-Norm Coercivity}
 In this section we prove the following proposition regarding the coercivity of \(\Lb\) for the fourth order portion of the inner product \(\inner{\cdot}{\cdot}_{\tH^4}\). The proof is largely analogous to the low norm case.
\begin{proposition}\label{prop:high_norm_coercive}
    There exists \(c_2, C > 0\) such that for all \(M \in \tH^4\),
    \[
    \int_{0}^{L}\tM^{(4)}(\theta)\Lb(M)^{(4)}(\theta)\psi(\theta) \d\theta \geq c_2 \int_{0}^{L}\tM^{(4)}(\theta)^2\psi(\theta) \d\theta - C\int_{0}^{L}\tM(\theta)^2\varphi(\theta)\d\theta.
    \]
    Here, the constants \(c_2, C\) depend on the profile and \(\l_1\). 
\end{proposition}
Isolating the top order terms of the local part of the operator \(\Lb_{L}\)
\begin{align*}
\int_{0}^{L}\tM^{(4)}(\theta)\Lb_{L}(M)^{(4)}(\theta)\psi(\theta)\d\theta &= \int_{0}^{L}\tM^{(4)}(\theta)\left[\tM^{(4)}(\theta) + 2G_*(\theta)\tM^{(5)}(\theta) + 7G_*'(\theta)\tM^{(4)}(\theta)\right]\psi(\theta)\d\theta \\
& \qquad - \int_{0}^{L}\tM^{(4)}(\theta)\left[2G_*(\theta)\tM^{(4)}(\theta)\tan\theta\right]\psi(\theta)\d\theta + \int_{0}^{L}I_{L}(\tM)\psi(\theta)\d\theta
\end{align*}
where \(I_L(\tM)\) are the lower order terms from \(\Lb_L\)
\[
I_{L}(\tM) := \tM^{(4)}\left[\sum_{k=0}^{2}\binom{4}{k}2G_*^{(4-k)}\tM^{(1+k)}- \sum_{k=0}^{3}\binom{4}{k}\left(G_*^{(5-k)}\tM^{(k)} -2\tM^{(k)}(G_*\tan\theta)^{(4-k)}\right)\right].
\]
Integrating by parts, and noting that the boundary terms vanish since \(G_*(0) = G_*(L) = 0\), the top order part becomes
\begin{align}\label{eq:top_order_local}
\int_{0}^{L} \tM^{(4)}(\theta)^2\left[1+ 6G_*'(\theta) - \frac{\psi'(\theta)}{\psi(\theta)}G_*(\theta) - 2G_*\tan\theta\right]\psi(\theta)\d\theta.
\end{align}
As in the low-norm case, we treat the three regions \([0, \l_1], [\l_1, \l_2], [\l_2, L]\) separately. First, to deal with the lower order terms, for convenience we recall a basic interpolation inequality,
\begin{lemma}\label{lemma:interp}
Let \(k \geq 1\) be an integer. Then, there exists a constant \(C = C(k) > 0\) such that for every \(f \in \H^k, j < k\) and \(\epsilon > 0\)
\begin{equation}
    \|f^{(j)}\|_{L^2}^2 \leq \frac{C_k}{\epsilon^2}\|f\|_{L^2} + C_k\epsilon^2\|f^{(k)}\|_{L^2}.
\end{equation}
\end{lemma}

\subsection{Local Coercivity}
Here, we prove that the following proposition about the coercivity of \(\Lb\) on \([0, \l_1]\). We will obtain coercivity up to a low norm term which can be overcome by choosing \(B\) sufficiently large in the definition of \(\inner{\cdot}{\cdot}_{\tH^4}\) \eqref{eq:inner_product_1}.

\begin{proposition}\label{prop:local_high}
	There exists \(C, \l_1^* > 0\) such that for all \(\l_1 < \l_1^*\), we have
	\begin{multline*}
	\int_{0}^{\l_1} \tM^{(4)}(\theta)^2\left[1+ 6G_*'(\theta) - 2G_*(\theta)\tan\theta\right]\d\theta + \int_{0}^{\l_1}I_{L}(\tM)\d\theta + \int_{0}^{\l_1}\tM^{(4)}(\theta)\Lb_{NL}(M)\d\theta \\
    \geq \left(3 - C\l_1^2\right) \int_{0}^{\l_1}\tM^{(4)}(\theta)^2\d\theta - C\l_1^2\int_{0}^{\l_1}\theta^{-8}\tM(\theta)^2\d\theta.
	\end{multline*}
\end{proposition}
\begin{proof}
    Note that \(\psi \equiv 1\) so \(\psi' \equiv 0\) on \([0, \l_1]\) so the first integral consists of exactly the top order terms identified in \eqref{eq:top_order_local}. As in the low norm case, the coercivity comes from the highest order derivatives in the transport term and the remaining lower order terms can be interpolated, costing only lower order norm which has been controlled in the previous section. As in the low norm case, for \(\l_1\) sufficiently small, since \(G_*'(0) = 1\) and \(G_*(0) = 0\) we have
    \[
    1 + 6G_*' - 2G_*\tan\theta \geq 7 - C\l_1^2.
    \]
    The remaining lower order terms in \(I_L(\tM)\) can be controlled through interpolation. To consider the terms generated by \(\Lb_{NL}\) we first prove the following high norm analog of Lemma \ref{lemma:G_bound_low_loc}.
    \begin{lemma}\label{lemma:G_high_loc}
	For any \(1 \leq k  \leq 5\) there exists \(C_k > 0\) such that for all \(\delta > 0\), 
	\begin{align*}
	\int_{0}^{\l_1}\tG^{(k+1)}(\theta)^2 \d\theta \leq (1+\delta^2)\int_{0}^{\l_1}\tM^{(k)}(\theta)^2\d\theta + C_k\delta^{-2}\sum_{j=0}^{k-1}\int_{0}^{\l_1}\tM^{(j)}(\theta)^2\d\theta.
	\end{align*}
    \label{G_loc_lemma}
    \end{lemma}
    \begin{proof}
	Differentiating \eqref{eq:G_loc_form_2} \(k+1\) times, the highest derivative term in \(\tM\) appearing is \(\tM^{(k)}\) which appears with the coefficient
	\[
		\frac{1}{2}\left(\cos2\theta K_1(\theta) + \sin2\theta K_2(\theta)\right) = \cos^2\theta.
	\]
	Thus, it is readily seen that we have the pointwise bound
	\begin{equation}\label{eq:G_pointwise}
		|\tG^{(k+1)}(\theta)| \leq |\tM^{(k)}(\theta)| + C_k\sum_{j=1}^{k-1}|\tM^{(j)}(\theta)| + C_k\int_{0}^{\theta}|\tM(\phi)|\d\phi
	\end{equation}
	for some constant \(C_k\) depending only on \(k\). Hence, for any \(\delta > 0\),
		\[
		\tG^{(k+1)}(\theta)^2 \leq(1+\delta^2) \tM^{(k)}(\theta)^2 + C_k\delta^{-2}\left(\sum_{j=1}^{k-1}\tM^{(j)}(\theta)^2 + \left(\int_{0}^{\theta}|\tM(\phi)|\d\phi\right)^2\right).
	\]
	Integrating the above and applying Cauchy--Schwarz on the integral term completes the proof of the lemma.
\end{proof}
\begin{cor}\label{cor:G5_est} In the case \(k = 5\), choosing \(\delta = \l_1\), we obtain the following local, high-norm estimate from the preceding lemma and interpolation (Lemma \ref{lemma:interp}). 
	\begin{align}
		\int_{0}^{\l_1}\tG^{(5)}(\theta)^2 \d\theta &\leq (1+\l_1^2)\int_{0}^{\l_1}\tM^{(4)}(\theta)^2\d\theta + C\l_1^{2}\int_{0}^{\l_1}\tM(\theta)^2 \theta^{-8}\d\theta
    \end{align}
\end{cor}
Considering now the terms generated by \(\Lb_{NL}\), there is only one top order term,
    \[
       \int_{0}^{\l_1}\tM^{(4)}(\theta)\Lb_{NL}(M)^{(4)}\d\theta = - \int_{0}^{\l_1}M_*(\theta)\tG^{(5)}(\theta)\tM^{(4)}(\theta)\d\theta + \int_{0}^{\l_1}I_{NL}(\tM)\d\theta
    \]
    where \(I_{NL}\) denotes the lower order terms,
    \[
        I_{NL}(\tM) := \tM^{(4)}\left[2\sum_{k=0}^{4}\binom{4}{k}\tG^{(4)}\left[M_*^{(5-k)} - (M_*\tan\theta)^{(4-k)}\right]- \sum_{k=0}^{3}\binom{4}{k}\tG^{(k+1)}M_*^{(4-k)}\right].
    \]
	Since \(M_*(\theta) \leq 4\), it follows that
	\begin{align*}
	\left|\int_{0}^{\l_1}M_*(\theta)\tG^{(5)}(\theta)\tM^{(4)}(\theta)\d\theta\right| &\leq  2 \int_{0}^{\l_1}\tG^{(5)}(\theta)^2 + \tM^{(4)}(\theta)^2 \d\theta \\
    &\leq (4+C\l_1^2)\int_{0}^{\l_1}\tM^{(4)}(\theta)^2\d\theta + C\l_1^2\int_{0}^{\l_1}\tM(\theta)^2\theta^{-8}\d\theta,
	\end{align*}
    where we have used Corollary \ref{cor:G5_est} in the final inequality. It now remains to control the lower order terms \(I_{L}, I_{NL}\) by interpolation. Using H\"older's inequality,
    \begin{align*}
        \int_{0}^{\l_1}|(I_L + I_{NL})(\tM)|\d\theta \leq C\l_1^2\int_{0}^{\l_1}\tM^{(4)}(\theta)^2\d\theta + \frac{C}{\l_1^2}\left[\sum_{j=0}^{3}\int_{0}^{\l_1}\tM^{(j)}(\theta)^2\d\theta + \sum_{j=0}^{4}\int_{0}^{\l_1}\tG^{(j)}(\theta)^2\d\theta\right]
    \end{align*}
    for some constant \(C = C(\|M_*\|_{C^5([0, \l_1])}, \|G_*\|_{C^5([0, \l_1])})\). Then, by Lemma \ref{lemma:G_high_loc} (with \(\delta = 1\)) and interpolation (Lemma \ref{lemma:interp}), 
    \begin{align*}
        \int_{0}^{\l_1}|(I_L + I_{NL})(\tM)|\d\theta &\leq C \l_1^2\int_{0}^{\l_1}\tM^{(4)}(\theta)^2\d\theta + \frac{C}{\l_1^2}\sum_{j=0}^{3}\int_{0}^{\l_1}\tM^{(j)}(\theta)^2\d\theta \\
        &\leq C \l_1^2\int_{0}^{\l_1}\tM^{(4)}(\theta)^2\d\theta + C\l_1^2\int_{0}^{\l_1}\tM(\theta)^2\theta^{-8}\d\theta.
    \end{align*}
    Altogether, we then have that
    \begin{multline*}
        \int_{0}^{\l_1} \tM^{(4)}(\theta)^2\left[1+ 6G_*' - 2G_*\tan\theta\right]\psi(\theta)\d\theta + \int_{0}^{\l_1}I_{L}(\tM)\d\theta + \int_{0}^{\l_1}\Lb_{NL}(M)\d\theta  \\
        \geq \left(7 - 4 - C\l_1^2\right)\int_{0}^{\l_1}\tM^{(4)}(\theta)^2\d\theta - C\l_1^2\int_{0}^{\l_1}\tM(\theta)^2\theta^{-8}\d\theta,
    \end{multline*}
    completing the proof.
    \end{proof}

\subsection{Bulk Coercivity}
Now, we prove coercivity in the bulk region \([\l_1, \l_2]\) up to lower order terms which have already been controlled.
\begin{proposition}\label{prop:bulk_high}
	There exists \(C  > 0\) such that for any \(K > 10\),
	\begin{align*}
		&\int_{\l_1}^{\l_2} \tM^{(4)}(\theta)^2\left[1+ 6G_*'(\theta) - \frac{\psi'(\theta)}{\psi(\theta)}G_*(\theta) - 2G_*(\theta)\tan\theta\right]\theta^{-K}\d\theta + \int_{\l_1}^{\l_2}I_{L}(\tM)\theta^{-K}\d\theta  \\
        &\hspace{23em} + \int_{\l_1}^{\l_2}\tM^{(4)}(\theta)\Lb_{NL}(M)\theta^{-K}\d\theta  \\
        &\geq [Kc(L-\l_2) - C]\int_{\l_1}^{\l_2}\tM^{(4)}(\theta)^2 \theta^{-K}\d\theta - C_{\l_1}\left[\int_{\l_1}^{\l_2}\tM(\theta)^2\theta^{-K}\d\theta + \int_{0}^{\l_1}\tM(\theta)^2\theta^{-8}\d\theta\right]
	\end{align*}
\end{proposition}
\begin{proof}
    In the region \([\l_1, \l_2]\), we have \(\psi'(\theta)/\psi(\theta) = -K\theta^{-1}\). Since \(G_* \geq 0\) and vanishes linearly at \(\theta = 0, L\) (see \eqref{eq:G_vanishing}), we have \(\theta^{-1}G_*(\theta) \geq KC(L-\l_2)\) for all \(\theta \in [0, \l_2]\)  and \(G_*, G_*'\) are bounded, it follows that
        \[
        1 + 6G_*' - 2G_*\tan\theta -\frac{\psi'(\theta)}{\psi(\theta)}G_* \geq K(L-\l_2) - C
        \]
        for some constant \(C > 0\) depending on the profile. Therefore,
        \begin{align*}
           &\int_{\l_1}^{\l_2} \tM^{(4)}(\theta)^2\left[1+ 6G_*'(\theta) - \frac{\psi'(\theta)}{\psi(\theta)}G_*(\theta) - 2G_*(\theta)\tan\theta\right]\theta^{-K}\d\theta 
           \geq [K(L-\l_2)-C]\int_{\l_1}^{\l_2} \tM^{(4)}(\theta)^2\theta^{-K}\d\theta.
        \end{align*}
        To deal with the nonlocal terms, we prove the following analog of Lemma \ref{lemma:G_bound_low_bulk}.
    \begin{lemma}\label{lemma:G_bound_high_bulk}
    For any \(0 \leq k  \leq 4\), there exists constants \(C_k, C_{k, \l_1} > 0\) such that
        \[
        \int_{\l_1}^{\l_2}\tG^{(k+1)}(\theta)^2\theta^{-K}\d\theta \leq C_k\int_{\l_1}^{\l_2}\tM^{(k)}(\theta)^2\theta^{-K}\d\theta + C_{k, \l_1}\int_{\l_1}^{\l_2}\tM(\theta)^2\theta^{-K}\d\theta + C_{k, \l_1}\int_{0}^{\l_1}\tM(\theta)^2\theta^{-8}\d\theta.
        \]
        Moreover,
        \[
        \int_{\l_1}^{\l_2}\tG(\theta)^2\theta^{-K}\d\theta \leq C\int_{\l_1}^{\l_2}\tM(\theta)^2\theta^{-K}\d\theta + C\l_1^{-K+10}\int_{0}^{\l_1}\tM(\theta)^2\theta^{-8}\d\theta.
        \]
    \end{lemma}
    \begin{proof}
        Integrating the pointwise bound \eqref{eq:G_pointwise} gives
        \begin{align*}
            \int_{\l_1}^{\l_2}\tG^{(k+1)}(\theta)^2\theta^{-K}\d\theta &\leq C\sum_{j=0}^{k} \int_{\l_1}^{\l_2}\tM^{(j)}(\theta)^2\theta^{-K}\d\theta + C\int_{\l_1}^{\l_2}\left(\int_{0}^{\theta}\tM(\phi)\d\phi\right)^2\theta^{-K}\d\theta \\
            &\hspace{-7em} \leq C\sum_{j=0}^{k} \int_{\l_1}^{\l_2}\tM^{(j)}(\theta)^2\theta^{-K}\d\theta + C\int_{\l_1}^{\l_2}\theta^{-K+9}\int_{0}^{\l_1}\tM(\phi)^2\phi^{-8}\d\phi + \theta\int_{\l_1}^{\l_2}\tM(\phi)^2\phi^{-K}\d\phi\d\theta \\
            &\hspace{-7em} \leq C\sum_{j=0}^{k} \int_{\l_1}^{\l_2}\tM^{(j)}(\theta)^2\theta^{-K}\d\theta + C\l_1^{-K+10}\int_{0}^{\l_1}\tM(\phi)^2\phi^{-8}\d\phi + C\int_{\l_1}^{\l_2}\tM(\phi)^2\phi^{-K}\d\phi
        \end{align*}
        where we have split the inner integral into the local and bulk range and applied Cauchy--Schwarz in the penultimate inequality. The first inequality in the lemma then follows by interpolation. The second inequality follows similarly.
    \end{proof}
    Considering now the terms generated by \(\Lb_{NL}\), there is only one top order term generated, and the remaining terms can be handled by interpolation. 
        \begin{align*}
            \int_{\l_1}^{\l_2}\tM^{(4)}(\theta)\Lb_{NL}(M)\theta^{-K}\d\theta &= \int_{\l_1}^{\l_2}I_{NL}(\tM)\theta^{-K}\d\theta - \int_{\l_1}^{\l_2}\tG^{(5)}(\theta)\tM^{(4)}(\theta)M_*(\theta)\theta^{-K}\d\theta
        \end{align*}
        Since \(M_*(\theta) \leq 4\), by Cauchy--Schwarz,
        \begin{align*}
        \left|\int_{\l_1}^{\l_2}M_*(\theta)\tG^{(5)}(\theta)\tM^{(4)}(\theta)\theta^{-K}\d\theta\right| &\leq 2 \int_{\l_1}^{\l_2}\tG^{(5)}(\theta)^2\theta^{-K}\d\theta +2\int_{\l_1}^{\l_2}\tM^{(4)}(\theta)^2\theta^{-K}\d\theta \\
        &\leq C\int_{\l_1}^{\l_2}\tM^{(4)}(\theta)^2\theta^{-K}\d\theta + C_{\l_1}\left[\int_{\l_1}^{\l_2}\tM(\theta)^2\theta^{-K}\d\theta + \int_{0}^{\l_1}\tM(\theta)^2\theta^{-8}\d\theta\right]
        \end{align*}
        where we have used Lemma \ref{lemma:G_bound_high_bulk} in the final inequality. Using H\"older's inequality and that \(G_* \in C^{\infty}([0, L)) \cap C^{1, \frac12}([0, L])\) we have
        \begin{align*}
             \left|\int_{\l_1}^{\l_2}I_{L}(\tM)\theta^{-K}\d\theta\right| &\leq C\int_{\l_1}^{\l_2}\tM^{(4)}(\theta)^2\theta^{-K}\d\theta + C(L-\l_2)^{-7}\sum_{k=0}^{3}\int_{\l_1}^{\l_2}\tM^{(k)}(\theta)^2\theta^{-K}\d\theta \\
            &\leq  C\int_{\l_1}^{\l_2}\tM^{(4)}(\theta)^2\theta^{-K}\d\theta + C(L-\l_2)^{-7}\l_1^{-K}\sum_{k=0}^{3}\int_{\l_1}^{\l_2}\tM^{(k)}(\theta)^2\d\theta.
        \end{align*}
       By the basic interpolation Lemma \ref{lemma:interp}, it follows that for \(\delta > 0\),
        \begin{align*}
             \left|\int_{\l_1}^{\l_2}I_{L}(\tM)\theta^{-K}\d\theta\right| &\leq C\int_{\l_1}^{\l_2}\tM^{(4)}(\theta)^2\theta^{-K}\d\theta \\
             &\qquad + C(L-\l_2)^{-7}\l_1^{-K}\left[\delta^{-2}\int_{\l_1}^{\l_2}\tM(\theta)^2\d\theta + \delta^{2}\int_{\l_1}^{\l_2}\tM^{(4)}(\theta)^2\d\theta\right],
        \end{align*}
        and choosing \(\delta^2 = (L-\l_2)^{7}\l_2^{-K}\l_1^{K}\) gives
        \begin{align*}
            \left|\int_{\l_1}^{\l_2}I_{L}(\tM)\theta^{-K}\d\theta\right| &\leq C\int_{\l_1}^{\l_2}\tM^{(4)}(\theta)^2\theta^{-K} + \l_2^{-K}\tM^{(4)}(\theta)^2\d\theta + C(L-\l_2)^{-14}\l_1^{-2K}\l_2^{K}\int_{\l_1}^{\l_2}\tM(\theta)^2\d\theta \\
            &\leq C\int_{\l_1}^{\l_2}\tM^{(4)}(\theta)^2\theta^{-K}\d\theta + C(L-\l_2)^{-14}\left(\frac{\l_2}{\l_1}\right)^{2K}\int_{\l_1}^{\l_2}\tM(\theta)^2\theta^{-K}\d\theta.
        \end{align*}
        Proceeding similarly for \(I_{NL}\), we find
        \begin{align*}
            \left| \int_{\l_1}^{\l_2}I_{NL}(\tM)\theta^{-K}\d\theta\right| &\leq C\int_{\l_1}^{\l_2}\tM^{(4)}(\theta)^2\theta^{-K}\d\theta + C(L-\l_2)^{-9}\sum_{k=0}^{4}\int_{\l_1}^{\l_2}\tG^{(k)}(\theta)^2\theta^{-K}\d\theta \\
            &\leq C\int_{\l_1}^{\l_2}\tM^{(4)}(\theta)^2\theta^{-K}\d\theta  + C_{\l_1}\left[\int_{\l_1}^{\l_2}\tM(\theta)^2\theta^{-K} \d\theta+ \int_{0}^{\l_1}\tM(\theta)^2\theta^{-8}\d\theta\right] 
        \end{align*}
        where we have applied Lemma \ref{lemma:G_bound_high_bulk} in the final inequality. This completes the proof of Proposition \ref{prop:bulk_high}. 
\end{proof}

\subsection{Endpoint Coercivity}
In this region, the key is to exploit the fact that we are working in a space with lower regularity near the boundary than the profile. Indeed, since \(M_* \in C^{\frac12}([0, L]) \cap C^{\infty}([0, L))\), we have \(M_*^{(4)}(\theta) \leq C(L-\theta)^{-7/2}\) and thus
\[
\int_{\l_2}^{L}M_*^{(4)}(\theta)^2(L-\theta)^{13/2}\d\theta \leq C\int_{\l_2}^{L}(L-\theta)^{-7+13/2}\d\theta \leq C(L-\l_2)^{1/2}
\]
and so we see that the profile is less singular than the weight \((L-\theta)^{13/2}\) allows for. This fact allows us to gain necessary small factors to prove coercivity near the boundary.
\begin{proposition}\label{prop:endpoint_high_coercivity}
There exists \(C_1 > 0\) depending only on the profile and \(C_2 = C_2(\l_1) > 0\) such that for all \(M \in \tH^4\),
    \begin{align*}
    &\int_{\l_2}^{L} \tM^{(4)}(\theta)^2\left[1+ 6G_*'(\theta) - \frac{\psi'(\theta)}{\psi(\theta)}G_*(\theta) - 2G_*(\theta)\tan\theta\right](L-\theta)^{13/2}\d\theta + \int_{\l_2}^{L}I_{L}(\tM)(L-\theta)^{13/2}\d\theta \\
    &\qquad+ \int_{\l_2}^{L}\tM^{(4)}(\theta)\Lb_{NL}(M)(L-\theta)^{13/2}\d\theta \\
    &\geq [1 - C_1(L-\l_2)^{1/2}]\int_{\l_2}^{L}\tM^{(4)}(\theta)^2(L-\theta)^{13/2}\d\theta - C_2\int_{0}^{L}\tM(\theta)^2\d\theta.
    \end{align*}
\end{proposition}
\begin{proof}
In the region \([\l_2, L]\) we have \(\psi'/\psi = -\frac{13}{2}(L-\theta)^{-1}\). Since \(G_*(L) = 0\), and 
\[
\lim_{\theta \to L}(L-\theta)^{-1}G_*(\theta) = -G_*'(L) > 0,
\]
by choosing \(\l_2\) sufficiently small we may ensure that
\[
6G_*'(\theta) + \frac{13}{2}(L-\theta)^{-1}G_*(\theta) - G_*(\theta)\tan\theta \geq 0
\]
for all \(\theta \in [\l_2, L]\). To estimate the nonlocal terms, we require the following lemma which is the high norm analog of Lemma \ref{lemma:G_bound_low_endpoint}
\begin{lemma}\label{lemma:G_bound_high_endpoint} For all \(0 \leq k \leq 4\), there exists \(C, p > 0\) such that
    \[
    \int_{\l_2}^{L}\tG^{(k+1)}(\theta)^2(L-\theta)^{13/2}\d\theta \leq C\sum_{j=0}^{k}\int_{\l_2}^{L}\tM^{(j)}(\theta)^2(L-\theta)^{13/2}\d\theta + C(L-\l_2)^{-p}\int_{0}^{L}\tM(\theta)^2\d\theta.
    \]
\end{lemma}
\begin{proof}
    First, we observe that from the definition of \(\tG\) it follows that
    \begin{align*}
    \int_{\l_2}^{L} \tG^{(k+1)}(\theta)^2(L-\theta)^{13/2}\d\theta &\leq C \int_{\l_2}^{L}\tG_{loc}^{(k+1)}(\theta)^2(L-\theta)^{13/2} + \frac{1}{(L-\l_2)^{2k+2}}\tG_{loc}(L)^2(L-\theta)^{13/2}\d\theta \\
    &\leq C \int_{\l_2}^{L}\tG_{loc}^{(k+1)}(\theta)^2(L-\theta)^{13/2}\d\theta + C(L-\l_2)^{11/2 - 2k}\tG_{loc}(L)^2
    \end{align*}
    where we have used that \(\frac{d^k}{d\theta^k}\eta((L-\theta)/(L-\l_2)) \leq C(L-\l_2)^{-k}\). Using the pointwise bound \eqref{eq:G_pointwise}, we see the first term satisfies
    \begin{align*}
        \int_{\l_2}^{L}\tG_{loc}^{(k+1)}(\theta)^2(L-\theta)^{13/2}\d\theta &\leq C \sum_{j=0}^{k}\int_{\l_2}^{L}\tM^{(j)}(\theta)^2(L-\theta)^{13/2} + (L-\theta)^{13/2}\int_{0}^{\theta}\tM(\phi)^2\d\phi\d\theta \\
        &\leq  C \sum_{j=0}^{k}\int_{\l_2}^{L}\tM^{(j)}(\theta)^2(L-\theta)^{13/2}\d\theta + C\int_{0}^{L}\tM(\phi)^2\d\phi.
    \end{align*}
    Finally, since
    \begin{align*}
        \tG_{loc}(L)^2 \leq C \int_{0}^{L}\tM(\theta)^2\d\theta
    \end{align*}
    the lemma follows.
\end{proof}

Considering now the terms generated by \(\Lb_{NL}\)
\begin{align*}
    \int_{\l_2}^{L}\tM^{(4)}(\theta)\Lb_{NL}(M)^{(4)} &= \int_{\l_2}^{L}I_{NL}(\tM)(L-\theta)^{13/2}\d\theta-\int_{\l_2}^{L}\tM^{(4)}(\theta)\tG^{(5)}(\theta)M_*(\theta)(L-\theta)^{13/2}\d\theta
\end{align*}
we see the top order term provides only a small contribution due to the vanishing of \(M_*\) near the boundary. Indeed, since \(M_* \in C^{1/2}([0, L])\) vanishes at \(\theta = L\),
\begin{multline*}
    \int_{\l_2}^{L}|\tM^{(4)}(\theta)| |\tG^{(5)}(\theta)||M_*(\theta)|(L-\theta)^{13/2}\d\theta \leq C(L-\l_2)^{1/2}\int_{\l_2}^{L}|\tM^{(4)}(\theta)||\tG^{(5)}(\theta)|(L-\theta)^{13/2}\d\theta \\
    \leq C(L-\l_2)^{1/2}\left[\sum_{j=0}^{4}\int_{\l_2}^{L}\tM^{(j)}(\theta)^2(L-\theta)^{13/2}\d\theta + C(L-\l_2)^{-p}\int_{0}^{L}\tM(\theta)^2\d\theta\right]
\end{multline*}
where we have used Lemma \ref{lemma:G_bound_high_endpoint}. It now remains to control the lower order terms \(I_{L}, I_{NL}\). The terms in \(I_L\) generated by derivatives of the transport and stretching terms \(2G_*\tM', G_*'\tM\) are all of the form
\[
\int_{\l_2}^{L}\tM^{(4)}(\theta)\tM^{(k)}(\theta)G_*^{(5-k)}(\theta)(L-\theta)^{13/2}\d\theta
\]
for \(0 \leq k \leq 3\). Since \(G_* \in C^{1, \frac12}([0, L])\) it follows that \((L-\theta)^{7/2-k}G_{*}^{(5-k)} \in L^{\infty}([0, L])\) for all \(0 \leq k \leq 3\) and thus by H\"older's inequality, for all \(\delta > 0\),
\begin{multline*}
\int_{\l_2}^{L}\tM^{(4)}(\theta)\tM^{(k)}(\theta)G_*^{(5-k)}(\theta)(L-\theta)^{13/2}\d\theta \\
\leq C\delta \int_{\l_2}^{L}\tM^{(4)}(\theta)^2(L-\theta)^{13/2}\d\theta + C\delta^{-1}\int_{\l_2}^{L}\tM^{(k)}(\theta)^2(L-\theta)^{-7+2k}(L-\theta)^{13/2}\d\theta.
\end{multline*}
Now applying Lemma \ref{lemma:weighted_interpolation}, it follows that
\begin{align*}
\int_{\l_2}^{L}\tM^{(k)}(\theta)^2(L-\theta)^{-7+2k}(L-\theta)^{13/2}\d\theta &\leq C\int_{\l_2}^{L}\tM^{(4)}(\theta)^2(L-\theta)(L-\theta)^{13/2}\d\theta + C\int_{\l_2}^{L}\tM(\theta)^2\d\theta\\
&\leq C(L-\l_2)\int_{\l_2}^{L}\tM^{(4)}(\theta)^2\d\theta + C\int_{\l_2}^{L}\tM(\theta)^2\d\theta.
\end{align*}
Choosing \(\delta = (L-\l_2)^{1/2}\), altogether we have
\[
\int_{\l_2}^{L}\tM^{(4)}\tM^{(k)}G_*^{(5-k)}(L-\theta)^{13/2}\d\theta \leq C(L-\l_2)^{1/2}\int_{\l_2}^{L}\tM^{(4)}(\theta)^2(L-\theta)^{13/2}\d\theta + C(L-\l_2)^{-1/2}\int_{\l_2}^{L}\tM(\theta)^2\d\theta
\]
The other terms in \(I_L\) coming from \(2G_*\tM\tan\theta\) are even less singular and can be dealt with similarly. Proceeding similarly for \(I_{NL}\) we find that,
\begin{align*}
   \int_{\l_2}^{L} I_{NL}(\tM)(L-\theta)^{13/2}\d\theta &\leq C(L-\l_2)^{1/2}\int_{\l_2}^{L}\tM^{(4)}(\theta) \tG^{(5)}(\theta)(L-\theta)^{13/2}\d\theta \\
   &\qquad + (L-\l_2)^{-1/2}\int_{\l_2}^{L}\tG(\theta)^2\d\theta.
\end{align*}
Applying Lemma \ref{lemma:G_bound_high_endpoint} gives
\begin{align*}
    \int_{\l_2}^{L}I_{NL}(\tM)(L-\theta)^{13/2}\d\theta \leq C(L-\l_2)^{1/2}\sum_{j=0}^{4}\int_{\l_2}^{L}\tM^{(j)}(\theta)^2(L-\theta)^{13/2}\d\theta + C(L-\l_2)^{-p}\int_{0}^{L}\tM(\theta)^2\d\theta.
\end{align*}
Altogether, we have
\begin{align*}
    &\int_{\l_2}^{L} \tM^{(4)}(\theta)^2\left[1+ 6G_*'(\theta) - \frac{\psi'(\theta)}{\psi(\theta)}G_*(\theta) - 2G_*(\theta)\tan\theta\right](L-\theta)^{13/2}\d\theta + \int_{\l_2}^{L}I_{L}(\tM)(L-\theta)^{13/2}\d\theta \\
    &\qquad  + \int_{\l_2}^{L}\tM^{(4)}(\theta)\Lb_{NL}(M)(L-\theta)^{13/2}\d\theta \\
    & \geq \left[1 - C(L-\l_2)^{1/2}\right]\int_{\l_2}^{L}\tM^{(4)}(\theta)^2(L-\theta)^{13/2}\d\theta - C\sum_{j=0}^{3}\int_{\l_2}^{L}\tM^{(j)}(\theta)^2(L-\theta)^{13/2}\d\theta  - C_{\l_1}\int_{0}^{L}\tM(\theta)^2\d\theta.
\end{align*}
Using Lemma \ref{lemma:weighted_interpolation}, we have
\begin{align*}
    \sum_{j=0}^{3}\int_{\l_2}^{L}\tM^{(j)}(\theta)^2(L-\theta)^{13/2}\d\theta  \leq C \int_{\l_2}^{L}\tM(\theta)^2\d\theta + C(L-\l_2)^2\int_{\l_2}^{L}\tM^{(4)}(\theta)^2(L-\theta)^{13/2}\d\theta,
\end{align*}
and the proposition follows.
\end{proof}
\subsection{Coercivity and Decomposition of \texorpdfstring{\(\L\)}{L}}
Proposition \ref{prop:high_norm_coercive} now follows. Indeed, summing Propositions \ref{prop:local_high}, \ref{prop:bulk_high}, \ref{prop:endpoint_high_coercivity} we have that
    \begin{align*}
        \int_{0}^{L}\tM^{(4)}(\theta)\Lb(M)^{(4)}(\theta)\psi(\theta) \d\theta &\geq \left[3 - C\l_1^2\right]\int_{0}^{L}\tM^{(4)}(\theta)^2\psi(\theta) \d\theta + \left[Kc(L-\l_2) - C\right]\int_{\l_1}^{\l_2}\tM^{(4)}(\theta)^2\theta^{-K}\d\theta \\
        &\quad +\left[1 - C(L-\l_2)^{1/2}\right]\int_{\l_2}^{L}\tM^{(4)}(\theta)^2(L-\theta)^{13/2}\d\theta- C_{\l_1}\int_{0}^{L}\tM(\theta)^2\varphi(\theta)\d\theta.
    \end{align*}
    Recalling that \(K = \l_1^{-4}, L-\l_2 = \l_1^2\) we conclude
    \begin{align*}
        \int_{0}^{L}\tM^{(4)}(\theta)\Lb(M)^{(4)}(\theta)\psi(\theta) \d\theta &\geq \min\left\{3-C\l_1^2, c\l_1^{-2} - C, 1-C\l_1\right\}\int_{0}^{L}\tM^{(4)}(\theta)^2\psi(\theta)\d\theta \\
        &\qquad - C_{\l_1}\int_{0}^{L}\tM(\theta)^2\varphi(\theta)\d\theta.
    \end{align*}
    Choosing \(\l_1\) sufficiently small we have \(\min\left\{3-C\l_1^2, c\l_1^{-2} - C, 1-C\l_1\right\} = c_2 > 0\) which completes the proof of Proposition \ref{prop:high_norm_coercive}. 
\begin{proposition}\label{prop:coercive}
    There exists \(c_* > 0\) such that for all \(M \in \tH^4\),
    \[
    \inner{M}{\Lb(M)}_{\tH^4} \geq c_*\|M\|_{\tH^4}^2.
    \]
\end{proposition}
\begin{proof}
    From Propositions \ref{prop:low_norm_coercive} and \ref{prop:high_norm_coercive} we have
    \begin{align*}
        \inner{M}{\Lb(M)}_{\tH^4} \geq M(0)^2 + M''(0)^2 + (B c_1 - C)\int_{0}^{L}\tM(\theta)^2\varphi(\theta)\d\theta + c_2\int_{0}^{L}\tM^{(4)}(\theta)^2\psi(\theta)\d\theta 
    \end{align*}
    where we have noted \(\Lb(M)(0) = M(0), \Lb(M)''(0) = M''(0)\). Choosing \(B\) sufficiently large that \(Bc_1 > C\) then completes the proof.
\end{proof}

Now we prove Proposition \ref{prop:decomp}.
\begin{proof}
    To show \(\Lb\) is maximally accretive, it suffices to show \(\Lb - \lambda I\) is maximally accretive for some \(\lambda > 0\). By Proposition \ref{prop:coercive}, it then suffices to show that \(\Lb - \lambda I\) is maximal. To prove that \(\Lb - \lambda I\) is maximal, we prove that \(\Lb\) is surjective. First observe that the transport operator \(M \mapsto G_*M'\) is closed on \(\mathcal{D}(\L)\). Since \(\Lb\) is a bounded perturbation of this operator, it follows that \(\Lb\) is also closed on the same domain \(\mathcal{D}(L) = \mathcal{D}(\Lb)\). Thus, by the closed range theorem, \(\text{Range}(\Lb) = \text{Ker}(\Lb^{\dag})^{\perp}\). Since \(\L + \L^{\dag}\) is bounded, we have \(\mathcal{D}(\mathcal{L}^{\dag}) = \mathcal{D}(\mathcal{L})\) and thus
    \[
    \inner{f}{\Lb^{\dag}f} = \inner{\Lb f}{f} \geq c_* \|f\|_{\tH^4}^2
    \]
    from which it follows that \(\text{Ker}(\L^{\dag}) = \{0\}\) and hence \(\text{Range}(\Lb) = \tH^4\) completing the proof.
\end{proof}

\section{Finite Codimension Stability}
As the profile \(M_*\) is non-smooth at \(\theta = \pi/2\), we aim to show there exists a small perturbation of \(M_*\) capable of truncating \(M_*\) near the boundary, which still blows-up in finite time. First, we use classical semigroup methods to obtain finite codimension stability for the linear problem. Then, performing suitable estimates on the nonlinear terms, we are able to utilize a stable manifold theorem argument which produces decaying solutions.
\subsection{Linear Theory}
We now study the linearized operator \(\L\) using the well-understood linear semigroup theory from \cite{EN}. We first prove that \(\L\) generates a strongly continuous semigroup.
\begin{proposition}
    \(\L: D(\L) \to \tH^4, \Lb: D(\Lb) \to \tH^4\) generate strongly continuous semigroups \(e^{t\L}, e^{t\Lb}\) respectively.
\end{proposition}
\begin{proof}
    By Proposition \ref{prop:decomp}, \(\Lb\) is maximally accretive and thus by Lumer--Phillips theorem (see Theorem 3.15, Chapter II, \cite{EN}) \(\Lb\) generates a strongly continuous contraction semigroup. By the bounded perturbation theorem (see Theorem 1.3 Chapter III of \cite{EN}) it then follows that \(\L\) generates a strongly continuous semigroup.
\end{proof}
We now recall some classical semigroup definitions from \cite{EN}. Throughout, \(A: D(A) \to \tH^4\) denotes a closed operator which generates a strongly continuous semigroup \(e^{t A}\). We denote the spectrum of \(A\) by \(\sigma(A)\) and the resolvent by \(\rho(A)\). The spectral bound of \(A\) is defined by
\[
 s(A) := \sup\big\{\mathfrak{Re}(\lambda) : \lambda \in \sigma(A)\big\},
\]
and the growth bound of \(A\) by
\[
\omega_0(A) := \inf\big\{w \in \R : \text{there exists } M_{w} > 1 \text{ such that } \|e^{tA}\|_{\tH^4 \to \tH^4} \leq M_{w}e^{wt} \text{ for all } t \geq 0\big\}.
\]
Moreover, we consider the essential norm
\[
\|e^{tA}\|_{\text{ess}} := \inf_{K \in \mathcal{K}(\tH^4)}\|e^{tA} - K\|_{\tH^4 \to \tH^4},
\]
where \(\mathcal{K}(\tH^4)\) denotes the ideal of compact operators on \(\tH^4\). The essential growth bound is then defined by
\begin{align*}
    \omega_{\text{ess}}(A) &:= \inf_{t > 0} t^{-1}\log\|e^{tA}\|_{\text{ess}}.
\end{align*}
We now consider the essential growth rate of \(\L\).
\begin{proposition}\label{prop:L_ess_spec}
    There exists \(\eta > 0\) such that \(\omega_{\text{ess}}(-\L) < -\eta\).
\end{proposition}
\begin{proof}
    By Theorem \ref{thm:coercivity}, there exists \(\eta > 0\) such that
    \[
    \|e^{-t\Lb}f\|_{\tH^4} \leq e^{-\eta t} \|f\|_{\tH^4}
    \]
    for all \(f \in \tH^4\). By Theorem 1.10, Chapter II of \cite{EN} this implies \(\{\lambda \in \C: \mathfrak{Re}(\lambda) < \eta\} \subset \rho(\Lb)\). Thus, \(\omega_{\text{ess}}(-\Lb) < -\eta\) and since \(\Lb\) is a compact perturbation of \(\L\), we have \(\omega_{\text{ess}}(-\Lb) = \omega_{\text{ess}}(-\L)\) by Proposition 2.12 of \cite{EN}.
\end{proof}

We now recall the following proposition regarding the growth rate \(\omega_0(A)\).
\begin{proposition}{(Corollary 2.11, Chapter IV, \cite{EN})}\label{prop:ess_spec}
    Let \(e^{tA}\) be a strongly continuous semigroup on \(\tH^4\) with generator \(A\). Then,
    \[
    \omega_0(A) = \max\{\omega_{\text{ess}}(A), s(A)\}.
    \]
    Moreover, for every \(w > \omega_{\text{ess}}(A),\) the set \(\sigma_{\geq w} := \sigma(A) \cap \{\lambda \in \C: \mathfrak{Re}(\lambda) \geq w\}\) is finite, and the corresponding spectral projection has finite-rank.
\end{proposition}
Applying Proposition \ref{prop:ess_spec} to \(A = -\L\) , by Proposition \ref{prop:L_ess_spec} we conclude that \(\sigma_{\geq -\eta/2} = \{\lambda_1, \dots, \lambda_n\}\) is finite and the associated spectral projection, which we henceforth denote \(\P_U\) is of finite-rank. Note that \(\P_U\) includes the projection onto any potentially unstable (positive) elements of \(\sigma(-\L)\). Define \(\P_S = \text{Id}-\P_U\) to be the orthogonal projection. We then obtain the orthogonal decomposition \(\tH^4 = \tH^4_{U} \bigoplus \tH^4_{S}\) corresponding to \(\P_{u}, \P_{s}\) respectively. Now, considering \(\L|_{\tH^4_S}\), since \(s(-\L|_{\tH^4_S}) \leq -\eta/2\) and \(\omega_{\text{ess}}(-\L) \leq -\eta\) so by Proposition \ref{prop:ess_spec} we have \(\omega_{0}(-\L|_{\tH^4_{S}}) \leq -\eta/2\) and thus the semigroup estimate
\begin{equation}
\|e^{-s\L}f\|_{\tH^4} \lesssim e^{-\eta s/2}\|f\|_{\tH^4}, \quad \text{for all } f \in \tH^4_S, s \geq 0. \label{eq:stable_semi}
\end{equation}
On the unstable part \(\tH^4_{U}\), the dynamics are finite-dimensional with all eigenvalues \(\lambda_i\) satisfying \(\mathfrak{Re}(\lambda_i) \geq -\eta/2\) and thus
\begin{equation}
    \|e^{s\L}f\|_{\tH^4} \lesssim e^{3 \eta s/4}\|f\|_{\tH^4}, \quad \text{for all } f \in \tH^4_U, s \geq 0. \label{eq:unstable_semi}
\end{equation}
Finally, we prove any unstable directions are smooth away from the boundary.
\begin{proposition}\label{prop:eigenfunctions_regularity}
    Consider the orthogonal decomposition \(\tH^4 = \tH^4_{U} \bigoplus \tH^4_S\) obtained above. If \(\Psi \in \tH^4_U\), then \(\Psi \in C^{\infty}([0, L))\).
\end{proposition}
\begin{remark}
    It is not true in general that the unstable modes are smooth at the boundary. Indeed, the profile \(M_*\) is itself an unstable mode corresponding to time translation of the blow-up which is non-smooth at the boundary.
\end{remark}
\begin{proof}
    The result follows from the following lemma,
    \begin{lemma}\label{lemma:eig_smooth_lemma}
        If \(\lambda \in \sigma_{\geq -\eta/2}, F \in C^{\infty}([0, L))\) and \(\Psi \in \tH^4\) solves
        \begin{align}\label{eq:eig}\begin{cases}
            \Psi + 2 \Xi M_*' + 2G_*\Psi' - \Xi' M_* - G_*' \Psi - 2M_*\Xi\tan\theta - 2 \Psi G_*\tan\theta + \lambda \Psi = F \\
            \Xi'' + 4\Xi = \Psi'\cos^2\theta \\
            \Xi(0) = \Xi(L) = 0
        \end{cases}
        \end{align}
       then, \(\Psi \in C^{\infty}([0, L))\).
    \end{lemma}
    Before proving the result we show how Proposition \ref{prop:eigenfunctions_regularity} follows. By Proposition \ref{prop:ess_spec}, \(\tH^4_{U}\) can be decomposed into a finite number of eigenspaces with eigenvalues \(\lambda \in \sigma_{\geq - \eta/2}\). Fixing an eigenvalue \(\lambda \in \sigma_{\geq -\eta/2}\) we obtain an eigenspace \(\psi_1, \dots, \psi_n\) with \(\psi_k\) solving
    \[
        -(\L + \lambda)\psi_1 = 0, \quad -(\L + \lambda)\psi_{k} = \psi_{k-1}.
    \]
    By Lemma \ref{lemma:eig_smooth_lemma}, it follows that \(\psi_1 \in C^{\infty}([0, L))\) and then inductively applying Lemma \ref{lemma:eig_smooth_lemma} we conclude \(\psi_k \in C^{\infty}([0, L))\) for all \(k \leq n\) completing the proof of Proposition \ref{prop:eigenfunctions_regularity}. Now, we prove the lemma. The equation \eqref{eq:eig} is singular at \(\theta = 0\) since \(G_*(0) = 0\). Let \(\Psi\) be as in Lemma \ref{lemma:eig_smooth_lemma}. We proceed as in Proposition \ref{prop:profile_local} to first show that \(\Psi\) is smooth locally near zero. To this end, we write
    \[
    \Psi(\theta) = \Psi(0) + \frac{\theta^2}{2}\Psi''(0) + \theta^4 \psi, \quad \Xi(\theta) = \Xi'(0)\theta + \frac{\Xi'''(0)}{6}\theta^3 + \theta^5 \xi(\theta).
    \]
    Then, \(\psi, \xi\) solve
    \begin{align*}
    \begin{cases}
        (8+\lambda)\psi + 2\theta\psi' = 20\xi + 4\theta \xi' + R_1 \\
         20\xi + 10\theta \xi' + \theta^2\xi'' = 4\psi + \theta \psi'  + R_2.
    \end{cases}
    \end{align*}
    Here, \(R_1, R_2\) are consist of the remaining higher order (in \(\theta\)) terms. Explicitly, we have
    \begin{align*}
        \theta^4 R_1 & = R_1(2 \Xi M_*') + R_1(2G_*\Psi') + R_1(\Xi' M_*) + R_1(G_*' \Psi) + R_1(2M_*\Xi\tan\theta) + R_1(2 \Psi G_*\tan\theta) +  R_1(F) \\
        \theta^3 R_2 &= \theta\sin^2\theta\Psi''(0) - 4\theta^3\sin^2\theta \psi - \theta^4 \sin^2\theta \psi'
    \end{align*}
    where the remainders comprising \(R_1\) are defined by 
    \begin{align*}
        \begin{dcases}
            R_1(2\Xi M_*') = -2\Xi'(0)\theta(M_*' - \P_1M_*') - 2\left(\frac{\Xi'''(0)}{6}\theta^3 + \theta^5\xi\right)M_*' \\
            R_1(2G_*\Psi') = -2(G_* - \P_1 G_*)(\theta\Psi''(0) + 4\theta^3\psi + 2\theta^4\psi') \\
            R_1(\Xi' M_*) = \Xi'(0)(M_* - \P_2M_*) + \left(\frac{\Xi'''(0)}{2}\theta^2 + 5\theta^4\xi + \theta^5\xi'\right)(M_* - \P_0 M_*) \\
            R_1(G_*' \Psi) = (G_*'-\P_2 G_*')\Psi(0) + (G_*' - \P_0 G_*')\left(\frac{\theta^2}{2}\Psi''(0)+\theta^4\psi\right) \\
            R_1(2M_*\Xi\tan\theta) = 2\Xi'(0)\theta(M_*\tan\theta - \P_1(M_*\tan\theta)) + 2M_*\tan\theta\left(\frac{\Xi'''(0)}{6}\theta^3 + \theta^5 \xi\right) \\
            R_1(2 \Psi G_*\tan\theta) = 2\Psi(0)(G_*\tan\theta - \P_2(G_*\tan\theta)) + 2G_*\tan\theta\left(\frac{\theta^2}{2}\Psi''(0) + \theta^4\psi(\theta)\right)\\
            R_1(F) = F - \P_2 F.
        \end{dcases}
    \end{align*}
    Note that both remainders \(R_i\) have the form
    \begin{align}\label{eq:eig_remainder}
        R_i(\theta) = f_{1, i}(\theta) + \theta^3\psi'(\theta)f_{2, i}(\theta) + \theta^2 \psi(\theta) f_{3, i}(\theta) + \theta^2 \xi(\theta)f_{4, i}(\theta) + \theta^3 \xi'(\theta)f_{5, i}(\theta)
    \end{align}
    where \(f_{i, j} \in C^{\infty}([0, L))\), \(1 \leq j \leq 5\) are smooth functions away from the boundary depending on the profile. 
    Setting \(q = \theta \xi'\) we obtain the singular first order system
    \[
    \theta y'(\theta) + Ay = R
    \]
    where \(y = (q, \psi, \xi), R = (R_2+\frac12 R_1, \frac12 R_1, 0)\) and
    \[
    A = \begin{pmatrix}
        7 & \frac{\lambda}{2} & 10 \\
        -2 & 4 + \frac{\lambda}{2} & -10 \\
        -1 & 0 & 0
    \end{pmatrix}.
    \]
    It is easily verified that \(A\) is positive definite for \(\eta\) sufficiently small with eigenvalues  \((\lambda + 4)/2, 4, 5\). If \(v_1, v_2, v_3\) denote the eigenvectors of \(A\) corresponding to eigenvalues \(\lambda_1, \lambda_2, \lambda_3 > 0\) respectively then using variation of parameters we write obtain the following fixed point problem
    \[
    \alpha_j(\theta) = \frac{1}{\theta^{\lambda_j}}\int_{0}^{\theta}\phi^{\lambda_j - 1}\left(Q^{-1}F\left(\sum_{k=1}^{n}\alpha_k(\theta)v_k, \phi\right)\right)_{j}\d\phi
    \]
    where \(Q\) diagonalizes \(A\). From the form of the remainder \eqref{eq:eig_remainder}, after integrating by parts the terms \(\theta^3\psi'(\theta)f_{2, i}(\theta), \theta^3 \xi'(\theta)f_{5, i}(\theta)\) we obtain a fixed point problem of the form in Lemma \ref{lemma:sing_ode}. Thus, by Lemma \ref{lemma:sing_ode}, \(\Psi, \Xi\) are smooth in a neighbourhood of \(0\). Finally, to see that the \(\Psi, \Xi\) are smooth away from zero and the boundary, for all \(\theta \neq 0, L\), we have \(G_*(\theta) > 0\) and thus we can divide to have
    \begin{equation}\label{eq:unstable_mode_bulk}
    \Psi' = \frac{1}{2G_*}\big(-\Psi - 2 \Xi M_*'  + \Xi' M_* + G_*' \Psi + 2M_*\Xi\tan\theta + 2 \Psi G_*\tan\theta + \lambda \Psi + F\big).
    \end{equation}
    Since \(\Psi \in \H^4\) it follows by Sobolev embedding that \(\Psi \in C^{3, \frac12-}((0, L))\) and thus that \(\Xi \in C^{2, \frac12-}((0, L))\). However \eqref{eq:unstable_mode_bulk} then implies that \(\Psi \in C^{\infty}((0, L))\).
\end{proof}

\subsection{Nonlinear Estimates}
In this section, we prove the necessary estimates on the nonlinear terms to establish local well-posedness in \(\tH^4\) and to truncate the profile \(M_*\) in a neighbourhood of \(L\). Consider the bilinear operator \(N(f_1, f_2)\) defined by 
\[
N(f_1, f_2) = 2F_1\p_\theta f_2 - (\p_\theta F_1) f_2 - 2 f_2 F_1\tan\theta,
\]
where \(F_1\) solves
\[
F_1''(\theta) + 4F_1(\theta) = F_1'(\theta)\cos^2\theta, \quad F_1(0) = F_1(L) = 0.
\]
\begin{lemma}\label{lemma:nl_est}
    Let \(f_1, f_2, f_3 \in \tH^4\). Then,
    \begin{align}
        |\inner{N(f_1, f_2)}{f_2}_{\tH^4}| \lesssim \|f_1\|_{\tH^4}\|f_2\|_{\tH^4}^2 \quad \text{and} \quad |\inner{N(f_1, f_2)}{f_3}| \lesssim \|f_1\|_{\tH^4}(\|\p_\theta f_2\|_{\tH^4} + \|f_2\|_{\tH^4})\|f_3\|_{\tH^4}.
    \end{align}
\end{lemma}
\begin{proof}
     The stretching terms are easily dealt with as by Cauchy--Schwarz and Lemmas \ref{lemma:G_bound_lwp}, \ref{lemma:algebra} we have
     \[
        |\inner{F_1'f_2 + 2f_2F_1\tan\theta}{f_2}_{\tH^4}| \lesssim \|f_2\|_{\tH^4}^2(\|F_1'\|_{\tH^4} + \|F_1\|_{\tH^4}) \lesssim \|f_2\|_{\tH^4}^2\|f_1\|_{\tH^4}.
     \]
     The transport term requires additional care however to avoid derivative loss. First, we deal with the low norm. A direct computation using the definition of \(\P_k\) and integration by parts shows that
     \begin{align*}
         &2\int_{0}^{L}[F_1f_2' - \P_2(F_1f_2')][f_2 - \P_2 f_2]\varphi(\theta) \\
        &= \int_{0}^{L}F_1(\theta)\p_\theta(f_2-\P_2 f_2)^2\varphi(\theta) + 2\theta f_2''(0)(F_1 - \P_1 F_1)(f_2 - \P_2 f_2)\varphi(\theta) \\
         &= \int_{0}^{L}\left[F_1'(\theta) + F_1(\theta) \frac{\varphi'(\theta)}{\varphi(\theta)}\right](f_2-\P_2 f_2)^2\varphi(\theta) + 2\theta f_2''(0)(F_1 - \P_1 F_1)(f_2 - \P_2 f_2)\varphi(\theta).
     \end{align*}
     Since \(\varphi'/\varphi \leq C\theta^{-1}(L-\theta)^{-1}\), using that \(F_1(0) = F_1(L) = 0\) it follows that \(F_1 + F_1 \varphi'/\varphi \leq C\|F'\|_{L^\infty}\) for some constant \(C > 0\). Applying this pointwise bound on the first term and H\"older's inequality on the second, we~find
     \begin{align*}
         \int_{0}^{L}[F_1f_2' - \P_2(F_1f_2')][f_2 - \P_2 f_2]\varphi(\theta) \lesssim &\|F_1'\|_{L^\infty}\|f_2\|_{\tH^4}^2 + |f_2''(0)|\left(\int_{0}^{L}\theta^2 (F_1 - \P_1 F_1)^2\varphi(\theta)\right)^{1/2}\|f_2\|_{\tH^4}.
     \end{align*}
     It follows as in Lemma \ref{lemma:equiv_norm} that
     \[
     \int_{0}^{\l_1}\frac{(F_1 - \P_1 F_1)^2}{\theta^6} \lesssim \int_{0}^{\l_1} F_1'''(\theta)^2 \lesssim \|f_1\|_{\tH^4}^2
     \]
     and since \(c_1 < \varphi(\theta) < c_2\) for some \(c_1, c_2 > 0\) for all \(\theta \in [\l_1, L]\) it follows that
     \[
     \int_{\l_1}^{L}\theta^2(F-\P_1 F_1)^2\varphi(\theta) \leq  C|F_1'(0)|^2 + C\int_{0}^{\l_1}F_1(\theta)^2\varphi(\theta).
     \]
     Since \(\|F_1'\|_{L^\infty} \lesssim \|f_1\|_{L^\infty}\), from the \(L^\infty\) embedding of Lemma \ref{lemma:infty_embed} it then follows that 
     \[
     \int_{0}^{L}[F_1f_2' - \P_2(F_1f_2')][f_2 - \P_2 f_2]\theta^{-8} \lesssim \|f_1\|_{\tH^4}\|f_2\|_{\tH^4}^2.
     \]
     In the high norm we deal only with the highest order term as the rest can be dealt with similarly by interpolating using Lemma \ref{lemma:weighted_interpolation}. The highest order terms gives
     \[
     \int_{0}^{L}2F_1(\theta)f_2^{(5)}(\theta)f_2^{(4)}(\theta)\psi(\theta) = \int_{0}^{L}\left[F_1' + \frac{\psi'}{\psi}F_1\right]f_2^{(4)}(\theta)^2\psi(\theta).
     \]
     Since \(\psi'/\psi \leq  C(L-\theta)^{-1}\) and \(F_1(L) = 0\), we have \(|F_1' + F_1\psi'/\psi| \lesssim \|F_1'\|_{L^\infty}\) and thus
     \[
     \int_{0}^{L}2F_1(\theta)f_2^{(5)}(\theta)f_2^{(4)}(\theta)\psi(\theta) \lesssim \|F_1'\|_{L^\infty}\|f_2\|_{\tH^4}^2 \lesssim C\|f_1\|_{\tH^4}\|f_2\|_{\tH^4}^2.
     \]
     This gives the first inequality of the lemma. The second inequality follows easily from Cauchy--Schwarz and the algebra property.
\end{proof}

\subsection{Existence of Decaying Solutions}
In this section, we construct a decaying solution of \eqref{eq:ipm_pert} such that \(M_0 + M_*\) is smooth and compactly supported away from the boundary completing the proof of Theorem \ref{theorem:ipm_1d}. First, we decompose the solution into a stable and unstable piece. This decomposition for finite codimension stability was developed by Elgindi and Pasqualotto in \cite{EP1}. A similar strategy was employed in the work of Chen and Hou \cite{CH1, CH2} in the case where full stability of the profile was known and has been used in the finite codimension case in the later works \cite{CH1, EP1, CCSV_CompressVorticity}. Ordinarily, when working with an approximate profile, one must rule out the presence of a centre subspace for the linearized operator which can be challenging. Fortunately, as seen in \cite{CCSV_CompressVorticity}, this is not necessary when an exact profile is available. We now consider the full evolution
\[
\p_s M + \L(M) = N(M, M), \quad M(0) = M_0 \in \tH^4.
\]
We decompose \(M = f + g\) where \(f, g\) solve
\begin{align}\label{eq:decomp_sys}
\begin{cases}
    \p_s f + \Lb f = N(f+g, f+g) \\
    \p_s g + \mathcal{L}g = - \Lk f.
\end{cases}
\end{align}
Here, \(f\) is the stable portion of the solution which solves an equation with a fully coercive linear part, and \(g\) is the unstable portion. Coercivity of \(\Lb\) and favourable estimates on the nonlinear terms will then allow us to solve the first equation in \eqref{eq:decomp_sys} for \(f\). The second equation can then be solved by an unstable manifold theorem argument to obtain a decaying solution \(g\).
\begin{proposition}\label{prop:decay}
    There exists \(\gamma, a > 0\) and \(f_0 \in \tH^4\), \(g_0 \in \tH^4_U\) such that
    \begin{enumerate}
        \item \(\text{supp}(f_0 + g_0 + M_*) \subset [0, L-a]\)

        \item \(f_0 + g_0 + M_*   \in C^{\infty}([0, L])\)

        \item There exists a global solution \(f, g \in C^0([0, \infty); \tH^4)\) such that
        \begin{equation}
            \|f(s)\|_{\tH^4} + \|g(s)\|_{\tH^4} \leq Ce^{-s\gamma}\|f_0\|_{\tH^4}.
        \end{equation}
    \end{enumerate}
\end{proposition}
\begin{proof}
To solve for the stable part \(f\), we will use the decay of solutions of the first equation in \eqref{eq:decomp_sys} due to the coercivity of \(\Lb\). Then, following classical stable manifold theorem constructions, we must construct \(g\) by integrating from \(s = +\infty\). Given \(f\), if we take \(\P_S g(0) \equiv 0\), then \(g\) is determined by
\begin{equation}
    g(s) = \int_{s}^{\infty}e^{(s'-s)\L}\P_U \Lk f(s') \d s'- \int_{0}^{s}e^{(s'-s)\L}\P_S \Lk f(s') \d s'.
\end{equation}
Here, we recall \(\P_U\) denotes the projection onto the unstable subspace \(\tH^4_U\) and \(\P_S\) its orthogonal projection \(\P_S = \text{Id} - \P_U\). Now we define the iterates \((f_n, g_n)\) and show the sequence converges to a solution of \eqref{eq:decomp_sys} with the properties demanded by Proposition \ref{prop:decay}. Define \(f_0 \equiv 0\) and then for \(n \geq 1\) set
\begin{equation}
    g_{n}(s) = \int_{s}^{\infty}e^{(s'-s)\L}\P_U \Lk f_{n-1}(s') \d s' - \int_{0}^{s}e^{(s'-s)\L}\P_S \Lk f_{n-1}(s') \d s' \\
\end{equation}
and let \(f_{n}\) solve
\begin{equation}
        \p_s f_{n} + \Lb f_n = N(f_n+g_n, f_n+g_n), \quad f_{n}(0) = -(M_{*} + g_n(0))\chi_a \label{eq:f_iter}.
\end{equation}
Here \(\chi_a(\theta) = \chi((L-\theta)/a)\) and \(\chi\) is a compactly supported, non-negative bump function such that \(\chi \equiv 1\) on \([0, 1]\) and vanishes on \([2, \infty)\). We have chosen our initial data such that
\[
f_n(0) + g_n(0) = -M_*\chi_{a} + (1-\chi_a)g_n(0)
\]
where we note that \(g_n(0) \in \tH^4_U\) for all \(n\). We have that \(-M_*\chi_a\) is smooth and by Lemma \ref{lemma:eig_smooth_lemma} \((1-\chi_a)\psi\) is smooth for all \(\psi \in \tH^4_U\) and thus it suffices to prove that \(f_n, g_n\) converge in \(\tH^4\) and have exponentially decaying limits. We prove that \(\|f_n(s)\|_{\tH^4} \leq \epsilon e^{-\eta s}\) for all \(n\) using induction and a bootstrapping. Clearly the result holds for \(n = 0\). Now, we make the bootstrap assumption that \(\|f_n(s)\|_{\tH^4} \leq \epsilon e^{-\eta s}\) and assume further that \(\|f_{n-1}(s)\|_{\tH^4} \leq \epsilon e^{-\eta s}\) for some \(\epsilon > 0\) to be chosen later. Using the semigroup estimates \eqref{eq:stable_semi}, \eqref{eq:unstable_semi} and the induction hypothesis yields the following decay on the unstable part \(g_n\)
\begin{align*}
    \|g_n(s)\|_{\tH^4} &\lesssim \int_{s}^{\infty}e^{3\eta (s'-s)/4}\|\L_K f_{n-1}(s')\|_{\tH^4}\d s' + \int_{0}^{s}e^{(s'-s)\eta/2}\|\L_K f_{n-1}(s')\|_{\tH^4}\d s' \\
    &\lesssim \epsilon\int_{s}^{\infty}e^{3(s'-s)\eta/4}e^{-\eta s'}\d s' + \epsilon \int_{0}^{s}e^{(s'-s)\eta/2}e^{-\eta s'}\d s'\\
    &\lesssim \epsilon e^{-\eta s/2}.
\end{align*}
Repeating the same argument and noting that \(\p_\theta\L_K\) is bounded on \(\tH^4\) we conclude that we actually have
\begin{equation}
\|g_n(s)\|_{\tH^4} + \|\p_\theta g_n(s)\|_{\tH^4} \lesssim \epsilon e^{-\eta s/2}. \label{eq:g_n_decay}
\end{equation}
Now, taking the inner product of \eqref{eq:f_iter} with \(f_n\),
\begin{align*}
    \frac12 \p_s \|f_n\|_{\tH^4}^2 + \eta \|f_n\|_{\tH^4}^2 = \inner {N(f_n + g_n, f_n+g_n)}{f_n}_{\tH^4}.
\end{align*}
From bilinearity and the nonlinear estimates in Lemma \ref{lemma:nl_est},
\[
|\inner{N(f_n+g_n, f_n+g_n)}{f_n}_{\tH^4}| \lesssim \|f_n\|_{\tH^4}\left(\|f_n\|_{\tH^4}^2 + \|g_n\|_{\tH^4}^2 + \|\p_\theta g_n\|_{\tH^4}^2\right).
\]
It then follows from the induction hypothesis and \eqref{eq:g_n_decay} that
\[
\p_s \|f_n(s)\|_{\tH^4} + \eta\|f_n(s)\|_{\tH^4} \leq C(\epsilon^2 e^{-\eta s} + \|f_n(s)\|_{\tH^4}^2).
\]
By Gr\"onwall's inequality and the bootstrap assumption,
\[
\|f_n(s)\|_{\tH^4} \leq C e^{-\eta s}(\|f_n(0)\|_{\tH^4} + C\epsilon^2\eta^{-1}).
\]
Finally, noting that
\[
\|f_n(0)\|_{\tH^4} \leq \|g_n(0)\|_{\tH^4} + Ca \leq Ca(1 + \|f_{n-1}(0)\|_{\tH^4}) \leq Ca(1 + \epsilon)
\]
choosing \(\epsilon, a\) sufficiently small we conclude
\[
\|f_n(s)\|_{\tH^4} \leq \frac{\epsilon}{2}e^{-\eta s}
\]
concluding the bootstrap. Finally, we show the sequence \((f_n, g_n)\) are Cauchy in the space of exponential decaying functions. For fixed \(\gamma > 0\) define the norm
\[
\|f\|_{L^\infty_\gamma\tH^4} = \sup_{s > 0}e^{s\gamma}\|f(s)\|_{\tH^4}.
\]
Consider the differences \(F_n = f_{n+1}-f_n\) and \(G_n = g_{n+1}-g_{n}\) which satisfy
\begin{align}\label{eq:F_eqn}\begin{dcases}
    \p_s F_n + \Lb F_n = N(f_{n+1}+g_{n+1}, f_{n+1}+g_{n+1}) - N(f_n+g_n, f_n+g_n) \\
     G_n(s) = \int_{s}^{\infty}e^{(s'-s)\L}\P_U \Lk F_{n-1}(s')\d s' - \int_{0}^{s}e^{(s'-s)\L}\P_S \Lk F_{n-1}(s')\d s'.
    \end{dcases}
\end{align}
Taking the inner product of \eqref{eq:F_eqn} with \(F_n\), we have
\begin{align*}
    \frac{1}{2}\p_s \|F_n\|_{\tH^4}^2 + \eta \|F_n\|_{\tH^4}^2 &\leq \inner{N(f_{n+1}+g_{n+1}, f_{n+1}+g_{n+1})}{F_n}_{\tH^4} - \inner{N(f_n+g_n, f_n+g_n)}{F_{n}}_{\tH^4} \\
    &\lesssim \|F_n\|_{\tH^4}(\|F_n\|_{\tH^4} + \|G_n\|_{\tH^4})(\|f_{n+1}\|_{\tH^4} + \|g_{n+1}\|_{\tH^4} + \|f_n\|_{\tH^4} + \|g_n\|_{\tH^4}).
\end{align*}
Since \(\|f_n(s)\|_{\tH^4} + \|g_n(s)\|_{\tH^4} \leq \epsilon e^{-\eta s/2}\) for all \(n\), we have
\begin{equation}\label{eq:Fn_ode}
\p_s \|F_n\|_{\tH^4} + \eta \|F_n\|_{\tH^4} \lesssim \epsilon e^{-\eta s/2}(\|F_n\|_{\tH^4} + \|G_n\|_{\tH^4}).
\end{equation}
Now, we note that from \eqref{eq:F_eqn}, it follows that \(\|G_n(s)\|_{\tH^4} \lesssim e^{-\eta s/2}\|F_{n-1}\|_{L^\infty_\gamma\tH^4}\) for all \(\frac{3\eta}{4} < \gamma\). Therefore from \eqref{eq:Fn_ode} we can conclude that
\[
\|F_n(s)\|_{\tH^4} \lesssim \epsilon e^{-9\eta s/10}(\|F_n(0)\|_{\tH^4} + \|F_{n-1}\|_{L^\infty_\gamma \tH^4}) \lesssim \epsilon e^{-9\eta s/10}\|F_{n-1}(s)\|_{L^\infty_\gamma \tH^4}
\]
and thus
\begin{align*}
    e^{\gamma s}\|F_n(s)\|_{\tH^4} \lesssim \epsilon e^{\left(\gamma -9\eta/10\right)s}\|F_{n-1}(s)\|_{L^\infty_\gamma, \tH^4}.
\end{align*}
Choosing \(\epsilon\) sufficiently small, and \(3\eta/4 < \gamma < 9\eta/10\), we conclude \(F_n\) are Cauchy in \(L^\infty_\gamma\tH^4\) (so that \(G_n\) are Cauchy in \(L^\infty_{\eta/2} \tH^4\)).
\end{proof}

\section{Blow Up For Solutions with Compact Support}
In this section, we show that the radially homogeneous blow-up solutions constructed in the previous sections can be made compactly supported. We follow the approach developed in \cite{EJ-B, EJ-S}. We decompose the initial data \(\rho_0 = \rho_{SI, 0} + \overline{\rho_{0}}\) where \(\rho_{SI, 0} = rP(\theta)\) is radially homogeneous and \(\nabla^{\perp}\overline{\rho_0} \in C^{\alpha}(\overline{\Omega})\). We then show that near \(r = 0\), the homogeneous part dominates and thus obtain singularity formation in the full system.

\begin{proposition}
Suppose \(\rho_0 = \rho_{SI, 0} + \overline{\rho_0} \in \Ca(\overline{\Omega})\) where \(\rho_{SI, 0} = rP_0(\theta)\) and \(\nabla^{\perp}\overline{\rho_0} \in C^\alpha(\overline{\Omega})\). Moreover, we assume that \(\rho_0\) is even in \(y\) and \(P_0\) is even in \(\theta\) and \(P_0 \in C^{2, \alpha}([0, L])\). Finally, we assume the vanishing, \(\nabla^{\perp}\overline{\rho_0}(0) = (0,0)\). Then,
\begin{enumerate}
    \item There exists \(T > 0\) such that the unique local solution of \eqref{eq:IPM} can be written
    \[
    \rho(t, \theta) = \rho_{SI}(t, \theta) + \overline{\rho}(t, \theta), \qquad \rho_{SI} = rP(t, \theta), \quad \nabla^{\perp}\overline{\rho} \in C^{0}([0, T]; C^{\alpha}(\overline{\Omega}))
    \]
    where \(P(t, \cdot) \in C^{2, \alpha}([0, L])\) is the unique local solution of the 1D system \eqref{eq:ipm_1d}. The time \(T > 0\) on which the decomposition is valid can be extended if and only if \(\rho, P\) do not blow-up.
    \item If there exists \(T_* < \infty\) such that \(\limsup_{t \to T_*}\|P(t, \cdot)\|_{L^\infty} = +\infty\) then there exists \(0 < T' \leq T_*\) such that \(\limsup_{t \to T'}\|\nabla^{\perp}\rho\|_{\Ca} = +\infty\).
\end{enumerate}
\end{proposition}
\begin{proof}
The proof follows that of \cite{EJ-B} but we include it for sake of completeness. We write \(\rho = \rho_{SI} + \overline{\rho}\) where \(\rho_{SI} = rP(t, \theta)\) and \(\overline{\rho} = \rho - \rho_{SI}\). Likewise write \(u = \nabla^{\perp}\Psi_{SI} + \nabla^{\perp}\overline{\Psi}\) where \(\Psi_{SI} = r^2G(\theta)\) and \(\overline{\Psi} = \Psi - \Psi_{SI}\). Here, \(\Psi\) is the stream function defined by \(\Psi = \Delta^{-1}\p_{x_2}\rho\) . Now, \(\nabla^{\perp}\overline{\rho}\) satisfies the equation
\begin{equation}
\p_t \nabla^{\perp}\overline{\rho} + (u_{SI} \cdot \nabla)\nabla^{\perp}\overline{\rho} + (\overline{u} \cdot \nabla)\nabla^{\perp}\rho_{SI} + (\overline{u} \cdot \nabla)\nabla^{\perp}\overline{\rho} = \nabla u_{SI} \cdot \nabla^{\perp}\overline{\rho} + \nabla \overline{u} \nabla^{\perp}\rho_{SI} + \nabla \overline{u} \nabla^{\perp}\overline{\rho} \label{eq:rho_bar_eqn}
\end{equation}
We begin by finding apriori \(L^{\infty}\) bounds for \(\nabla^{\perp}\overline{\rho}\). First, note that \(\|\nabla u_{SI}\|_{L^\infty} + \|\nabla^{\perp}\rho_{SI}\|_{L^\infty} \leq C \|P\|_{C^1}\). Then, from \eqref{eq:rho_bar_eqn},
\begin{align*}
\frac{d}{dt}\|\nabla^{\perp}\overline{\rho}\|_{L^{\infty}} &\leq \|(\overline{u} \cdot \nabla) \nabla^{\perp}\rho_{SI}\|_{L^{\infty}} + \|\nabla u_{SI} \nabla^{\perp}\overline{\rho}\|_{L^\infty} + \|\nabla \overline{u} \nabla^{\perp}\rho_{SI}\|_{L^\infty} + \|\nabla \overline{u} \nabla^{\perp}\overline{\rho}\|_{L^\infty} \\
&\lesssim \left\|\frac{\overline{u}}{|x|}\right\|_{L^\infty}\|P\|_{C^2} + \|P\|_{C^1}\|\nabla^{\perp}\overline{\rho}\|_{L^\infty} + \|\nabla \overline{u}\|_{L^\infty}\|P\|_{C^1} + \|\nabla \overline{u}\|_{L^\infty}\|\nabla^{\perp}\overline{\rho}\|_{L^\infty}.
\end{align*}
Now, if \(\nabla^{\perp} \rho \lesssim |x|^{\alpha}\) we claim \(\nabla \overline{u}(x) \lesssim |x|^{1+\alpha}\). Indeed, splitting into near and far field,  we have
\begin{align*}
    |\overline{u}(x)| &\lesssim \int_{\Omega_\beta}|\nabla_{x} G(x, z)| |z|^{\alpha}\d z \\
    &= \int_{[|x-z| < 2|x|] \cap \Omega_\beta}|\nabla_{x} G(x, z)| |z|^{\alpha}\d z + \int_{[|x-z| \geq 2|x| \cap \Omega_\beta]}|\nabla_{x} G(x, z)| |z|^{\alpha}\d z.
\end{align*}
In the near field, using that \(|\nabla_x G(x, z)| \lesssim |x-z|^{-1}\) and \(|z|^{\alpha} \lesssim |x|^{\alpha}\) we have
\[
\int_{[|x-z| < 2|x|] \cap \Omega_\beta}|\nabla_{x} G(x, z)| |z|^{\alpha}\d z \lesssim \int_{|x-z| < 2|x|}\frac{|x|^{\alpha}}{|x-z|}\d z\lesssim |x|^{1+\alpha}.
\]
In the far field, we use that \(\nabla_x G(x, z) \lesssim |x|^{1/\beta}/|x-z|^{1/\beta +1}\) where \(\beta < 1/2\) to again obtain 
\[
 \int_{[|x-z| \geq 2|x| \cap \Omega_\beta]}|\nabla_{x} G(x, z)| |z|^{\alpha}\d z \lesssim \int_{|x-z| \geq 2|x|}\frac{|z|^\alpha}{|x-z|^{1/\beta + 1}}\d z \lesssim |x|^{1/\beta + 1}.
\]
In particular, we now conclude \(\nabla^{\perp}\overline{u}(0) = 0\), and from \eqref{eq:rho_bar_eqn} it follows that \(\nabla^{\perp}\overline{\rho}(0) = 0\) is propagated. We therefore have
\[
\left|\frac{\overline{u}(x)}{|x|}\right| + |\nabla \overline{u}(x)| \lesssim \|\nabla \overline{u}(x)\|_{C^\alpha}|x|^{\alpha}, \quad \left|\frac{\overline{\rho}(x)}{|x|}\right| + |\nabla^\perp \overline{\rho}(x)| \lesssim \|\nabla^{\perp} \overline{\rho}(x)\|_{C^\alpha}|x|^{\alpha}.
\]
Since \(\nabla^{\perp}\overline{\rho}(0) = 0, \nabla \overline{u}(0) = 0\) we may apply Lemma \ref{lemma:Ca_product} to \(\nabla \overline{u}, \nabla^{\perp}\overline{\rho}, \overline{u}/|x|\) in the forthcoming \(C^{\alpha}\) estimates. From \eqref{eq:rho_bar_eqn}, and the product estimate of Lemma \ref{lemma:Ca_product} we have
\begin{align*}
    \frac{d}{dt}\|\nabla^{\perp}\overline{\rho}\|_{C^{\alpha}} &\leq \|(\overline{u} \cdot \nabla) \nabla^{\perp}\rho_{SI}\|_{C^\alpha} + \|\nabla u_{SI} \nabla^{\perp}\overline{\rho}\|_{C^\alpha} + \|\nabla \overline{u} \nabla^{\perp}\rho_{SI}\|_{C^\alpha} + \|\nabla \overline{u} \nabla^{\perp}\overline{\rho}\|_{C^\alpha} \\
    &\lesssim \left\|\frac{\overline{u}(x)}{x}\right\|_{C^{\alpha}}\|P\|_{C^{2,\alpha}} + \|P\|_{C^{1, \alpha}}\|\nabla^{\perp}\overline{\rho}\|_{C^\alpha} + \|\nabla \overline{u}\|_{C^{\alpha}}\|P\|_{C^{1, \alpha}} + \|\nabla \overline{u}\|_{L^\infty}\|\nabla^{\perp}\overline{\rho}\|_{C^\alpha}.
\end{align*}
Upon noting that \(\|\overline{u}(x)/|x|\|_{C^\alpha} \lesssim \|\nabla \overline{u}\|_{C^{\alpha}} \lesssim \|\nabla^{\perp}\overline{\rho}\|_{C^\alpha}\) we have
\[
 \frac{d}{dt}\|\nabla^{\perp}\overline{\rho}\|_{C^{\alpha}} \lesssim \left\|\nabla^{\perp} \overline{\rho}\right\|_{C^{\alpha}}(\|P\|_{C^{2,\alpha}} + \|\nabla \overline{u}\|_{L^\infty}).
\]
Given this apriori estimate, it is straightforward to verify that \eqref{eq:rho_bar_eqn} has a unique, local in time solution and the solution satisfies \(\overline{\rho} = \rho - rP\). Moreover, if \(\rho\) and \(P\) remain regular up to time \(T\), then \(\|\nabla^{\perp}\rho\|_{L^\infty}, \|P\|_{C^{2, \alpha}}\) are uniformly bounded and it then follows that \(\|\nabla^{\perp}\overline{\rho}\|_{L^\infty}\) remains bounded and hence \(\overline{\rho}\) remains regular as well. Thus, the decomposition persists until either \(\rho, P\) blow-up for their respective systems. Now we prove the second statement of the theorem. Suppose for sake of contradiction that
\[
\|\nabla u\|_{L^{\infty}} + \|\nabla \rho\|_{L^{\infty}} \leq C
\]
remain uniformly bounded for \(0 \leq t < T_*\).
Then, since \(|\nabla \overline{\rho}| \lesssim |x|^{\alpha}\) locally, it follows that
\begin{align*}
    \limsup_{|x| \to 0}|\nabla^{\perp}\rho| = \limsup_{|x| \to 0}|\nabla^{\perp}\rho_{SI}+ \nabla^{\perp}\overline{\rho}| = \limsup_{|x| \to 0}|\nabla^{\perp}\rho_{SI}| \geq \|P\|_{\text{Lip}} 
\end{align*}
which is a contradiction as \(t \to T_*\).
\end{proof}

\section*{Appendix}
Here, we prove a general result regarding the existence and smoothness of solutions to a fixed point problem which arises when solving singular ODE of the form
\begin{equation}\label{sing_ode}
    \theta y' + Ay = F(y, y', \theta)
    \end{equation}
    where \(y: [0, d]^n \to \R^n\),  and \(A\) is a constant, positive definite matrix. We use this result to construct a local solution of the profile equation in \S 4.2 and to prove regularity of any unstable modes in \S 7.1
\begin{lemma}\label{lemma:sing_ode}
    Consider the fixed point problem
    \begin{equation}\label{eq:sing_ode}
    y(\theta) = \theta^2G(y, \theta) + \theta^{-\lambda}\int_{0}^{\theta}\phi^{\lambda-1}(v+\phi^2 F(y, \phi))\d\phi
    \end{equation}
    where \(y: [0, d]^n \to \R^n\), \(v \in \R^n\) is a constant vector, \(\lambda > 0\) and \(F, G\) are smooth. Then there exists \(d > C/(\|F\|_{C^1} + \|G\|_{C^1})^{1/2}\) where \(C\) is a constant depending only on \(\lambda\) such that there is a unique solution \(y \in C^{\infty}([0, d])\) to \eqref{sing_ode}. Moreover, \(\|y\|_{L^\infty} \leq 2\|v\|\).
\end{lemma}
\begin{proof}
   The result follows easily from the Banach fixed point theorem. Indeed, if \(y_1, y_2 \in B_{R}(0) \subset C^0([0, d])\) where \(R = 2\|v\|\) then,
   \begin{align*}
       &\left|\theta^2(G(y_1, \theta) - G(y_2, \theta)) + \theta^{-\lambda}\int_{0}^{\theta}\phi^{\lambda-1}\phi^2 (F(y_1, \phi) - F(y_2, \phi))\d\phi \right| \\
       &\hspace{15em}\leq C\|G\|_{C^1}\theta^2\|y_1 - y_2\|_{C^0} + C\|F\|_{C^1}\|y_1 - y_2\|_{C^0}\theta^{-\lambda}\int_{0}^{\theta}\phi^{\lambda+1}\d\phi \\
       &\hspace{15em}\leq C_{\lambda}(\|F\|_{C^1} + \|G\|_{C^1})d^2\|y_1 - y_2\|_{C^0}.
   \end{align*}
   Taking \(a\) sufficiently small, we conclude that the associated mapping is a contraction on \(B_{R}(0) \subset C^0\) and thus has a unique fixed point \(y \in C^{0}([0, d])\). We now use \eqref{eq:sing_ode} to show that \(y \in C^{\infty}([0, d])\). Suppose \(y \in C^k([0, d])\) for some \(k \geq 0\). Then, 
   \[
   \frac{d^{k+1}}{d\theta^{k+1}}\left[y(\theta) - \theta^2 G(y, \theta)\right] = (1-\theta^2\p_1^{k+1} G(y, \theta))y^{(k+1)}(\theta) - G_k(y, \theta)
   \]
   where \(G_k \in C^{0}([0, d])\) denotes all the lower order terms and \(\p_1\) denotes a derivative in the first component. Since \(1-\theta^2\p_1^{k+1} G(y, \theta) > 0\) for \(\theta\) sufficiently small we can divide to isolate \(y^{(k+1)}\). From \eqref{eq:sing_ode}, it now suffices to show that if \(g \in C^k([0,d])\) then
   \[
   \theta^{-\lambda}\int_{0}^{\theta}\phi^{\lambda+1}g(\phi)\d\phi \in C^{k+1}([0, d]).
   \]
   Making the change of variables \(\theta\eta=\phi\) we have
   \begin{equation}\label{eq:fp_smoothing}
   \theta^{-\lambda}\int_{0}^{\theta}\phi^{\lambda+1}g(\phi)\d\phi = \theta^2\int_{0}^{1}\eta^{\lambda+1}g(\theta\eta)\d\eta.
   \end{equation}
   Taking \(k+1\) derivatives of \eqref{eq:fp_smoothing}, the highest order term occurs when all derivatives land on \(g\). Considering this term and integrating by parts, we obtain
   \begin{align*}
   \theta^2\int_{0}^{1}\eta^{\lambda+k+2}g^{(k+1)}(\theta\eta)\d\eta &= \theta^2\int_{0}^{1}\eta^{\lambda+k+2}\theta^{-1}\frac{d}{d\eta}g^{(k)} (\theta\eta)\d\eta \\
   &= \theta g^{(k)}(\theta\eta)-\theta(\lambda+k+2)\int_{0}^{1}\eta^{\lambda+k+1}g^{(k)}(\theta\eta)\d\eta,
   \end{align*}
   which is continuous since \(g \in C^k\). Moreover, since \(g \in C^k\) all lower order terms are also continuous and thus we conclude that \(y \in C^{\infty}([0, d])\) as desired.
\end{proof}

\begin{lemma}\label{lemma:holder}
    Suppose \(f \in C^{1}((0, 1]) \cap C([0,1])\), \(f(0) = 0\) and
    \[
    \lim_{\theta \to 0}\frac{\theta f'(\theta)}{f(\theta)} = \alpha
    \]
    for some \(\alpha > 0\). Then, \(f \in C^{\beta}([0,1])\) for all \(\beta < \max\{1, \alpha\}\).
\end{lemma}
\begin{proof}
    We note, that it is not necessarily true that \(f \in C^{\alpha}([0,1])\). For instance, one can take \(f(\theta) = \theta^\alpha\log(\theta)\). For any \(\epsilon > 0\), choose \(\delta > 0\) such that \(|\theta (\log(f(\theta)))' - \alpha| < \epsilon\) for all \(0 < \theta < \delta\). Then, for \(0 < \theta < \delta\)
    \[
     | (\log(f(\theta)))' - \alpha \theta^{-1}| < \epsilon \theta^{-1}
    \]
    and integrating on \([\theta, \delta]\) we have
    \begin{align}\label{eq:holder_vanish}
       \left|\log(f(\theta) - \log(f(\delta)) + \alpha \log(\theta) - \alpha \log(\delta) \right| \leq \int_{\theta}^{\delta} | (\log(f(\theta)))' - \alpha \theta^{-1}| \leq -\epsilon\log(\theta/\delta).
    \end{align}
    Exponentiating the above gives \(|f(\theta)\theta^{-\alpha}||f(\delta)\delta^{-\alpha}| \leq \theta^{-\epsilon}\delta^{\epsilon}\) and therefore
    \[
    \left|\frac{f(\theta)}{\theta^\beta}\right| \leq \theta^{\alpha-\beta - \epsilon}\frac{\delta^{\alpha+\epsilon}}{|f(\delta)|}.
    \]
    Then, using that \(|\theta f'(\theta)/f(\theta) - \alpha| < \epsilon\) and the fundamental theorem of calculus,
    \begin{align*}
    |f(x) - f(y)| = \left|\int_{x}^{y}f'(z)\d z\right| \leq (\alpha + \epsilon)\left|\int_{x}^{y} \frac{f(z)}{z}\d z\right| 
    \end{align*}
    Finally, using \eqref{eq:holder_vanish} we conclude
    \[
    |f(x) - f(y)| \leq (\alpha + \epsilon)\int_{x}^{y}z^{\alpha - \epsilon-1}\frac{\delta^{\alpha + \epsilon}}{|f(\delta)|}\d z = \frac{(\alpha + \epsilon)\delta^{\alpha + \epsilon}}{|f(\delta)|}|y^{\alpha - \epsilon} - x^{\alpha - \epsilon}|
    \]
    from which it follows that \(f \in C^{\beta}([0, 1])\) for \(\beta < \min\{1, \alpha\}\). 
\end{proof}

\section*{Acknowledgments}
I would like to thank my advisor Professor Tarek M. Elgindi for suggesting the problem and for his guidance and insights throughout the completion of this work. I would also like to thank Professor Federico Pasqualotto for his comments and helpful discussions.
\printbibliography
\end{document}